\documentclass[a4paper]{amsart}

\usepackage[top=1.5in,right=1in,left=1in,bottom=1in,a4paper]{geometry}
\usepackage{amsmath}
\usepackage{amsthm}
\usepackage{amssymb}
\usepackage{amsfonts}
\usepackage{enumerate}
\usepackage[all]{xy}
\usepackage[titletoc,toc,title,page]{appendix} 
\usepackage{tikz}
\usepackage{caption}

\usetikzlibrary{arrows}

\newtheoremstyle{theoremstyle}
  {10pt}      
  {5pt}       
  {\itshape}  
  {}          
  {\bfseries} 
  {:}         
  {.5em}      
  {}          

\newtheoremstyle{examplestyle}
  {10pt}      
  {5pt}       
  {}          
  {}          
  {\bfseries} 
  {:}         
  {.5em}      
  {}          

\newtheoremstyle{sub-style}
  {10pt}      
  {5pt}       
  {}  
  {}          
  {\bfseries} 
  {.}         
  {.5em}      
  {}          

\theoremstyle{theoremstyle}
\newtheorem{theorem}{Theorem}[section]
\newtheorem*{theorem*}{Theorem}
\newtheorem{lemma}[theorem]{Lemma}
\newtheorem{proposition}[theorem]{Proposition}
\newtheorem*{proposition*}{Proposition}
\newtheorem{corollary}[theorem]{Corollary}
\newtheorem*{corollary*}{Corollary}

\theoremstyle{examplestyle}
\newtheorem{example}[theorem]{Example}
\newtheorem*{example*}{Example}
\newtheorem{definition}[theorem]{Definition}
\newtheorem*{definition*}{Definition}
\newtheorem{remark}[theorem]{Remark}
\newtheorem*{remarks*}{Remarks}

\newtheorem*{remark*}{Remark}

\theoremstyle{sub-style}
\newtheorem{sub}[theorem]{}

\newcommand{\comment}[1]{}

\newcommand{\im}{\operatorname{im}}
\newcommand{\sh}[1]{\mathcal{#1}}
\newcommand{\csh}[1]{\mathcal{#1}}
\newcommand{\rk}{\operatorname{rk}}
\newcommand{\Hom}{\operatorname{Hom}}
\newcommand{\RHom}{R\mathcal{H}om}
\newcommand{\rhom}{\operatorname{RHom}}
\newcommand{\Ext}{\operatorname{Ext}}
\newcommand{\End}{\operatorname{End}}

\newcommand{\Ch}{\operatorname{CH}}

\newcommand{\spec}{\operatorname{spec}}
\newcommand{\characteristic}{\operatorname{char}}
\newcommand{\diag}{\operatorname{diag}}
\newcommand{\knum}{K_0^{\operatorname{num}}}
\newcommand{\chnum}{\operatorname{CH}_{\operatorname{num}}}

\newcommand{\K}{\mathbb{K}}
\newcommand{\Z}{\mathbb{Z}}

\newcommand{\Q}{\mathbb{Q}}

\DeclareMathOperator\ch{ch}

\begin{document}

\title[Exceptional Sequences]{Combinatorial aspects of exceptional sequences on (rational) surfaces}


\author{Markus Perling}
\address{Universit\"at Bielefeld, Fakult\"at f\"ur Mathematik\\Postfach 100 131\\33501 Bielefeld\\Germany}
\email{perling@math.uni-bielefeld.de}
\thanks{With support of the DFG priority program 1388 ``Representation theory'' and the Max Planck
Institute for Mathematics.}

\begin{abstract}
We investigate combinatorial aspects of exceptional sequences in the derived category
of coherent sheaves on certain smooth and complete algebraic surfaces.
We show that to any such sequence
there is canonically associated a complete toric surface whose torus fixpoints are
either smooth or cyclic T-singularities (in the sense of Wahl) of type
$\frac{1}{r^2}(1, kr - 1)$. We also show that any exceptional sequence can be transformed by
mutation into an exceptional sequence which consists only of objects of rank one.
\end{abstract}

\maketitle

\tableofcontents

\section{Introduction}

In this article we want to work out certain combinatorial aspects associated with
exceptional sequences in the derived category $D^b(X)$ of coherent sheaves on rational
surfaces. In earlier work \cite{HillePerling11} it was found, somewhat surprisingly, that
to any exceptional sequence of invertible sheaves on a rational surface there is associated
in a canonical way the combinatorial data of a smooth complete toric surface. This finding
suggests that in many interesting cases there could be a link between semi-orthogonal
decompositions of derived categories and toric geometry. However, so far this is not very
well understood, even for the case of line bundles. An important development in
this direction is work by Hacking and Prokhorov \cite{HackingProkhorov10} and
Hacking \cite{Hacking13}.
In \cite{HackingProkhorov10}, singular surfaces with ample anticanonical divisor and Picard number one
which admit $\Q$-Gorenstein smoothings are classified. These surfaces necessarily have
$T$-singularities in the
sense of Wahl \cite{Wahl81}. Among such surfaces, there is one family of weighted
projective planes $\mathbb{P}(e^2, f^2, g^2)$ such that $e, f, g$ satisfy the Markov equation
\begin{equation*}
e^2 + f^2 + g^2 = 3efg.
\end{equation*}
This classification resembles Rudakov's interpretation  \cite{Rudakov89a} of the classification
of exceptional bundles on $\mathbb{P}^2$ by Drezet and Le Potier \cite{DrezetLePotier}. Rudakov showed
that any exceptional sequence $\sh{E}, \sh{F}, \sh{G}$ on $\mathbb{P}^2$ is essentially uniquely
determined by the ranks ($e, f, g$, say) and the possible ranks correspond to solutions of
the Markov equation $e^2 + f^2 + g^2 = 3efg$. In \cite{Hacking13}, Hacking shows that indeed
there exists a natural bijective correspondence between degenerations of $\mathbb{P}^2$ and
exceptional bundles on $\mathbb{P}^2$. More generally, for a rather large class of surfaces
(which includes rational surfaces and certain surfaces of general type) Hacking constructs
a correspondence between exceptional vector bundles and $\Q$-Gorenstein degenerations.

In this article, we want to follow some ideas of both \cite{HillePerling11} and \cite{Hacking13}
and show that the relation between exceptional sequences and toric surfaces with $T$-singularities
is a quite general phenomenon. Our main result is the following:

\begin{theorem*}[\ref{maintheorem2}]
Let $X$ be a numerically rational surface and let $\mathbf{E} = \sh{E}_1,\dots, \sh{E}_n$
be a numerically exceptional sequence whose length equals $\rk \knum(X)$ such that
$\rk \sh{E}_i = e_i \neq 0$ for every $i$. Then
to this sequence there is associated in a canonical way a complete toric surface $Y(\mathbf{E})$
with $n$ torus fixpoints which are either smooth (if $e_i^2 = 1$) or $T$-singularities of type
$\frac{1}{e_i^2}(1, k_i e_i - 1)$, where $\gcd\{k_i, e_i\} = 1$.
Moreover, this correspondence induces a natural isomorphism of Chow rings
$\Ch^*(Y(\mathbf{E}))_\Q \rightarrow \chnum^*(X)_\Q$ which maps $K_{Y(\mathbf{E})}$ to
$K_X$.
\end{theorem*}

Here, $\knum(X)$ denotes the numerical Grothendieck group of $X$ which is
defined as the quotient of the Grothendieck group $K_0(X)$ by the null space of the Euler form,
and $\chnum^i(X)$ denotes the $i$-th Chow group modulo numerical equivalence. In particular,
$\chnum^1(X)$ is the {\em N\'eron-Severi lattice} of $X$.
The ring structure of $\chnum^*(X)$ is induced from $\Ch^*(X)$.

The term {\em numerically exceptional} refers to a weaker version of exceptionality and
semi-orthogonality which
only requires the vanishing of Euler characteristics rather than $\Hom$-vanishing (see Definition
\ref{exceptionaldef}). By $X$ being a numerically rational surface, we mean that $X$ is an
algebraic surface whose effective numerical properties are that of a rational surface.
More precisely, we define:

\begin{definition*}
Let $X$ be a smooth complete surface defined over a ground field. We call $X$
{\em numerically rational} if the following hold:
\begin{enumerate}[1)]
\item $\chi(\sh{O}_X) = 1$.
\item $K_X^2 = 12 - \rk \knum(X)$.
\end{enumerate}
\end{definition*}
By construction, both $\knum$ and $\chnum^1(X)$ are torsion free abelian groups, and
the Chern isomorphism $K_0(X)_\Q \rightarrow \Ch^*(X)_\Q$ descends to a ring isomorphism
$\knum(X)_\Q \rightarrow \chnum^*(X)_\Q$. So $\knum(X)$ and $\chnum^1(X)$ both are finitely generated
with $\rk \knum(X) = \rk \chnum^1(X) + 2$.
The class of numerically rational surfaces in particular includes surfaces of general type
with $p_g = q = 0$.

We will always assume that our numerically rational surface admits a numerically exceptional
sequence of {\it maximal length}, i.e. a semi-orthonormal basis of $\knum(X)$. Note that
condition 2) can be dropped in many cases (see Remark \ref{relaxconditions}),
though it is an open question whether it can be removed entirely.

Another interesting result
which will be part of our analysis leading to Theorem \ref{maintheorem2} is the following.

\begin{theorem*}[\ref{rankone}]
Let $X$ be a numerically rational surface. Then any numerically exceptional sequence of maximal
length on $X$ can be transformed by mutation into a numerically exceptional sequence consisting
only of objects of rank one.
\end{theorem*}

Both theorems a fortiori apply as well to proper exceptional sequences. The main reason why
we formulated them for numerically exceptional sequences is that indeed they are purely
a result of the surprisingly rich Riemann-Roch arithmetic. The correspondence between
exceptional sequences and toric surfaces therefore does not depend on any geometric construction
or any refined geometric properties of $X$. This is of particular interest in light of recent work
(e.g. \cite{BBKS12}) where exceptional sequences have been constructed on certain complex surfaces of
general type with $p_g = q = 0$. These sequences are almost full --- their complements in the
derived category are among the first examples of so-called phantom categories. By
\cite[Theorem 3.1]{Vial15} our results are applicable to all surfaces of general type with $p_g = q = 0$
including those of \cite{BBKS12}.
Indeed it is an open question whether the existence of a full exceptional sequence on a variety
implies that this variety is rational. So far, no example of a non-rational variety which admits such
a sequence is known.
We believe that our results in conjunction with Hacking's provide tentative evidence that indeed
the existence of a full exceptional sequence implies rationality.

One important aspect of our results is the connection between mutation of exceptional sequences
and what we propose to call minimal model program for a class of toric surfaces which includes, but
is strictly bigger than, the class of smooth toric surfaces. We will spend
the remainder of this introduction to explain this connection and to relate it to some of the
technical results in the paper. Assume that $\sh{E}_1, \dots, \sh{E}_n$ is an exceptional sequence,
where $n = \rk \knum(X)$ and for simplicity, $0 \neq e_i := \rk \sh{E}_i$ for every $i$. Generalizing
an idea of \cite{HillePerling11}, we denote $A_i := c_1(\sh{E}_{i + 1})/e_{i + 1} - c_1(\sh{E}_i)/e_i
\in \chnum^1(X)_\Q$ for $1 \leq i < n$, $A_n :=  c_1(\sh{E}_1)/e_1 - c_1(\sh{E}_n)/e_n - K_X$,
and $A$ the $\Z$-linear span of $A_1, \dots, A_n$ in $\chnum^1(X)_\Q$. Then $A$ is a free $\Z$-module
of rank $n - 2$ and we have a short exact sequence:
$$
0 \longrightarrow \Z^2 \overset{L}\longrightarrow \Z^n \overset{c}{\longrightarrow} A
\longrightarrow 0,
$$
where $c$ is the map which sends the $i$-th standard basis vector of $\Z^n$ to $A_i$.
We can now choose to associate the rows $l_1, \dots, l_n$ of $L$ with the columns of
$c$ (the $A_i$, disregarding a choice of basis for $A$), i.e. we consider $l_i \in \Z^2$ as
associated to $A_i \in A$. This association is often called {\em Gale duality}. It is
elementary to see that Gale duality is complementary with respect to linear dependency,
e.g. if some of the $l_i$ form a basis of $\Z^2$ then the complementary $A_j$ form a basis
of $A$; if some of the $l_i$ form a minimal linearly dependent set, then the complementary
$A_j$ generate a hyperplane in $A$, and so on. A substantial part of this paper is devoted to
determine from the $A_i$ and their respective intersection products that
the $l_i$ form a circularly ordered set of primitive lattice vectors which generate the fan
of a complete toric surface (Proposition \ref{winding}) which has at most $T$-singularities
(Theorem \ref{maintheorem2}). The corresponding fans are characterized by their collection of primitive
vectors $l_1, \dots, l_n$ and nonzero integers $e_1, \dots, e_n$ such that:
\begin{enumerate}[1)]
\item the determinants of two adjacent lattice vectors are squares: $\det(l_{i - 1}, l_i) = e_i^2$,
\item the differences $l_i - l_{i - i}$ have lattice length $|e_i|$. We think of these difference
as segments of the ``circumference'' of the fan which connects the $l_i$. In particular, as in the
depictions below, one often does not care about the orientation of these segments
(see Section \ref{mutationsection}, however).
\end{enumerate}
Now recall that a complete toric surface $Y(\mathbf{E})$ has $n$ torus invariant prime divisors
$D_1, \dots, D_n$ which form a cycle, i.e. $D_i \simeq \mathbb{P}^1$ for every $i$ and
$D_i \cdot D_j = 0$ whenever $|i - j| > 1$ and $D_i \cdot D_{i + 1} = 1/e_{i + 1}^2 \in \Q
= \Ch^2(Y)_\Q$ for every $i$, where
the product on $\Ch^*(Y(\mathbf{E}))_\Q$ is the orbifold intersection product. For our constructions, we have
the correspondence
$A_i \leftrightarrow l_i$ via Gale duality, and toric geometry relates $l_i \leftrightarrow D_i$.
The ring isomorphism $\Ch^*(Y(\mathbf{E}))_\Q \rightarrow \chnum^*(X)_\Q$ of Theorem \ref{maintheorem2}
then is determined by mapping $D_i$ to $A_i$ (see Remark \ref{divisorcorrespondence}).

\captionsetup{margin=.3cm}

To give an example, consider some Hirzebruch surface $\mathbb{F}_a$ for some $a \geq 0$ and denote
$P, Q$ the primitive integral generators of its nef cone, where $P^2 = 0$ is the class of the fiber
and $Q^2 = a$ is the class of the relative ample divisor of the fibration
$\mathbb{F}_a \rightarrow \mathbb{P}^1$.
Then for any $s \in \Z$,
$$
\sh{O}, \sh{O}(P), \sh{O}((s + 1)P + Q), \sh{O}((s + 2)P + Q)
$$
is a full exceptional sequence on $\mathbb{F}_a$ (see e.g. \cite[Proposition 5.2]{HillePerling11}).
According to \cite[Theorem 3.5]{HillePerling11} the toric surface associated to this sequence is
the Hirzebruch surface $\mathbb{F}_{b}$ where $b = |a + 2s|$. The associated fan is specified by
lattice vectors $l_1, l_2, l_3, l_4$, where, if we choose, say, $l_1, l_2$ as a
basis of $\Z^2$ then we have $l_3 = -l_1 - (a + 2s) l_2$ and $l_4 = -l_2$. Figure \ref{f3pic1} shows
the fan associated to this toric surface for the case $a = 3$ and $s = 1$.
\begin{figure}[ht]
\centering
\begin{minipage}{.49\linewidth}
\centering
\begin{tikzpicture}
\draw[style=help lines,dashed,very thin] (-1.3,-1.3) grid[step=.7cm] (1.3, 4.1);

\draw[very thick, dashed, red] (.7, 3.5) -- (0, .7);
\draw[very thick, dashed, red] (0, .7) -- (-.7, 0);
\draw[very thick, dashed, red] (-.7, 0) -- (0, -.7);
\draw[very thick, dashed, red] (0, -.7) -- (.7, 3.5);

\draw[->, thick] (0, 0) -- (.7, 3.5);
\draw[->, thick] (0, 0) -- (0, .7);
\draw[->, thick] (0, 0) -- (-.7, 0);
\draw[->, thick] (0, 0) -- (0, -.7);

\draw (-.5, .5) node { $1$ };
\draw (.1, 1.8) node { $1$};
\draw (-.5, -.5) node { $1$ };
\draw (.35, .3) node { $1$ };

\draw (1, 3.7) node { $l_3$ };
\draw (.25, -1) node { $l_2$ };
\draw (-.9, .2) node { $l_1$ };
\draw (-.2, 1) node { $l_4$ };
\end{tikzpicture}
\caption{The fan associated to
$\sh{O}, \sh{O}(P)$, $\sh{O}(2P + Q), \sh{O}(3P + Q)$
on $\mathbb{F}_3$.
}\label{f3pic1}
\end{minipage}
\begin{minipage}{.49\linewidth}
\centering
\begin{tikzpicture}
\draw[style=help lines,dashed,very thin] (-1.3,-1.3) grid[step=.7cm] (5.5, 4.1);

\draw[very thick, dashed, red] (.7, 3.5) -- (0, .7);
\draw[very thick, dashed, red] (0, .7) -- (-.7, 0);
\draw[very thick, dashed, red] (-.7, 0) -- (4.9, -.7);
\draw[very thick, dashed, red] (4.9, -.7) -- (.7, 3.5);

\draw[->, thick] (0, 0) -- (.7, 3.5);
\draw[->, thick] (0, 0) -- (0, .7);
\draw[->, thick] (0, 0) -- (-.7, 0);
\draw[->, thick] (0, 0) -- (4.9, -.7);

\draw (-.5, .5) node { $1$ };
\draw (.1, 1.8) node { $1$};
\draw (.35, -.4) node { $1$ };
\draw (1.7, 1) node { $36$ };

\draw (1, 3.7) node { $l_3$ };
\draw (5.15, -1) node { $l_2'$ };
\draw (-.9, .2) node { $l_1$ };
\draw (-.2, 1) node { $l_4$ };
\end{tikzpicture}
\caption{The fan associated to
$\sh{O}, \sh{O}(2P + Q), \sh{R}, \sh{O}(3P + Q)$
on $\mathbb{F}_3$.}\label{f3pic2}
\end{minipage}
\end{figure}
The numbers, which in
this case are always $1$, represent the lattice volumes $\det(l_i, l_{i + 1})$. The circumference
is indicated by the dashed line. We
can see that the lattice lengths of its segments are always $1$.
Applying a right mutation to the pair $\sh{O}(P), \sh{O}((s + 1)P + Q)$ yields another exceptional
sequence
$$
\sh{O}, \sh{O}((s + 1)P + Q), \sh{R}, \sh{O}((s + 2)P + Q)
$$
with $\sh{R} = R_{\sh{O}((s + 1)P + Q)} \sh{O}(P)$ and $r := \rk \sh{R} = a + 1 + 2s$.
The effect of the mutation to the combinatorial picture is the {\em mutation} of $l_2$ into
$l_2' = l_2 -(a + 2 + 2s) l_1$ (see Proposition \ref{mutationlocal} and Lemma \ref{peri2})
and the
corresponding toric surface has a cyclic singularity of type $\frac{1}{r^2}(1, r(r - 1) - 1)$.
Figure \ref{f3pic2} shows the case $a = 3$, $s = 1$ with $r = 6$, so that in the fan we have
created a cone of lattice
volume $36$ and the corresponding circumference segment of length $6$.

In \cite{Orlov93}, Orlov described how exceptional sequences (and more general semi-orthognal
decompositions) can be completed along blow-ups. Consider for example the exceptional sequence
$\mathcal{T}, \mathcal{O}(2), \mathcal{O}(4)$ on $\mathbb{P}^2$, where $\sh{T}$ denotes the tangent
bundle, and a blow-up $b : X \longrightarrow
\mathbb{P}^2$ in one point; we denote $E$ the exceptional curve with $E^2 = -1$. Figure \ref{p2pic1}
shows the fan corresponding to this sequence which describes the weighted projective plane
$\mathbb{P}(1, 1, 4)$.
\begin{figure}[ht]
\centering
\begin{minipage}{.49\linewidth}
\centering
\begin{tikzpicture}
\draw[style=help lines,dashed,very thin] (-1.3,-1.3) grid[step=.7cm] (1.3, 3.4);

\draw[->, thick] (0, 0) -- (.7, 0);
\draw[->, thick] (0, 0) -- (0, -.7);
\draw[->, thick] (0, 0) -- (-0.7, 2.8);

\draw (.3, 1.6) node { $4$ };
\draw (.55, -.5) node { $1$ };
\draw (-.35, .3) node { $1$ };

\draw (1, .2) node { $l_3$ };
\draw (-.25, -1) node { $l_2$ };
\draw (-1, 3) node { $l_1$ };

\draw[very thick, dashed, red] (.7, 0) -- (0, -.7);
\draw[very thick, dashed, red] (0, -.7) -- (-.7, 2.8);
\draw[very thick, dashed, red] (-.7, 2.8) -- (.7, 0);
\end{tikzpicture}
\caption{The fan associated to $\mathcal{T}, \mathcal{O}(2), \mathcal{O}(4)$
on $\mathbb{P}^2$.}\label{p2pic1}
\end{minipage}
\begin{minipage}{.49\linewidth}
\centering
\begin{tikzpicture}
\draw[style=help lines,dashed,very thin] (-1.3,-1.3) grid[step=.7cm] (1.3, 3.4);

\draw[very thick, dashed, red] (.7, 0) -- (0, -.7);
\draw[very thick, dashed, red] (0, -.7) -- (-.7, 2.8);
\draw[very thick, dashed, red] (-.7, 2.8) -- (.7, 0);

\draw[->, thick, double] (0, 0) -- (.7, 0);
\draw[->, thick] (0, 0) -- (0, -.7);
\draw[->, thick] (0, 0) -- (-0.7, 2.8);

\draw (.3, 1.6) node { $4$ };
\draw (.55, -.5) node { $1$ };
\draw (-.35, .3) node { $1$ };

\draw (1, .2) node { $l_3$ };
\draw (-.25, -1) node { $l_2$ };
\draw (-1, 3) node { $l_1$ };
\end{tikzpicture}
\caption{The fan associated to $\sh{O}_E(E), b^*\mathcal{T},$ $b^*\mathcal{O}(2),
b^*\mathcal{O}(4)$ on $X$.}\label{p2pic2}
\end{minipage}
\end{figure}
The pull-back of this sequence $b^*\mathcal{T}, b^*\mathcal{O}(2), b^*\mathcal{O}(4)$ is an
exceptional sequence on $X$ which can be completed to the full sequence
$\sh{O}_E(E), b^*\mathcal{T}, b^*\mathcal{O}(2), b^*\mathcal{O}(4)$. The object
$\sh{O}_E(E)$ has rank zero which leads in the combinatorial picture to a {\em doubling} of
a primitive vector ($l_3$ in this case, see Figure \ref{p2pic2}) or, if we like, to the creation
of a new cone of volume zero (see Proposition \ref{multiplicity} and Theorem \ref{deltatheorem}).
Applying right mutation to the pair $\sh{O}_E(E), b^*\sh{T}$ yields
a sequence $b^*\mathcal{T}, \sh{R}', b^*\mathcal{O}(2), b^*\mathcal{O}(4)$ where $\rk \sh{R}' = -4$.
Figure \ref{p2pic3} shows the effect of this mutation: one of the multiple lattice vectors ``jumps''
into the neighouring cone, thereby subdividing it into two cones of lattice volumes $4$ and $16$,
respectively. The corresponding toric surface therefore has two cyclic $T$-singularities of orders
$4$ and $16$.
\begin{figure}[ht]
\centering
\begin{minipage}{.49\linewidth}
\centering
\begin{tikzpicture}
\draw[style=help lines,dashed,very thin] (-1.3,-1.3) grid[step=.7cm] (2.7, 3.4);

\draw[very thick, dashed, red] (.7, 0) -- (0, -.7);
\draw[very thick, dashed, red] (0, -.7) -- (-.7, 2.8);
\draw[very thick, dashed, red] (-.7, 2.8) -- (2.1, 2.8);
\draw[very thick, dashed, red] (2.1, 2.8) -- (.7, 0);

\draw[->, thick] (0, 0) -- (.7, 0);
\draw[->, thick] (0, 0) -- (0, -.7);
\draw[->, thick] (0, 0) -- (-.7, 2.8);
\draw[->, thick] (0, 0) -- (2.1, 2.8);

\draw (1.6, 1.15) node { $4$ };
\draw (.55, -.5) node { $1$ };
\draw (-.35, .3) node { $1$ };
\draw (.35, 1.75) node { $16$ };

\draw (1, -.3) node { $l_3$ };
\draw (-.25, -1) node { $l_2$ };
\draw (-.9, 3) node { $l_1$ };
\draw (2.3, 3) node { $l_4$ };
\end{tikzpicture}
\caption{The fan for $b^*\mathcal{T}, \sh{R}', b^*\mathcal{O}(2),
b^* \mathcal{O}(4)$.}\label{p2pic3}
\end{minipage}
\begin{minipage}{.49\linewidth}
\centering
\begin{tikzpicture}
\draw[style=help lines,dashed,very thin] (-1.3,-1.3) grid[step=.7cm] (1.3, 1.3);

\draw[very thick, dashed, red] (.7, 0) -- (0, -.7);
\draw[very thick, dashed, red] (0, -.7) -- (-.7, .7);
\draw[very thick, dashed, red] (-.7, .7) -- (.7, 0);

\draw[->, thick, double] (0, 0) -- (.7, 0);
\draw[->, thick] (0, 0) -- (0, -.7);
\draw[->, thick] (0, 0) -- (-.7, .7);

\draw (1, -.3) node { $l_3$ };
\draw (-.25, -1) node { $l_2$ };
\draw (-.9, .9) node { $l_1$ };
\end{tikzpicture}
\caption{The fan for
$\sh{O}_E(E), b^*\sh{O}(2), b^*\sh{O}(3),$ $b^*\sh{O}(4)$.}\label{p2pic4}
\end{minipage}
\end{figure}
If instead we apply a right mutation to the pair $b^*\sh{T}, b^*\sh{O}(2)$, we obtain a sequence
isomorphic to $\sh{O}_E(E), b^*\sh{O}(2), b^*\sh{O}(3), b^*\sh{O}(4)$. Figure \ref{p2pic4}
shows that we end up with the fan for $\mathbb{P}^2$, again with $l_3$ doubled.
The reader may have noticed that the cyclic enumerations of the $l_i$ in Figure \ref{p2pic3}
are now shifted by one position as compared to the characterisation given earlier. For a complete
description of the correspondence between an exceptional sequence $\sh{E}_1, \dots, \sh{E}_n$
that may contain objects of rank zero and lattice vectors $l_1, \dots, l_n$ we refer to Sections
\ref{galedualsection} and \ref{rankzeromoving}; in particular compare Example \ref{p2blowupexample2}.

The transition from $b^*\sh{T}, \sh{R}', b^*\sh{O}(2), b^*\sh{O}(4)$ to $\sh{O}_E(E),
b^*\sh{O}(2), b^*\sh{O}(3), b^*\sh{O}(4)$
illustrates for a simple case how the minimal model program works for toric surfaces which
are associated to an exceptional sequence on some numerically rational surface
(see Example \ref{p2blowupexample} for an explicit representation of the associated toric systems).
In general, assume
we have a numerically exceptional sequence of maximal length $\mathbf{E} = \sh{E}_1, \dots, \sh{E}_n$
on such a surface $X$ and the corresponding toric surface $Y = Y(\mathbf{E})$.
\begin{enumerate}
\item For every pair $\sh{E}_i, \sh{E}_{i + 1}$, the ranks $e_i, e_{i + 1}$ and Euler characteristic
$\chi(\sh{E}_i, \sh{E}_{i + 1})$ translate into certain local convexity properties of the fan of
$Y$. These properties and their behaviour under mutation are studied in Section \ref{mutationsection}.
\item By Lemma \ref{concavemutation} there is a criterion when we can use these properties in order
to reduce the ranks of objects in $\mathbf{E}$ by mutation.
\item Taking global properties into account, we will show that we can use this criterion
to transform our sequence by
mutation into a sequence $\mathbf{E}' = \sh{Z}_1, \dots, \sh{Z}_t$, $\sh{F}_1, \dots, \sh{F}_{n - t}$
where either
$t = n - 3$ or $t = n - 4$, $\rk \sh{Z}_i = 0$ for every $i$, and $\rk \sh{F}_j \neq 0$ for every $j$
(Corollaries \ref{algorithm} \& \ref{rankonezero}). As in above examples, this means that the associated
fan $Y(\mathbf{E}')$ is generated by either $3$ or $4$ lattice vectors, one of which appears with
multiplicity $n - t + 1$. Note however that in Section \ref{rankzeromoving} we will see that the
distribution of multiplicities of these lattice
vectors is essentially arbitrary and in Section \ref{mutationsection} we decide to accumulate them
onto one lattice vector purely for convenience.
\item The cases $t = n - 3$ and $t = n - 4$ are easily analyzed. In the first case
(see \ref{p2example}), the triple $\rk \sh{F}_1, \rk \sh{F}_2, \rk \sh{F}_3$ is a solution of the
Markov equation and, as has already been explored by Rudakov \cite{Rudakov89a}, by further mutation
we can transform the $\sh{F}_i$ into objects of rank $\pm 1$. The corresponding toric surface then
is $\mathbb{P}^2$. In the case $t = n - 4$ we will have directly $\rk \sh{F}_j = \pm 1$ for every $j$,
in which case the corresponding fan will be that of a
Hirzebruch surface (see \ref{nequals4}).
\end{enumerate}

In summary, the minimal model program for toric surfaces associated to exceptional sequences
consists of minimizing lattice volumes
via mutation. Occasionally, we may create a cone of volume zero, which then will live on as the
multiplicity of some primitive vector (this effect corresponds to the blow-down of a smooth toric
surface at a $(-1)$-divisor). Analogous to blowing down smooth toric surfaces, the process
ends when we arrive at a Hirzebruch surface or $\mathbb{P}^2$. If we neglect multiple lattice vectors,
we get the following result.

\begin{theorem}[see Corollaries \ref{algorithm} \& \ref{rankonezero}]
Let $Y$ be a toric surface associated to a numerically exceptional sequence of maximal length
on a numerically rational
surface. Then it can be transformed via mutation into a Hirzebruch surface or a projective plane.
\end{theorem}

Note that we have not explicitly formulated the theorem in the body of this article and state it
here in order to give a summary of some of the technical statements of this paper and for the
reader's guidance. Also note that, if we translate the usual braid group action on exceptional
sequences to a braid group action on toric surfaces with T-singularities, then the theorem could
further be strenghtened to the effect that the
minimal model program for smooth toric surfaces embeds into this braid group action.
However, the classification of toric
surfaces with $T$-singularities is of broader interest (see e.g. \cite{KasprzykNillPrince15} for
recent results) and we leave a more detailed treatment to future work.

\smallskip

\paragraph{\bf Overview}
In Section \ref{generalitiessection} we introduce some basic notions. The reader will probably avoid
some confusion by paying attention to standing conventions as stated in paragraphs
\ref{conventions1} and \ref{conventions2}.
Section \ref{triplesection} is devoted to the exploitation of the
Riemann-Roch formula for exceptional objects. Sections \ref{rankzerosection} and \ref{rankzeromoving} deal with some crucial
aspects which arise if exceptional objects of rank zero are involved. In Sections
\ref{toricsystemsection} and \ref{galedualsection} toric systems and their Gale dual are introduced.
The latter will be analyzed locally in Sections \ref{localconstellations} and \ref{mutationsection}.
The global analysis and the main theorems are contained in Sections \ref{globalsection} and
\ref{maintheorem}. Section \ref{singularitiessection} then concludes with some observations related
to $T$-singularities. For easier reference, we collect some facts on toric surfaces in an
appendix.

\smallskip

\paragraph{\bf Acknowledgements}
I want to thank Lutz Hille for discussions at an early stage of this project and the referee
for thorough reading and pointing out a mistake in an earlier version of this article.
I also want to thank Charles Vial for discussion which helped to significantly widen the
applicability of the results.
I am grateful for its hospitality to the Max Planck Institute in Bonn, where
part of this work was done.

\section{Some generalities}\label{generalitiessection}

\begin{sub}[\bf Standing conventions throughout the rest of this paper]\label{conventions1}
We assume that $X$ is a numerically rational surface as defined in the introduction over some
ground field $\K$. We denote $D^b(X)$ the bounded derived category of coherent sheaves on
$X$. We will always write objects of $D^b(X)$ in calligraphic style, $\sh{E}, \sh{F}, \dots, \sh{Z}$.
Then their ranks will be denoted in the corresponding lower case letters $e, f, \dots, z$.
By $n$ we will always denote the rank of $\knum(X)$.
\end{sub}

For any two objects $\sh{E}, \sh{F}$ of $D^b(X)$ the Euler characteristic is defined as:
\begin{equation*}
\chi(\sh{E}, \sh{F}) = \sum_{k \in \Z} (-1)^k \dim \Hom_{D^b(X)}(\sh{E}, \sh{F}[k]) =
\sum_{k \in \Z} (-1)^k \dim \Ext^k_{\sh{O}_X}(\sh{E}, \sh{F}).
\end{equation*}

\begin{definition}\label{exceptionaldef}
\begin{enumerate}[(i)]
\item We call an object $\sh{E}$ of $D^b(X)$ {\em exceptional} if $\End(\sh{E}) \simeq \K$ and
$\Hom_{D^b(X)}(\sh{E}, \sh{E}[k]) = 0$ for all $k \neq 0$. We call $\sh{E}$ {\em numerically
exceptional} if $\chi(\sh{E}, \sh{E}) = 1$.
\item A sequence of objects $\sh{E}_1, \dots, \sh{E}_t$ is called an {\em exceptional
sequence} if all $\sh{E}_i$ are exceptional and $\Hom_{D^b(X)}(\sh{E}_i, \sh{E}_j[k]) = 0$
for all $i > j$ and all $k \in \Z$. Similarly, we call it a {\em numerically exceptional
sequence} if all $\sh{E}_i$ are numerically exceptional and $\chi(\sh{E}_i, \sh{E}_j) = 0$
for all $i > j$.
\item\label{exceptionaldefiii}
Denote $\omega_X = \sh{O}(K_X)$ the canonical sheaf on $X$. Then we can extend any
exceptional sequence $\sh{E}_1, \dots, \sh{E}_t$ to an infinite sequence
$\dots, \sh{E}_i, \sh{E}_{i + 1}, \dots$ such that $\sh{E}_{i + t} = \sh{E}_i \otimes
\omega_X^{-1}$ holds for any $i \in \Z$.
We call such a sequence a {\em cyclic exceptional sequence}. If the sequence is only numerically
exceptional, then we call it {\em cyclic numerically exceptional sequence}. For any $i \in \Z$
we call the subsequence $\sh{E}_{i + 1}, \dots, \sh{E}_{i + t}$ a {\em winding}.
\item An exceptional sequence is called {\em strongly} exceptional if
$\Hom_{D^b(X)}(\sh{E}_i, \sh{E}_j[k]) = 0$ for all $i, j$ and all $k \neq 0$. A cyclic exceptional
sequence is called cyclic {\em strongly} exceptional if every winding is strongly exceptional.
\item Any collection of objects in $D^b(X)$ is called {\em full} if it generates $D^b(X)$.
\end{enumerate}
\end{definition}

As general references for exceptional sequences we refer to \cite{Bondal90} and \cite{Rudakov90}.

\begin{sub}\label{shiftchern}
Note that any sub-interval of length at most $t$ of a cyclic (numerically) exceptional
sequence is a (numerically) exceptional sequence
(see \cite[Proposition 5.1]{HillePerling11}).
By convention, if we are given a fixed exceptional sequence $\sh{E}_1, \dots, \sh{E}_t$,
then we will always implicitly assume that it
is extended cyclically, i.e. we consider $\sh{E}_i$ for any $i \in \Z$, denoting
any element of the original sequence twisted by an appropriate power of $\omega_X$
as in \ref{exceptionaldef} (\ref{exceptionaldefiii})
\end{sub}

\begin{sub}
If $\sh{E}_1, \dots, \sh{E}_t$ is a (numerically) exceptional sequence, then so is
$\sh{E}_1, \dots, \sh{E}_{i - 1}, \sh{E}_i[j], \sh{E}_{i + 1}, \dots, \sh{E}_t$ for any
$i$ and any $j \in \Z$. So, as long as we are not interested in concrete representations
for the $\sh{E}_i$, we have some flexibility in considering exceptional sequences up to shift.
For instance, there usually is no loss of generality to assume $e_i \geq 0$ for all $i$. Note that
for any object $\sh{E}$ in $D^b(X)$ we have $\ch(\sh{E}) = -\ch(\sh{E}[1])$ which implies
$c_1(\sh{E}) = -c_1(\sh{E}[1])$ and $c_2(\sh{E}) + c_2(\sh{E}[1]) = c_1(\sh{E})^2$.
\end{sub}

\begin{sub}
For any pair of objects $\sh{E}, \sh{F}$ there exist the following two distinguished
triangles:
\begin{align*}
L_\sh{E}\sh{F} \longrightarrow \rhom(\sh{E}, \sh{F}) \otimes_\K \sh{E}
\xrightarrow{\text{can}} \sh{F}, \\
\sh{E} \xrightarrow{\text{can}} \rhom(\csh{E}, \csh{F})^* \otimes_\K \sh{F}
\longrightarrow R_\sh{F}\sh{E},
\end{align*}
where {\em can} in both cases denotes the canonical evaluation map. If $\sh{E}, \sh{F}$
form an exceptional pair, then it follows that both
$\sh{F}, R_\sh{F} \sh{E}$ and $L_\sh{E} \sh{F}, \sh{E}$ form exceptional pairs as well.
\end{sub}

\begin{definition}
For an exceptional or numerically exceptional pair $\sh{E}, \sh{F}$, we call the pairs
$\sh{F}, R_\sh{F} \sh{E}$ and $L_\sh{E} \sh{F}, \sh{E}$ its {\em right}- and
{\em left-mutation}, respectively.
\end{definition}

\begin{sub}
More generally, for a (numerical or proper) exceptional sequence
$\mathbf{E} := \sh{E}_1, \dots, \sh{E}_t$, we can consider mutations at the $i$-th position:
\begin{align*}
R_i \mathbf{E} & := \sh{E}_1, \dots, \sh{E}_{i - 1}, \sh{E}_{i + 1}, R_{\sh{E}_{i + 1}}
\sh{E}_i, \sh{E}_{i + 2}, \dots, \sh{E}_t,\\
L_i \mathbf{E} & := \sh{E}_1, \dots, \sh{E}_{i - 1}, \ L_{\sh{E}_i} \sh{E}_{i + 1}, \ 
\sh{E}_i, \ \sh{E}_{i + 2}, \dots, \sh{E}_t.
\end{align*}
Both $R_i \mathbf{E}$ and $L_i \mathbf{E}$ are exceptional sequences again. Moreover,
up to natural equivalence, the $R_i$ and $L_i$ satisfy the following properties:
\begin{enumerate}[(i)]
\item $L_i = R_i^{-1}$;
\item the braid relations
$R_i R_{i + 1} R_i = R_{i + 1} R_i R_{i + 1}$, $L_i L_{i + 1} L_i = L_{i + 1} L_i L_{i + 1}$.
\end{enumerate}
In particular, the operators $L_1, \dots, L_{t - 1}, R_1, \dots, R_{t - 1}$ establish a braid group
action on the exceptional sequences of length $t$.
Note that mutations extend in a canonical way to cyclic exceptional sequences.
\end{sub}

\begin{sub}\label{mutationrank}
The usual invariants for sheaves such as rank and Chern classes extend naturally to $D^b(X)$
(and then factor naturally through $\knum(X)$). In particular, the rank function is additive on
triangles and for an exceptional pair $\sh{E}, \sh{F}$ of ranks $e$ and $f$, respectively, we obtain
\begin{equation*}
\rk L_\sh{E}\sh{F} = \chi(\sh{E}, \sh{F}) e - f \quad \text{ and } \quad
\rk R_\sh{F}\sh{E} = \chi(\sh{E}, \sh{F}) f - e.
\end{equation*}
The first Chern classes can be written down directly:
\begin{equation*}
c_1(L_\sh{E}\sh{F}) = \chi(\sh{E}, \sh{F}) c_1(\sh{E}) - c_1(\sh{F}) \quad \text{ and }
\quad
c_1(R_\sh{F}\sh{E}) = \chi(\sh{E}, \sh{F}) c_1(\sh{F}) - c_1(\sh{E}).
\end{equation*}
For the second Chern classes, one can make use of the fact that for any triangle
$\sh{T}' \rightarrow \sh{T} \rightarrow \sh{T}''$, the Chern character satisfies
$\ch(\sh{T}) = \ch(\sh{T}') + \ch(\sh{T}'')$. With this, we obtain the following formula
for the second Chern classes of mutations:
\begin{equation*}
c_2(L_\sh{E} \sh{F}) = \binom{\chi(\sh{E}, \sh{F})}{2}c_1(\sh{E})^2 - \chi(\sh{E}, \sh{F})
c_1(\sh{E}) c_1(\sh{F}) + c_1(\sh{F})^2 + \chi(\sh{E}, \sh{F}) c_2(\sh{E}) - c_2(\sh{F}),
\end{equation*}
and similarly for $c_2(R_\sh{F} \sh{E})$.
\end{sub}

\begin{sub}[\bf More standing conventions]\label{conventions2}
In the following, mutations will be our main tool for manipulating (numerically) exceptional sequences
and we may keep in mind that any mutation of a proper exceptional sequence is proper exceptional
again. So, if we like to, we can distinguish between (proper) exceptional and numerically exceptional
orbits of the braid group action. In order to avoid cumbersome language or awkward abbreviations,
we will throughout sections \ref{rankzerosection} to \ref{mutationsection} use the term ``exceptional
sequence'' for both, proper and numerical exceptional sequences. The reader who does not care about
numerical exceptional sequences can safely assume that we are only dealing with proper exceptional
sequences. From section \ref{globalsection} on, we will start making the distinction between both
cases. Again, the non-numerically inclined reader can safely assume that all results a fortiori
apply to proper exceptional sequences.
\end{sub}

\section{Exceptional pairs and triples on surfaces}\label{triplesection}

In \cite{HillePerling11}, for an exceptional
sequence of invertible sheaves $\sh{O}(D_1), \dots, \sh{O}(D_n)$, we have considered
the differences of divisor classes $D_{i + 1} - D_i$. In our more general setting,
we use the following generalization.

\begin{definition*}
\begin{enumerate}[(i)]
\item For any objects $\sh{E}, \sh{F}$ in $D^b(X)$, we denote
$c_i(\sh{E}, \sh{F}) = c_i\big(\RHom(\sh{E}, \sh{F})\big)$.
\item For any object $\sh{E}$ of $D^b(X)$ of nonzero rank, we set $s(\sh{E}) := c_1(\sh{E}) / e
\in \chnum^1(X)_\Q$. For any two such objects
$\sh{E}, \sh{F}$ we set $$s(\sh{E}, \sh{F}) := s(\sh{F}) - s(\sh{E}) = \frac{1}{ef}c_1(\sh{E}, \sh{F}) .$$
\end{enumerate}
\end{definition*}

\begin{sub}\label{difflemma}
For any two objects $\sh{E}$, $\sh{F}$ the following formula holds:
\begin{equation*}
c_1(\csh{E}, \csh{F}) = e c_1(\csh{F}) - f c_1(\csh{E}).
\end{equation*}
This is immediately clear for vector bundles, because in this case
$\RHom(\sh{E}, \sh{F}) \simeq \sh{E}^* \otimes \sh{F}$. The extension to the general
case follows from the fact that, because $X$ is smooth, any object in $D^b(X)$
is quasi-isomorphic to a finite complex of vector bundles. In the case $\sh{E}$ and $\sh{F}$
have nonzero rank it follows that $s(\sh{E}, \sh{F}) = s(\RHom(\sh{E}, \sh{F}))$.
\end{sub}

\begin{sub}[Riemann-Roch formula]\label{RR}
For any $\sh{E}, \sh{F}$ in $D^b(X)$, the Riemann-Roch formula is:
\begin{gather*}
\chi(\csh{E}, \csh{F})
= ef - \frac{1}{2}K_X c_1(\csh{E}, \csh{F}) + \frac{1}{2}\big(fc_1(\csh{E})^2
+ e c_1(\csh{F})^2 - 2 c_1(\csh{E})c_1(\csh{F})\big) -  \big(fc_2(\csh{E}) +e c_2(\csh{F})\big).
\end{gather*}
\end{sub}

We now collect some identities which we obtain from simple inspection of the
Riemann-Roch formula.

\begin{sub}\label{eulerchar1}
Let $\sh{E}$ be any object in $D^b(X)$.
\begin{enumerate}[(i)]
\item\label{eulerchar1i} If $e = 0$ then $\sh{E}$ is numerically exceptional iff
$c_1(\sh{E})^2 = -1$.
\item\label{eulerchar1ii} If $e \neq 0$ then $\sh{E}$ is numerically exceptional iff
\begin{equation*}
c_2(\csh{E}) = \frac{1}{2e}(e^2 + (e - 1) c_1(\sh{E})^2 - 1).
\end{equation*}
\end{enumerate}
\end{sub}

\begin{sub}\label{eulerchar2}
Now for objects $\sh{E}, \sh{F}$ with $\chi(\sh{E}, \sh{E}) = \chi(\sh{F}, \sh{F}) = 1$,
we can use \ref{eulerchar1} to simplify the Riemann-Roch formula:
\begin{enumerate}[(i)]
\item\label{eulerchar2i} If $e,f  \neq 0$, then:
\begin{equation*}
\chi(\sh{E}, \csh{F}) = -\frac{1}{2} K_X c_1(\sh{E}, \sh{F}) +
\frac{1}{2ef}(c_1(\sh{E}, \sh{F})^2 + e^2 + f^2).
\end{equation*}
\item\label{eulerchar2ii} If $e = 0$ and $f \neq 0$, then:
\begin{equation*}
\chi(\sh{E}, \csh{F}) = \frac{f}{2} K_X c_1(\sh{E}) -
\big(\frac{f}{2} + c_1(\sh{E}) c_1(\sh{F}) + f c_2(\sh{E})).
\end{equation*}
\item\label{eulerchar2iii} If $e \neq 0$ and $f = 0$, then:
\begin{equation*}
\chi(\sh{E}, \csh{F}) = - \frac{e}{2} K_X c_1(\sh{F}) -
\big(\frac{e}{2} + c_1(\sh{E}) c_1(\sh{F}) + e c_2(\sh{F})\big).
\end{equation*}
\item\label{eulerchar2iv} If $e = f = 0$, then: $$\chi(\sh{E}, \sh{F}) =  \chi(\sh{F}, \sh{E}) = -c_1(\sh{E}) c_1(\sh{F}).$$
\end{enumerate}
\end{sub}

\begin{sub}\label{eulersym}
Anti-symmetrizing of the Euler form yields for any two objects $\sh{E}$, $\sh{F}$:
\begin{enumerate}[(i)]
\item\label{eulersymi}
$\chi(\csh{E}, \csh{F}) - \chi(\csh{F}, \csh{E}) = -K_X c_1(\csh{E}, \csh{F})$.
\end{enumerate}
If $\sh{E}$ and $\sh{F}$ are numerically exceptional, we moreover get by symmetrizing the Euler form:
\begin{enumerate}[(i)]
\setcounter{enumi}{1}
\item\label{eulersymii} If $e, f \neq 0$, then
$\chi(\csh{E}, \csh{F}) + \chi(\csh{F}, \csh{E}) = \frac{1}{ef}(c_1(\csh{E}, \csh{F})^2 + e^2 + f^2)$.
\item\label{eulersymiii} If $e = 0$ and $f \neq 0$, then $\chi(\csh{E}, \csh{F}) +
\chi(\csh{F}, \csh{E}) = -\big(f + 2 c_1(\sh{E}) c_1(\sh{F}) + 2f c_2(\sh{E})\big)$.
\item\label{eulersymiv} If $e \neq 0$ and $f = 0$, then $\chi(\csh{E}, \csh{F}) +
\chi(\csh{F}, \csh{E}) = -\big(e + 2 c_1(\sh{E}) c_1(\sh{F}) + 2e c_2(\sh{F})\big)$.
\item\label{eulersymv} If $e = f = 0$, then $\chi(\csh{E}, \csh{F}) + \chi(\csh{F}, \csh{E}) = -2 c_1(\sh{E})
c_1(\sh{F})$.
\end{enumerate}

In the case that $\sh{E}$ and $\sh{F}$ are numerically
exceptional and  $\chi(\csh{F}, \csh{E}) = 0$, the
formulas \ref{eulersym} yield particularly nice identities for the Euler
characteristic $\chi(\csh{E}, \csh{F})$:
\begin{enumerate}[(i)]
\setcounter{enumi}{5}
\item\label{eulersymvi}\begin{align*}
\chi(\sh{E}, \sh{F}) = -K_X c_1(\sh{E}, \sh{F}) = 
\begin{cases}
\frac{1}{ef}(c_1(\csh{E}, \csh{F})^2 + e^2 + f^2) & \text{ if } e, f \neq 0, \\
-\big(f + 2 c_1(\sh{E}) c_1(\sh{F}) + 2f c_2(\sh{E})\big)  & \text{ if } e=  0, f \neq 0, \\
-\big(e + 2 c_1(\sh{E}) c_1(\sh{F}) + 2e c_2(\sh{F})\big)  & \text{ if } e \neq 0, f = 0, \\
-2 c_1(\sh{E}) c_1(\sh{F}) = 0  & \text{ if } e = f = 0.
\end{cases}
\end{align*}
\end{enumerate}
\end{sub}

\begin{lemma}\label{paircommute}
The pair $\sh{E}, \sh{F}$ is numerically exceptional iff $\sh{F}, \sh{E} \otimes \omega_X^{-1}$ is
and properly exceptional iff
$\sh{F}, \sh{E} \otimes \omega_X^{-1}$ is. Moreover, the following equality holds:
\begin{equation*}
\chi(\sh{E}, \sh{F}) + \chi(\sh{F}, \sh{E} \otimes \omega^{-1}) = ef K_X^2
\end{equation*}
\end{lemma}

\begin{proof}
The first two assertions follow from Serre duality. For the last assertion, we use
\ref{eulersym} (\ref{eulersymi}) and $c_1(\sh{F}, \sh{E} \otimes \omega_X^{-1}) =
-c_1(\sh{E}, \sh{F}) - ef K_X$.
\end{proof}

\begin{sub}\label{additivity}
From \ref{difflemma}, for any three objects $\sh{E}, \sh{F}, \sh{G}$ we get:
\begin{equation*}
fc_1(\csh{E},\csh{G}) = g c_1(\csh{E},\csh{F}) + e c_1(\csh{F},\csh{G}).
\end{equation*}
If moreover these objects form a numerically exceptional triple, we can multiply both sides
of this equation with $-K_X$ and then with \ref{eulersym} (\ref{eulersymi}) the equality
extends to the Euler characteristic:
\begin{equation*}
f \chi(\csh{E}, \csh{G}) = g \chi(\csh{E}, \csh{F}) +e \chi(\csh{F}, \csh{G}).
\end{equation*}
\end{sub}

\begin{proposition}\label{intersectionprop}
Let $\csh{E}, \csh{F}, \csh{G}$ in $D^b(X)$ with $e, f, g \neq 0$ such that $\sh{E}, \sh{F}$
and $\sh{F}, \sh{G}$ form numerically exceptional pairs. Then $\sh{E}, \sh{F}, \sh{G}$ forms
a numerically exceptional triple
(i.e. $\chi(\csh{G}, \csh{E}) = 0$) if and only if $c_1(\csh{E}, \csh{F})
\cdot c_1(\csh{F}, \csh{G}) = eg$.
\end{proposition}

\begin{proof}
By equation \ref{eulerchar2} (\ref{eulerchar2i}) we have
\begin{equation*}
\chi(\csh{G}, \csh{E}) = \frac{-K_X}{2}c_1(\csh{G}, \csh{E}) + \frac{1}{2eg}(c_1(\csh{G}, \csh{E})^2
+ e^2 + g^2).
\end{equation*}
Using \ref{difflemma} and \ref{eulersym} we
get:
\begin{align*}
& \chi(\csh{G}, \csh{E}) = \\
& \frac{-K_X}{2f}\big(e c_1(\csh{G}, \csh{F}) + g c_1(\csh{F}, \csh{E}) \big)
+ \frac{1}{2 e f^2 g}\big(e^2 c_1(\csh{G}, \csh{F})^2 + g^2 c_1(\csh{F}, \csh{E})^2
+ 2 eg c_1(\csh{G}, \csh{F}) \cdot c_1(\csh{F}, \csh{E})\big) + \frac{e^2 + g^2}{2eg} =\\
& \frac{e}{f}(\chi(\csh{G}, \csh{F}) - \frac{f^2 + g^2}{2fg}) + \frac{g}{f}(\chi(\csh{F}, \csh{E})
- \frac{e^2 + f^2}{2ef}) + \frac{1}{f^2}c_1(\csh{G}, \csh{F}) \cdot c_1(\csh{F}, \csh{E}) +
\frac{e^2 + g^2}{2eg} =\\
& \frac{1}{f^2}(-eg + c_1(\csh{G}, \csh{F}) \cdot c_1(\csh{F}, \csh{E})).
\end{align*}
Hence, we get $\chi(\csh{G}, \csh{E}) = 0$ iff $c_1(\csh{G}, \csh{F}) \cdot c_1(\csh{F}, \csh{E})
= c_1(\csh{E}, \csh{F}) \cdot c_1(\csh{F}, \csh{G}) = eg$.
\end{proof}

\begin{remark*}
We point out that for the case that $\sh{E}, \sh{F}$ and $\sh{F}, \sh{G}$ are proper exceptional pairs,
Proposition \ref{intersectionprop} yields only a necessary, but not sufficient criterion for
$\sh{E}, \sh{F}, \sh{G}$ to form a proper exceptional triple.
\end{remark*}

\begin{sub}\label{mutlemma}
For a numerically exceptional triple $\sh{E}, \sh{F}, \sh{G}$, the Chern classes and Euler
characteristic transform for right mutation as follows:
\begin{align*}
c_1(\sh{F}, R_\sh{F}\sh{E}) & = c_1(\sh{E}, \sh{F}), \quad
c_1(R_\sh{F}\sh{E}, \sh{G}) = \chi(\sh{E}, \sh{F}) c_1(\sh{F}, \sh{G}) -
c_1(\sh{E}, \sh{G})\\
\chi(\sh{F}, R_\sh{F}\sh{E}) & = \chi(\sh{E}, \sh{F}), \quad
\chi(R_\sh{F}\sh{E}, \sh{G}) = \chi(\sh{E}, \sh{F}) \chi(\sh{F}, \sh{G}) -
\chi(\sh{E}, \sh{G}).
\end{align*}
Similarly, for left mutation, we get:
\begin{align*}
c_1(L_\sh{F}\sh{G}, \sh{F}) & = c_1(\sh{F}, \sh{G}), \quad
c_1(\sh{E}, L_\sh{F}\sh{G}) = \chi(\sh{F}, \sh{G}) c_1(\sh{E}, \sh{F}) -
c_1(\sh{E}, \sh{G})\\
\chi(L_\sh{F}\sh{G}, \sh{F}) & = \chi(\sh{F}, \sh{G}), \quad
\chi(\sh{E}, L_\sh{F}\sh{G}) = \chi(\sh{F}, \sh{G}) \chi(\sh{E}, \sh{F}) -
\chi(\sh{E}, \sh{G}).
\end{align*}
If $e, f \neq 0$, we have by \ref{eulersym} (\ref{eulersymvi}) that
$\chi(\sh{E}, \sh{F}) = \frac{1}{ef}(a + e^2 + f^2)$ where
$a = c_1(\sh{E}, \sh{F})^2$, and the rank formulas of \ref{mutationrank} specialize as follows:
\begin{equation*}
\rk L_\sh{E}\sh{F} = \frac{a + e^2}{f} \quad \text{ and } \quad
\rk R_\sh{F}\sh{E} = \frac{a + f^2}{e}.
\end{equation*}
\end{sub}

\section{Exceptional Sequences containing objects of rank zero}\label{rankzerosection}

As we have seen in paragraphs \ref{eulerchar1} and \ref{eulersym}, the Riemann-Roch formula
does not lead to a uniform treatment of exceptional objects of rank zero as it does for objects
of nonzero rank. Indeed, one should think of such objects as associated to exceptional
divisors of blow-ups, as indicated by a classical construction due to Orlov.

\begin{example}\label{blowupexample}
Let $b: \tilde{X} \rightarrow X$ be a blow-up of a point with exceptional divisor $E$ and let
$\sh{E}_1, \dots, \sh{E}_n$ be a full exceptional sequence on $X$. Then by
\cite{Orlov93}, $\sh{O}_E(E), \mathbf{L}b^*\sh{E}_1, \dots, \mathbf{L}b^*\sh{E}_n$
is a full exceptional sequence on $\tilde{X}$. Clearly, $c_2(\sh{O}_E(E)) = 0$
and $c_1(\sh{O}_E(E)) c_1(\mathbf{L}b^*\sh{E}_i) = 0$ for all $i$. It follows from \ref{eulersym}
(\ref{eulersymiii}) that $\chi(\sh{O}_E(E), \mathbf{L}b^*\sh{E}_i) = -e_i$ for every $i$.
\end{example}

\begin{sub}
A distinctive feature of this example is that the first Chern class of $\sh{O}_E(E)$ is
orthogonal to the first Chern classes of the rest of the sequence and that its second Chern
class is zero. We will see (Theorem \ref{deltatheorem}) that, possibly after twisting with
a line bundle, this is always true for rank zero objects in an exceptional sequence of
maximal length. However, as we can see by formulas \ref{eulersym}, this is not an immediate
consequence of the Riemann-Roch formula and we have not yet developed enough machinery to
prove this fact. In this section we will describe some general features of exceptional
objects of rank zero and their semi-orthogonal complements sufficient to motivate and
state Theorem \ref{deltatheorem},
but we will only be able to prove the theorem in Section \ref{globalsection}.
\end{sub}

Let $\sh{E}_1, \dots, \sh{E}_n$ be an exceptional sequence and denote
$\sh{Z}_1, \dots, \sh{Z}_t$ the maximal sub-sequence consisting of objects of rank
zero. The following lemma shows that the $\sh{Z}_i$ can never represent a semi-orthogonal
basis of $\knum(X)$.

\begin{lemma}\label{rankzerosequences}
Under above assumptions we have $t \leq n - 3$.
\end{lemma}

\begin{proof}
By Lemmas \ref{eulerchar1} (\ref{eulerchar1i}) and \ref{eulersym} (\ref{eulersymv}),
the Chern classes $c_1(\sh{Z}_i)$ form an orthogonal system of vectors of length
$-1$ in $\chnum^1(X)$. Then $t \leq n - 3 = \rk \chnum^1(X) - 1$ by the Hodge index theorem.
\end{proof}

Given a divisor $D$, we can consider the twisted sequence $\sh{E}_1(D), \dots, \sh{E}_n(D)$.
For the sub-sequence of the $\sh{Z}_i$, we observe the following:

\begin{lemma}\label{shiftlemma}
Let $m_1, \dots, m_t$ be any integers. Then with above notation, there exists a divisor $D$
such that $c_2(\sh{Z}_i(D)) = m_i$ for every $1 \leq i \leq t$.
\end{lemma}

\begin{proof}
By the multiplicative property of the Chern character we have $\ch(\sh{Z}_i(D)) =
\ch(\sh{Z}_i) \cdot \ch(\sh{O}(D))$ for every $i$. From this we compute that
$c_2(\sh{Z}_i(D)) = c_2(\sh{Z}_i) - c_1(\sh{Z}_i) D$. As we have observed in the
proof of Lemma \ref{rankzerosequences}, the $c_1(\sh{Z}_i)$ form an orthogonal set of
vectors of length $-1$ with respect to the intersection form. Hence, the divisor
$D = -\sum_{i = 1}^t (c_2(\sh{Z}_i) - m_i) c_1(\sh{Z}_i)$ satisfies
$c_1(\sh{Z}_i) D = c_2(\sh{Z}_i) - m_i$ for all $i$.
\end{proof}

Now consider any exceptional object $\sh{Z}$ of rank zero. We want to describe the
relative configurations of the left- and right-orthogonal complements in $\knum(X)$.

\begin{sub}\label{leftrighteuler}
Both $\chi(-, \sh{Z})$ and $\chi(\sh{Z}, -)$ induce linear forms on $\knum(X)$ and with respect
to these forms we denote $\mathbf{L}$ and $\mathbf{R}$ the left- and right-orthogonal complements of
$\sh{Z}$ in $\knum(X)$, respectively. By the integrality of the Euler form and the fact that
$\chi(\sh{Z}, \sh{Z}) = 1$ it follows that $\mathbf{L}$ and $\sh{Z}$ (respectively, $\mathbf{R}$
and $\sh{Z})$ generate $\knum(X)$.
Furthermore, because the first Chern class is additive on
complexes, we obtain another linear form on $\knum(X)$ which is given as $$\phi_\sh{Z}: \knum(X)
\longrightarrow \Z, \quad \sh{E} \mapsto c_1(\sh{Z}) c_1(\sh{E}).$$ We denote its orthogonal complement
in $K_0(X)$ by $\mathbf{C}$. Note that because $\sh{Z}$ is primitive in $\knum(X)$ and $\chi(\sh{Z},
\sh{Z}) = 1 = -c_1(\sh{Z})^2$ it follows that all these forms are integral and primitive in $\knum(X)^*$.
By a result of Thomason \cite[Theorem 2.1]{Thomason97}, the subgroups of $K_0(X)$ correspond precisely
to the full dense triangulated subcategories of $D^b(X)$. In particular, we denote by $\mathfrak{L}$
and $\mathfrak{R}$ those subcategories which correspond to the preimages of $\mathbf{L}$ and
$\mathbf{R}$ in $K_0(X)$, respectively. For any object $\sh{E}$ of $\mathfrak{L}$ we have
by \ref{eulersym} (\ref{eulersymiii}), (\ref{eulersymv}):
$$
\chi(\sh{Z}, \sh{E}) =
\begin{cases}
0 & \text{if } e = 0,\\
-e(1 + 2 c_1(\sh{Z}) s(\sh{E}) + 2 c_2(\sh{Z})) & \text{otherwise}.
\end{cases}
$$
Similarly, for any object $\sh{F}$ of $\mathfrak{R}$ we have:
$$
\chi(\sh{F}, \sh{Z}) =
\begin{cases}
0 & \text{if } f = 0,\\
-f(1 + 2 c_1(\sh{Z}) s(\sh{F}) + 2 c_2(\sh{Z})) & \text{otherwise}.
\end{cases}
$$
\end{sub}

\begin{lemma}\label{zerointersectionequal}
Let $\sh{E}, \sh{E}' \in \mathfrak{L}$ and $\sh{F}, \sh{F}' \in \mathfrak{R}$.
Then with above notation, we have $c_1(\sh{Z}) s(\sh{E}) = c_1(\sh{Z}) s(\sh{E}')$ and
$c_1(\sh{Z}) s(\sh{F}) = c_1(\sh{Z}) s(\sh{F}')$ whenever $e, e', f, f' \neq 0$.
\end{lemma}

\begin{proof}
We have $c_1(\sh{Z}, \sh{E}) = -e c_1(\sh{Z})$ for any $\sh{E}$ in $\mathfrak{L}$ and hence by
Lemma \ref{eulersym} (\ref{eulersymi}), we get $e \chi(\sh{Z}, \sh{E}') =
e' \chi(\sh{Z}, \sh{E})$. Then the assertion follows from
\ref{eulersym} (\ref{eulersymiii}). The statement for $\sh{F}$ and $\sh{F}'$ follows
analogously.
\end{proof}

\begin{definition}\label{defectdef}
With above notation, for some $\sh{E} \in \mathfrak{L}$, $\sh{F} \in \mathfrak{R}$ with $e, f \neq 0$,
we denote
\begin{align*}
\delta_\sh{Z} & := c_1(\sh{Z}) s(\sh{E}) + c_2(\sh{Z}),\\
\varepsilon_\sh{Z} &:= c_1(\sh{Z}) s(\sh{F}) + c_2(\sh{Z}).
\end{align*}
\end{definition}

By Lemma \ref{zerointersectionequal}, $\delta_\sh{Z}$ and $\varepsilon_\sh{Z}$ are independent
of the choice of $\sh{E}$ and $\sh{F}$.

\begin{lemma}\label{chirank}
\begin{enumerate}[(i)]
\item\label{chiranki} We have $K_X c_1(\sh{Z}) = -(1 + 2 \delta_\sh{Z}) = 1 + 2 \varepsilon_\sh{Z}$ and
therefore $\delta_\sh{Z} + \varepsilon_\sh{Z} + 1 = 0$.
\item\label{chirankiii} The restriction of $\chi(\sh{Z}, -)$ to $\mathbf{L}$ coincides
with $-(1 + 2 \delta_\sh{Z})$ times the rank function.
\item\label{chirankiv} The restriction of $\chi(-, \sh{Z})$ to $\mathbf{R}$ coincides
with $-(1 + 2 \varepsilon_\sh{Z})$ times the rank function.
\item\label{chirankii} The following formulas hold for any object $\sh{G}$ of $D^b(X)$:
\begin{align*}
\chi(\sh{Z}, \sh{G}) + c_1(\sh{G}) c_1(\sh{Z}) + g c_2(\sh{Z}) &
= -g (1 + \delta_\sh{Z}) = g \varepsilon_\sh{Z},\\
\chi(\sh{G}, \sh{Z}) + c_1(\sh{G}) c_1(\sh{Z}) + g c_2(\sh{Z}) &
= g \delta_\sh{Z} = -g (1 + \varepsilon_\sh{Z}).
\end{align*}
In particular, by taking a rank one object for $\sh{G}$ we see that both $\delta_\sh{Z}$ and $\varepsilon_\sh{Z}$ are integers.
\end{enumerate}
\end{lemma}

\begin{proof}
(\ref{chiranki})
For any $\sh{E} \in \mathfrak{L}$, $\sh{F} \in \mathfrak{R}$ with $e, f \neq 0$ follows from
\ref{eulersym} (\ref{eulersymi}), (\ref{eulersymiii}), (\ref{eulersymiv}) that
$\chi(\sh{Z}, \sh{E}) = e K_X c_1(\sh{Z}) = -e(1 + 2 \delta_\sh{Z})$ and
$\chi(\sh{F}, \sh{Z}) = -f K_X c_1(\sh{Z}) = -f(1 + 2 \varepsilon_\sh{Z}).$

(\ref{chirankiii})
For any $\sh{E} \in \mathfrak{L}$, we have $\chi(\sh{Z}, \sh{E}) = - e (1 + 2 \delta_\sh{Z})$ by
\ref{leftrighteuler}. As the restriction of the rank function to
$\mathbf{L}$ then is completely determined by these values, the assertion follows.

(\ref{chirankiv})
For any $\sh{F} \in \mathfrak{R}$, we have $\chi(\sh{F}, \sh{Z}) = -f (1 + 2 \delta_\sh{Z})$ by
\ref{leftrighteuler} and we conclude as in (\ref{chirankiii}).

(\ref{chirankii})
From \ref{RR} and (\ref{chiranki}) we get immediately $\chi(\sh{Z}, \sh{G}) + c_1(\sh{G}) c_1(\sh{Z})
+ g c_2(\sh{G}) = \frac{g}{2}K_X c_1(\sh{Z}) - \frac{g}{2} = -g(1 + \delta_\sh{Z})$
and $\chi(\sh{G}, \sh{Z}) + c_1(\sh{G}) c_1(\sh{Z}) + g c_2(\sh{G}) = - \frac{g}{2}K_X c_1(\sh{Z}) -
\frac{g}{2} = g \delta_\sh{Z} = -g (1 + \varepsilon_\sh{Z})$.
\end{proof}

\begin{sub}\label{spread}
Let $\sh{G}$ be an object in $\mathfrak{L} \cap \mathfrak{R}$, then the left hand sides of
both equations in Lemma \ref{chirank} (\ref{chirankii}) coincide and, by the integrality of
$\delta_\sh{Z}$, the right hand sides can only be equal if $g = 0$. Then
from the existence of objects of nonzero rank it follows that $\mathbf{L} \cap \mathbf{R}$
and $[\sh{Z}]$ cannot generate $\knum(X)$. Moreover, both $\mathbf{L}$ and $\mathbf{R}$ are
saturated sublattices of corank $1$ in $\knum(X)$. It follows that $\mathbf{L} \neq \mathbf{R}$
and $\mathbf{L} \cap \mathbf{R}$ is a saturated sublattice of corank two in $\knum(X)$ consisting
of objects of rank zero. Furthermore, it follows that the linear forms $\chi(-, \sh{Z})$ and
$\chi(\sh{Z}, -)$ are linearly independent.

If we denote $L := \knum(X) / \mathbf{L} \cap \mathbf{R} \simeq \Z^2$ it follows immediately that
$\chi(-, \sh{Z})$ and $\chi(\sh{Z}, -)$ descend to linearly independent linear forms on $L$.
Moreover, as $\rk(\mathbf{L} \cap \mathbf{R}) = 0$, the rank function descends as well, as does
$\phi_\sh{Z}$ by Lemma \ref{chirank} (\ref{chirankii}).

Another consequence of Lemma \ref{chirank} (\ref{chirankii}) is that the intersection products
$c_1(\sh{Z}) s(\sh{E})$ and $c_1(\sh{Z}) s(\sh{F})$ are integral for $\sh{E} \in \mathfrak{L}$
and $\sh{F} \in \mathfrak{R}$, respectively. So, noting that $c_1(\sh{Z}(D)) = c_1(\sh{Z})$,
$\delta_{\sh{Z}(D)} = \delta_\sh{Z}$, and $\varepsilon_{\sh{Z}(D)} = \varepsilon_\sh{Z}$ for
any divisor $D$, we can choose by Lemma \ref{shiftlemma} a divisor $D$ such that either
$\mathbf{L} = \ker \phi_{\sh{Z}(D)}$ or $\mathbf{R} = \ker \phi_{\sh{Z}(D)}$. So, up to
twist by an invertible sheaf we may assume without loss of generality that, say, $\chi(-, \sh{Z}) =
-\phi_\sh{Z}$ and, in particular, $\sh{O} \in \mathfrak{L}$, $\omega \in \mathfrak{R}$.
Figure \ref{rankzerofigure} shows the configurations of $\mathbf{L}$ and
$\mathbf{R}$ in $L$
for the case $\delta_\sh{Z} = c_2(\sh{Z}) = 2$.

\begin{figure}[h]
\begin{tikzpicture}[axis/.style={very thick, ->, >=stealth'}]
\draw[axis] (0,0)  -- (1,0) node(xline)[right] {$\sh{Z}$};
\draw[axis] (0,0)  -- (0,1) node(yline)[left] {$\sh{O}$};
\draw (0,-1) -- (0,2) node [left] {$\mathbf{L}$};
\draw (-5,-1) -- (7.5, 1.5) node [above]{\footnotesize $\mathbf{R}$};
\draw (1,1) node[circle,fill,inner sep=1pt,label=above:$\sh{O}(c_1(\sh{Z}))$]{};
\draw (2,1) node[circle,fill,inner sep=1pt]{};
\draw (3,1) node[circle,fill,inner sep=1pt]{};
\draw (4,1) node[circle,fill,inner sep=1pt]{};
\draw (5,1) node[circle,fill,inner sep=1pt,label=above:$\omega$]{};
\draw (-1,1) node[circle,fill,inner sep=1pt]{};
\draw (-2,1) node[circle,fill,inner sep=1pt]{};
\draw (-3,1) node[circle,fill,inner sep=1pt]{};
\draw (-4,1) node[circle,fill,inner sep=1pt]{};
\draw (-5,1) node[circle,fill,inner sep=1pt,label=above:$\omega^{-1}$]{};
\end{tikzpicture}
\caption{The configuration of $\mathbf{L}$ and $\mathbf{R}$ in $L$ for $\delta_\sh{Z} = 2$.}
\label{rankzerofigure}
\end{figure}
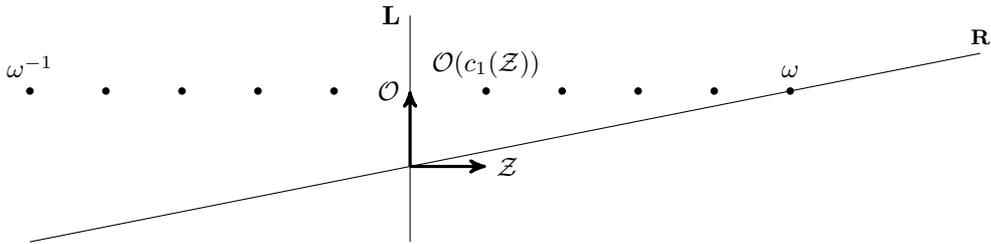

Here, $\sh{Z}$ and $\sh{O}$ represent a basis of $L$ whose dual is naturally given by
$\chi(-, \sh{Z}) = -\phi_\sh{Z}$ and $\rk$. In $L$, the classes of rank one objects have
coordinates $\sh{O} + k \sh{Z}$ which can be represented by line bundles $\sh{O}(k c_1(\sh{Z}))$.
In particular, $\omega \sim \sh{O}((1 + 2\delta_\sh{Z}) c_1(\sh{Z}))$.
With the relation $\delta_\sh{Z} + \varepsilon_\sh{Z} + 1
= 0$ and $\delta_\sh{Z} = c_2(\sh{Z})$ we moreover observe with \ref{shiftchern}:
$$
\varepsilon_\sh{Z} = \delta_{\sh{Z}[1]} \text{ and } \delta_\sh{Z} = \varepsilon_{\sh{Z}[1]}.
$$
\end{sub}

\begin{sub}
We are interested in the particular situation, where $\sh{Z}$ is part of an exceptional sequence
of the form, say, $\sh{Z}, \sh{E}_2, \dots, \sh{E}_n$ where we can assume without loss of generality
that $c_1(\sh{Z}) c_1(\sh{E}_i) = 0$ for any $i$. Then the classes of $\sh{E}_2, \dots, \sh{E}_n$ in
$\knum(X)$ form a semiorthogonal basis of $\mathbf{L}$ and $\sh{E}_2 \otimes \omega, \dots,
\sh{E}_n \otimes \omega$ is a semiorthogonal basis of $\mathbf{R}$ Alternatively, by $n - 1$ left
mutations we can move $\sh{Z}$ to the
rightmost end of a sequence $\sh{F}_1, \dots, \sh{F}_{n - 1}, \sh{Z}$ and obtain another semiorthogonal
basis
$\sh{F}_1, \dots, \sh{F}_{n - 1}$ of $\mathbf{R}$.

In any case, $\delta_\sh{Z} = c_2(\sh{Z})$ represents the ``spread'' of $\mathbf{L}$ and $\mathbf{R}$
in $L$. In Example \ref{blowupexample}, we have seen that $\delta_{\sh{O}_E(E)} = 0$ holds and
therefore $\mathbf{L}$ and $\mathbf{R}$ generate $\knum(X)$. It turns out that this indeed is a
general feature of exceptional sequences, as the following theorem shows, which, however, we cannot
yet prove.

\begin{theorem}\label{deltatheorem}
Let $\sh{Z}$ be an exceptional object of rank zero which can be included in an exceptional sequence
$\sh{Z}, \sh{E}_2, \dots, \sh{E}_n$. Then $\delta_\sh{Z} \in \{0, -1\}$.
\end{theorem}

The proof will be postponed until Section \ref{globalsection}. Until then we will have to take a
defect $\delta_\sh{Z}$ into account whenever an exceptional rank zero object $\sh{Z}$ is part of our
exceptional sequence. However, once Theorem \ref{deltatheorem} is established, the following
corollary shows that we can essentially forget about them.

\begin{corollary}\label{deltaz}
\begin{enumerate}[(i)]
\item\label{deltaziii}
If $\delta_\sh{Z} = 0$ then $\chi(\sh{Z}[1], -)$ coincides with the rank
function on $\mathbf{L}$.
If $\delta_\sh{Z} = -1$ then $\chi(-, \sh{Z}[1])$ coincides with the rank function on $\mathbf{R}$.
\item\label{deltaziv} If $\sh{Z}, \sh{E}_2, \dots, \sh{E}_n$ is an exceptional sequence then
so is $\sh{Z}[\epsilon](D), \sh{E}_2(D), \dots, \sh{E}_n(D)$ for any divisor $D$ and $\epsilon \in \Z$.
In particular, we can choose $D$ and $\epsilon
\in \{0, 1\}$ such that $c_2(\sh{Z}[\epsilon](D)) = \delta_{\sh{Z}[\epsilon](D)} = 0$ and
$-c_1(\sh{Z}[\epsilon](D)) K_X = 1$.
\end{enumerate}
\end{corollary}
\end{sub}

\section{Toric systems}\label{toricsystemsection}

In \cite{HillePerling11}, exceptional sequences of invertible sheaves $\sh{O}(D_1), \dots,
\sh{O}(D_n)$ have been considered. For such sequences, so-called {\em toric systems} have
been introduced, which represent a normal form for such sequences. More precisely, a
toric system is simply given by forming the differences $A_i := D_{i + 1} - D_i$ for all
$1 \leq i < n$ and $A_n := D_1 - D_n - K_X$. Such a toric
system satisfies the following equations:
\begin{enumerate}[(i)]
\item $A_i \cdot A_{i + 1} = 1$ for all $i$,
\item $A_i \cdot A_j = 0$ if $i \neq j$ and $\{i, j\} \neq \{k, k + 1\}$ for any $1 \leq k \leq n$,
\item $\sum_{i = 1}^n A_i = -K_X$.
\end{enumerate}
In \cite{HillePerling11} the peculiar fact was observed that a toric system is equivalent to
the data of a smooth complete toric surface, which this way becomes a combinatorial invariant
of an exceptional sequence of invertible sheaves.

In this section we will extend the notion of toric systems to the case of general exceptional sequences.
This generalization will be straightforward for the most part, with two notable differences:
\begin{enumerate}
\item It is necessary to pass to rational Chern classes, i.e. to an exceptional sequence
we will associate elements $A_i$ in a similar fashion, but they are now constructed as
elements of $\chnum^1(X)_\Q$.
\item Objects of rank zero cannot be treated uniformly together with objects of nonzero
rank.
\end{enumerate}

\begin{sub}\label{int}
We start with an exceptional sequence $\sh{E}_1, \dots, \sh{E}_t$, where $e_i \neq 0$ for
all $i$, which we assume extended to a cyclic exceptional sequence. This in particular
implies that $c_1(\sh{E}_i, \sh{E}_{i + 1}) = c_1(\sh{E}_{i + t}, \sh{E}_{i + 1 + t})$
for all $i$.
For $t > 2$ the following are straightforward consequences of \ref{additivity},
Proposition \ref{intersectionprop}, Serre duality, and the Riemann-Roch formula:
\begin{enumerate}[(i)]
\item $c_1(\sh{E}_{i - 1}, \sh{E}_i) \cdot c_1(\sh{E}_i, \sh{E}_{i + 1}) =
e_{i - 1} e_{i + 1}$ for every $i \in \Z$.
\item $c_1(\sh{E}_{i - 1}, \sh{E}_i) \cdot c_1(\sh{E}_{j - 1}, \sh{E}_j) = 0$
for $1 < \vert i - j \vert < t - 1$.
\end{enumerate}
\end{sub}

\begin{sub}\label{nonzerotoricsystem}
Observe that for any three objects $\sh{E}, \sh{F}, \sh{G}$ we have
$s(\sh{E}, \sh{F}) + s(\sh{F}, \sh{G}) = s(\sh{E}, \sh{G})$ by \ref{additivity}.
 The intersection product extends in a natural way to a $\Q$-valued
bilinear form on $\chnum^1(X)_\Q$, so that we can reformulate the equalities of \ref{int}
as follows:
\begin{enumerate}[(i)]
\item $s(\sh{E}_{i - 1}, \sh{E}_i) \cdot s(\sh{E}_i, \sh{E}_{i + 1}) = 1 / e_i^2$,
\item $s(\sh{E}_{i - 1}, \sh{E}_i) \cdot s(\sh{E}_{j - 1}, \sh{E}_j) = 0$
for $1 < \vert i - j \vert < t - 1$.
\end{enumerate}	
Moreover, by $s(\sh{E}_t, \sh{E}_{t + 1}) = s(\sh{E}_t, \sh{E}_1 \otimes
\omega^{-1}) = s(\sh{E}_t, \sh{E}_1) - K_X$, we have:
\begin{enumerate}[(i)]
\setcounter{enumi}{2}
\item $\sum_{i = 1}^t s(\sh{E}_i, \sh{E}_{i + 1}) = -K_X$.
\end{enumerate}
\end{sub}

\begin{sub}\label{zerotoricsystem}
Assume that we have an exceptional sequence $\sh{E}_1, \dots, \sh{E}_t$ and assume that one of
the $\sh{E}_i$ has rank zero. By choosing the appropriate winding in the cyclic sequence, we
can always assume without loss of generality that we have $\sh{E}_1 \simeq \sh{Z}$ with $z = 0$.
Then for any pair $\sh{E}_i, \sh{E}_j$ with $e_i, e_j \neq 0$ we have by Lemma
\ref{zerointersectionequal}, that $c_1(\sh{Z}) s(\sh{E}_i, \sh{E}_j) = 0$ if
$1 < i < j \leq t$. If $1 - t < j < 1 < i \leq t$ and $i - j < t$,
then $\sh{E}_j = \sh{E}_{j + t} \otimes \omega$ and therefore $s(\sh{E}_j, \sh{E}_i) =
s(\sh{E}_{j + t}, \sh{E}_i) - K_X$, hence $c_1(\sh{Z}) s(\sh{E}_j, \sh{E}_i) = - c_1(\sh{Z})
K_X = -(1 + 2 \delta_\sh{Z})$ for some undetermined integer $\delta_\sh{Z}$ by Lemma
\ref{chirank} (\ref{chiranki}).
\end{sub}

\begin{sub}\label{toricsystemconstruction}
Now consider an arbitrary exceptional sequence $\sh{E}_1, \dots, \sh{E}_n$. Then we can
partition $\{1, \dots, n\} = I\, \amalg J$, where $I = \{i_1 < \dots < i_t\}$, $J =
\{j_1 < \dots <  j_{n - t}\}$, and such that $e_i = 0$ iff $i \in I$. Then we set:
\begin{align*}
E_k & := c_1(\sh{E}_{i_k}) \ \text{ for } 1 \leq k \leq t,\\
A_k & := s(\sh{E}_{j_k}, \sh{E}_{j_{k + 1}}) \ \text{ for } 1 \leq k < n - t,\\
A_{n - t} & := s(\sh{E}_{j_{n - t}}, \sh{E}_{j_1} \otimes \omega^{-1})
\end{align*}
Clearly, the $A_k$ satisfy the conditions listed in  \ref{nonzerotoricsystem} and it follows from
from \ref{eulerchar1} (\ref{eulerchar1i}), \ref{eulersym} (\ref{eulersymiv}) and
\ref{zerotoricsystem} that $E_i^2 = -1$ and
there exist integers $\delta_1, \dots, \delta_t$ such that
 $-K_X E_i = 1 + 2 \delta_i$ for all $i \in I$, and $E_i \cdot E_k = 0$ for all $i \neq k$.
Moreover, by \ref{zerotoricsystem} there exists for every $i \in I$ precisely one $j \in J$
such that $E_i . A_j \neq 0$.
\end{sub}

For easier notation, we give a formal definition for above data.

\begin{definition}\label{toricsystemdef}
An {\em abstract toric system} on $X$ is given by the following data:
\begin{enumerate}
\item\label{toricsystemdef1} For $0 \leq t \leq n - 3$ a collection of integral divisor classes $E_1, \dots, E_t$ and
integers $\delta_1, \dots, \delta_t$ with
$E_i^2 = -1$ and $-K_X E_i = 1 + 2 \delta_i$ for all $i$ and $E_i \cdot E_j = 0$ for all $i \neq j$.
\item\label{toricsystemdef2} A sequence of ranks $r_1, \dots, r_{n - t} \in \Z
\setminus \{0\}$ with $\gcd\{r_1, \dots, r_{n - 1}\} = 1$.
\item\label{toricsystemdef3} A sequence of $\Q$-divisor classes $A_1, \dots, A_{n - t} \in \chnum^1(X)_\Q$
such that
\begin{enumerate}[(i)]
\item\label{toricsystemdef3i} $r_i r_{i + k + 1} (A_i + \cdots + A_{i + k})$ is integral for every $i$ and every $0 \leq k < n - t$,
\item\label{toricsystemdef3ii} $A_i \cdot A_{i + 1} = \frac{1}{r_{i + 1}^2}$ for every $i$,
\item\label{toricsystemdef3iii} $A_i \cdot A_j = 0$ if $i \neq j$ and $\{i, j\} \neq \{k, k + 1\}$ for any $1 \leq k \leq t$,
\item\label{toricsystemdef3iv} $\sum_{i = 1}^{n - t} A_i = -K_X$.
\end{enumerate}
\item\label{toricsystemdef4} A function $\phi: \{1, \dots, t\} \rightarrow
\{1, \dots, n - t\}$ such that $E_i \cdot A_j \neq 0$ if and only if $j = \phi(i)$
(and thus $E_i \cdot A_{\phi(i)} = 1 + 2 \delta_i$ by (\ref{toricsystemdef1}) and
(\ref{toricsystemdef3iv})).
\end{enumerate}
Note that the indices of the $r_i$ and $A_i$ are to be read cyclically; in particular, we have
$A_{n - t} \cdot A_1 = 1 / r_1^2$. Also note that by Lemma \ref{rankzerosequences} we can make
the implicit assumption that $n - t \geq 3$. Moreover, $\phi$ and the $r_i$ are completely determined
by the divisors $A_j$, $E_k$, so, usually we will specify an abstract toric system only by the data
$E_1, \dots, E_t, A_1, \dots, A_{n - t}$.

A {\em toric system} is an abstract toric system which can be constructed from an exceptional
sequence $\sh{E}_1, \dots, \sh{E}_n$ by the procedure described in \ref{toricsystemconstruction}.
\end{definition}

Once Theorem \ref{deltatheorem} is proven, by Corollary \ref{deltaz} (\ref{deltaziv})
it will in many situations be harmless to require that the $\delta_i$ are zero.

\begin{remark}\label{delayremark}
In the following we will exclusively consider actual (i.e. non-abstract) toric systems.
In \cite[\S 2]{HillePerling11}, some effort has been devoted to the inverse problem, i.e. the
question whether for a given toric system we can check implications such as vanishing
of the $\chi(\sh{O}(-A_i))$.
However, contrary to the case of line bundles, the association of toric systems to exceptional
objects is not as straightforward at this stage. Instead, our strategy in the subsequent sections
will be to reduce such questions to the case of sequences of rank one objects (see Remark
\ref{delayedremark} below).
\end{remark}

\begin{example}\label{p2blowupexample}
Consider the strongly exceptional sequence $\sh{T}, \sh{O}(2), \sh{O}(4)$ on $\mathbb{P}^2$,
where $\mathcal{T}$ denotes the tangent sheaf. If we denote $H$ the class of a line in
$\Ch^1(\mathbb{P}^2)$, then the toric system associated to this sequence is given by
$A_1, A_2, A_3$ =
$\frac{1}{2}H, 2H, \frac{1}{2}H$. Now we take any point $x \in \mathbb{P}^2$ and
denote $b: \mathbb{F}_1 \simeq \widetilde{\mathbb{P}}^2 \rightarrow \mathbb{P}^2$ the blow-up at $x$ with exceptional curve $E$. For ease of notation we identify $E$ with its class
in $\Ch^1(\mathbb{F}_1)$. We also identify $H$ with its pull-back in $\Ch^1(\mathbb{F}_1)$.
Completing to a full exceptional sequence by adding $\sh{O}_E(E)$ we get
$\sh{O}_E(E), b^*\sh{T}, b^*\sh{O}(2), b^*\sh{O}(4)$.
Then the toric system associated to this sequence is given by the
($-1$)-divisor $E_1 = E$, the rational classes $A_1, A_2, A_3$ = $\frac{1}{2}H, 2H, \frac{1}{2}H - E$,
and $\phi: \{1\} \rightarrow \{1, 2, 3\}$ with $\phi(1) = 3$.
Now, by right mutating
the pair $\sh{O}_E(E), b^*\sh{T}$, we obtain $b^*\sh{T}, \sh{R}, b^*\sh{O}(2), b^*\sh{O}(4)$.
We have $\chi(\sh{O}_E(E), b^*\sh{T}) = -2$, hence
$\rk \sh{R} = -4$. Moreover, we get $s(b^*\sh{T}, \sh{R}) = \frac{1}{4} E$ and
consequently the new toric system consists of four rational divisor classes which are given by
\begin{equation*}
\frac{1}{4} E, \frac{1}{2} H - \frac{1}{4} E, 2H, \frac{1}{2} H - E.
\end{equation*}
\end{example}

\section{Toric systems and their Gale dual}\label{galedualsection}

Let $\mathbf{A} = E_1, \dots, E_t$, $A_1, \dots, A_{n - t}$ be an abstract toric system. It will be an important
technical aspect to consider the projection of the $A_j$ onto the orthogonal complement of the $E_i$ in
$\chnum^1(X)_\Q$. Recall that $E_i \cdot A_{\phi(i)} = -(1 + 2\delta_i)$ and $E_i \cdot A_j = 0$ for $j \neq \phi(i)$.

\begin{definition}\label{contractiondef}
The {\em contraction} $\tilde{\mathbf{A}}$ of $\mathbf{A}$ is given by
$\tilde{A}_1, \dots, \tilde{A}_{n - t}$, where
\begin{equation*}
\tilde{A}_i = A_i + \sum_{j \mid \phi(j) = i} (1 + 2 \delta_j) E_j
\end{equation*}
\end{definition}

Both $\mathbf{A}$ and $\tilde{\mathbf{A}}$ give rise to subgroups of
$\chnum^1(X)_\Q \simeq \Q^{n - 2}$
given by $A := \langle E_1, \dots, E_t, A_1, \dots, A_{n - t} \rangle_\Z$ and $\tilde{A} :=
\langle \tilde{A}_1, \dots, \tilde{A}_{n - t} \rangle_\Z$.
Clearly, both $A$ and $\tilde{A}$ are finitely generated and torsion free $\Z$-modules
of rank at most $n - 2$. It is easy to see that the $\tilde{A}_i$ still satisfy conditions
(\ref{toricsystemdef3i}),  (\ref{toricsystemdef3ii}), (\ref{toricsystemdef3iii}) of Definition
\ref{toricsystemdef}, but $\sum_{i = 1}^{n - t} \tilde{A}_i = -K_X + \sum_{j = 1}^t (1 + 2 \delta_j) E_j$.

\begin{proposition}\label{toricsystemrank}
$\rk A = n - 2$ and $\rk \tilde{A} = n - t - 2$.
\end{proposition}

\begin{proof}
As the $E_i$ form an orthogonal system of divisors which by construction contain
$\tilde{A}$ in their orthogonal complement, it suffices to show that $\rk \tilde{A} = n - t - 2$.
Starting with the observation that by Definition \ref{toricsystemdef} (\ref{toricsystemdef3ii}) the
$\tilde{A}_i$, and in particular $\tilde{A}_1$, are all nonzero
we will show by induction that $\tilde{A}_1, \dots, \tilde{A}_i$ are $\Q$-linearly independent for
$1 \leq i \leq n - t - 2$. So, for $1 < i \leq n - t - 2$ we assume that $\tilde{A}_1, \dots,
\tilde{A}_{i - 1}$ are linearly independent. Then, for any $\Q$-linear combination
$B := \sum_{j = 1}^{i - 1} \alpha_j \tilde{A}_j$, we have $B \cdot \tilde{A}_{i + 1} = 0$.
However, we have $\tilde{A}_i \cdot \tilde{A}_{i + 1} = \frac{1}{r_i^2} \neq 0$, hence $\tilde{A}_i$
cannot be contained in the linear span of $\tilde{A}_1, \dots, \tilde{A}_{i - 1}$, hence
$\tilde{A}_1, \dots, \tilde{A}_i$ are linearly independent for
all $1 \leq i \leq n - t - 2$ and the assertion follows.
\end{proof}

\begin{sub}\label{dualitydef}
Consider the structural linear maps $c: \Z^n \twoheadrightarrow A$ and $\tilde{c}: \Z^{n - t}
\twoheadrightarrow \tilde{A}$. That is, if we denote $b_1, \dots, b_n$ the standard basis of
$\Z^n$, then we have $c(b_i) = E_i$ for $1 \leq i \leq t$ and $c(b_i) =  A_{i - t}$ for $t < i \leq n$.
For $\Z^{n - t}$ with standard basis $b'_1, \dots, b'_{n - t}$, we have $\tilde{c}(b'_i) = \tilde{A}_i$
for every $i$. We define a linear map $\Phi: \Z^{n - t} \rightarrow \Z^n$ by setting $\Phi(b'_i)
= b_{i + t} + \sum_{j \mid \phi(j) = i} (1 + 2\delta_j) b_j$ for every $1 \leq i \leq n - t$.
Then $\Phi$ induces a linear map $\bar{\Phi}: \tilde{A} \rightarrow A$ with $\bar{\Phi}(\tilde{A}_i) =
A_i + \sum_{j \mid \phi(j) = i}  (1 + 2\delta_j) E_j$ for $1 \leq i \leq n - t$.
Clearly, both
$\Phi$ and $\bar{\Phi}$ are injective and their image is saturated in $\Z^n$ and $A$, respectively.
We obtain the following commutative diagram:
\begin{equation*}
\xymatrix{
0 \ar[r] & M \ar^{\tilde{L}}[r] \ar^{\Psi}[d] & \Z^{n - t} \ar^{\tilde{c}}[r] \ar^{\Phi}[d] & \tilde{A}
\ar[r] \ar^{\bar{\Phi}}[d] & 0 \\
0 \ar[r] & M' \ar^{L}[r] & \Z^n \ar^{c}[r] & A \ar[r] & 0,
}
\end{equation*}
where we set $M = \ker \tilde{c}$ and $M' = \ker(c)$.
We can represent $L$ and $\tilde{L}$ as
row matrices with rows $l_1, \dots, l_n \in (M')^*$ and $\tilde{l}_1, \dots, \tilde{l}_{n - t} \in M^*$,
respectively.
We have $M, M' \simeq \Z^2$ by Proposition \ref{toricsystemrank}. Clearly, $\Psi$ is injective and
it follows from the saturatedness of $\Phi(\Z^{n - t})$ in $\Z^n$ that its cokernel is trivial, hence
$\Psi$ is an isomorphism. Moreover, by dualizing the left part of the diagram and the constrution of
$\Psi$, we immediately obtain following statement.
\end{sub}

\begin{proposition}\label{multiplicity}
Denote $N := M^*$, $N' := (M')^*$ and consider the dual maps
\begin{equation*}
\xymatrix{
\Z^n \ar^{L^T}[r] \ar^{\Phi^T}[d] & N' \ar^{\Psi^T}[d] \\
\Z^{n - t}  \ar^{\tilde{L}^T}[r] & N
}
\end{equation*}
where we identify the column vectors $l_i^T$ and $\tilde{l}_i^T$ with the images of the $i$-th standard
basis vector of $\Z^n$ and $\Z^{n - t}$, respectively, in $N$. Then $\Psi^T$
is an isomorphism which maps $l_i$ to $\tilde{l}_{i - t}$ for
$t < i \leq n$ and $l_i$ to $(1 + 2 \delta_i) \tilde{l}_{\phi(i)}$ for $1 \leq i \leq t$.
\end{proposition}

So we can naturally identify $M = M'$ and $N = N'$, respectively, and consider $\Psi^T$ as the
identity map. Then
both $L^T$ and $\tilde{L}^T$ give rise to almost the same set of column vectors:
the column vectors of $\tilde{L}^T$ coincide with the last $n - t$ column vectors of $L^T$ and the first $t$
column vectors of $L^T$ are multiples of column vectors of $\tilde{L}^T$ by some factors $(1 + 2\delta_i)$.
We will see later that the columns of $\tilde{L}^T$ appear with multiplicity $1$ and, once Theorem \ref{deltatheorem}
is established, we can assume that every column vector $l_{t + j}$ of $L^T$ occurs (up to sign)
with multiplicity $1 + |\{i \mid \phi(i) = j\}|$.

\begin{example}\label{p2blowupexample2}
In Example \ref{p2blowupexample} a toric system was given with ($-1$)-divisor $E$ and
$A_1, A_2, A_3$ = $\frac{1}{2} H, 2H$, $\frac{1}{2} H - E$ such that $E \cdot (\frac{1}{2} H - E) = 1$.
For the Gale dual, we obtain vectors $l_1, l_2, l_3, l_4$ which for a suitable choice
of basis can be represented as $l_1 = l_4 = (1, 0)$, $l_2 = (-1, 4)$, $l_3 = (0, -1)$, i.e.
up to latter multiplicity, the $l_i$ generate fan of $\mathbb{P}(1, 1, 4)$.
The mutated toric system $A_1', A_2', A_3', A_4'$ =
$\frac{1}{4} E, \frac{1}{2} H - \frac{1}{4} E, 2H, \frac{1}{2} H - E$ has Gale duals $l_1, l_2, l_3,
l_4'$ with $l_4' = (3, 4)$ which can be interpreted to generate the fan of a weighted blow-up of
$\mathbb{P}(1, 1, 4)$
with two singular points of order $4$ and $16$, respectively. The corresponding fans are shown
in figures \ref{p2pic2} and \ref{p2pic3} in the introduction. Note that with our
current terminology, the enumeration of the $l_i$ in figure \ref{p2pic2} is that of the
$\tilde{l}_i$ rather than the $l_i$.

In both cases the $l_i$ are primitive lattice vectors and generate the fan of a complete
toric surface. It is easy to see that the singularities are $T$-singularities. We will show that
this and the observation that the
multiplicity $l_1 = l_4$ translates into a weighted blow-up via mutation
are general properties of toric systems.
\end{example}

\section{Moving around objects of rank zero}\label{rankzeromoving}

Let $\mathbf{E} = \sh{E}_1, \dots, \sh{E}_n$ be an exceptional sequence and denote
$E_1, \dots, E_t, A_1, \dots, A_{n - t}$ its associated toric system and $\tilde{A}_1, \dots,
\tilde{A}_{n - t}$ its contraction. Via Gale duality, we have extracted certain collections of integer vectors
from both data in terms of rows of certain matrices $L$ and $\tilde{L}$, respectively. By Proposition
\ref{multiplicity}, the rows of $\tilde{L}$ coincide with the last $n - t$ rows of $L$ and
for $1 \leq i \leq t$, the $i$-th row $l_i$ coincides with $(1 + 2\delta_i) l_{t + \phi(i)}$.
It will be the subject of the subsequent sections to show that the vectors $\tilde{l}_1, \dots,
\tilde{l}_{n - t}$ are cyclically ordered and generate the fan associated to a complete toric
surface. This section is devoted to the first $t$ columns of $L$ and their behaviour under
mutation.

\begin{sub}\label{zerojumpsprep}
Consider an exceptional triple $\sh{E}, \sh{Z}, \sh{F}$ with $e, f \neq 0$
and $z = 0$. By moving $\sh{Z}$
to the left or right via mutating $\sh{E}$ or $\sh{F}$, respectively, we obtain exceptional triples
$\sh{Z}, R_\sh{Z}\sh{E}, \sh{F}$ and
$\sh{E}, L_\sh{Z} \sh{F}, \sh{Z}$, respectively. With \ref{mutationrank} and Lemma \ref{chirank}
we see that $s(R_\sh{Z} \sh{E}) = s(\sh{E}) + (1 + 2 \delta_\sh{Z}) c_1(\sh{Z})$ and
$s(L_\sh{Z} \sh{F}) = s(\sh{F}) - (1 + 2 \delta_\sh{Z}) c_1(\sh{Z})$. Thus we get
$s(R_\sh{Z} \sh{E}, \sh{F}) = s(\sh{E}, L_\sh{Z} \sh{F}) = s(\sh{E}, \sh{F}) - (1 + 2 \delta_\sh{Z})
c_1(\sh{Z})$. If we can extend our exceptional triple, say, to the left, i.e.
we have an exceptional sequence $\sh{D}, \sh{E}, \sh{Z}, \sh{F}$ with $d \neq 0$, then we get
furthermore that $s(\sh{D}, L_\sh{Z} \sh{F}) = s(\sh{D}, \sh{F}) + (1 + 2 \delta_\sh{Z}) c_1(\sh{Z})$.
In the following proposition we apply this simple modification of Chern classes to toric systems.
\end{sub}

\begin{proposition}\label{zerojumps}
Let $\mathbf{E} = \sh{E}_1, \dots, \sh{E}_n$ be an exceptional sequence with associated toric system
$E_1, \dots, E_t$, $A_1, \dots, A_{n - t}, \phi$.
Assume that $\sh{E}_k$ has rank zero for some $1 \leq k \leq n$ and let $1 \leq i \leq t$ such that
$E_i = c_1(\sh{E}_k)$. Consider the mutations
$L_k\mathbf{E}$ and $R_{k - 1}\mathbf{E}$. Then the corresponding toric systems are given by
$E'_1, \dots, E'_t, A_1', \dots, A_{n - t}', \phi'$ (for $L_k\mathbf{E}$) and $E''_1, \dots, E''_t,
A_1'', \dots, A_{n - t}'', \phi''$ (for $R_{k - 1}\mathbf{E}$), where
\begin{enumerate}[(i)]
\item\label{zerojumpsi} If $e_{k + 1} = 0$ (resp. $e_{k - 1} = 0$), then $E'_i = -E_{i + 1}$,
$E'_{i + 1} = E_i$ and
$E'_j = E_j$ otherwise, $A_j' = A_j$ for all $j$, and $\phi' = \phi$
(resp. $E''_{i - 1} = E_i$, $E''_i = -E_{i - 1}$ and
$E''_j = E_j$ otherwise, $A_j'' = A_j$ for all $j$ and $\phi'' = \phi$).
\item\label{zerojumpsii} If $e_{k + 1} \neq 0$ then $E'_j = E_j$ for all $j$,
$A_{\phi(i)}' = A_{\phi(i)} - (1 + 2 \delta_i) E_i$,
$A_{\phi(i) + 1}' = A_{\phi(i) + 1} + (1 + 2 \delta_i)E_i$,
and $A_j' = A_j$ otherwise. Moreover, $\phi'(j) = \phi(j)$ for $j \neq i$ and $\phi'(i) = \phi(i) + 1$.
\item\label{zerojumpsiii} If $e_{k - 1} \neq 0$ then $E''_j = E_j$ for all $j$,
$A_{\phi(i) - 1}'' = A_{\phi(i) - 1} + (1 + 2 \delta_i) E_i$, $A_{\phi(i)}'' =
A_{\phi(i)} - (1 + 2 \delta_i)E_i$, and $A_j'' = A_j$. Moreover, $\phi''(j) = \phi(j)$ for $j \neq i$
and $\phi''(i) = \phi(i) - 1$.
\end{enumerate}
\end{proposition}

\begin{proof}
(\ref{zerojumpsi}) By \ref{eulersym}  (\ref{eulersymiv}) , we have for any exceptional pair
$\sh{E}, \sh{F}$ with $e = f = 0$ that 
$[L_\sh{E}\sh{F}] = -[\sh{F}]$ and $[R_\sh{F}\sh{E}] = -[\sh{E}]$ in $\knum(X)$, so that
on the level of toric systems $E_i$ and $E_{i + 1}$ (respectively $E_{i - 1}$ and $E_i$)
get replaced by $-E_{i + 1}$ and $E_i$ (respectively $E_i$ and $-E_{i - 1}$).
The assertion correspondingly just reflects the reshuffling of data.

(\ref{zerojumpsii}) We have already seen in \ref{zerojumpsprep} that the $A_j$ behave in the
described way (in particular, the $A_j$ for $j \neq \phi(i), \phi(i + 1)$ remain the same).
Also, as the sequence of $E_j$'s remains constant, the function $\phi(i)$ changes as described.

(\ref{zerojumpsiii}) follows completely analogously to (\ref{zerojumpsii}).
\end{proof}

By Proposition \ref{multiplicity}, the translation of above modifications into
the Gale dual picture can be described as follows.

\begin{corollary}\label{zerojumpcorollary}
With the notation of Proposition \ref{zerojumps} we denote $L'$ and $L''$ (respectively $\tilde{L}'$
and $\tilde{L}''$) the Gale dual matrices corresponding to $L_k\mathbf{E}$ and $R_{k - 1} \mathbf{E}$,
respectively. Then:
\begin{enumerate}[(i)]
\item\label{zerojumpcorollaryi} $\tilde{L}' = \tilde{L}'' = \tilde{L}$.
\item\label{zerojumpcorollaryii}
If $e_{k + 1} = 0$ (resp. $e_{k - 1} = 0)$ then $l'_i = -l_{i + 1} =
-(1 + 2\delta_{i + 1})l_{t + \phi(i + 1)}$, $l'_{i + 1} = l_i$
(resp. $l''_{i - 1} = l_i$ and $l'_i = -(1 + 2 \delta_{i - 1})l_{i - 1}$), and $l_j = l_j'$ otherwise.
\item\label{zerojumpcorollaryiii}
If $e_{k + 1} \neq 0$ then $l'_i = -(1 + 2 \delta_i) l_{i + \phi'(i)}$ and $l'_j = l_j$ for all
$j \neq i$.
\item\label{zerojumpcorollaryiv}
If $e_{k - 1} \neq 0$ then $l'_i = -(1 + 2 \delta_i) l_{i + \phi''(i)}$ and $l''_j = l_j$ for all
$j \neq i$.
\end{enumerate}
\end{corollary}

\begin{sub}\label{zerojumpremark}
Whenever we have an exceptional sequence of maximal length which may contain objects of rank zero,
it follows from our discussion in paragraph \ref{spread} that it as well contains objects of nonzero rank
wich we can use to mutate this sequence into a sequence which does not contain any object of rank zero.
Conversely, if we start with a sequence which does not contain any rank zero objects, we will always
have to take into account the possibility that a series of mutations will eventually (or
intermediately) create a rank zero object. Therefore it will be important in the subsequent
sections that we are able to handle such objects flexibly. The statements in this section
imply that we have such a flexibility where we can ``move around'' rank zero objects via mutation
without essentially changing the combinatorial data associated to it. To exemplify this, consider
the exceptional sequence $\sh{O}_E(E), b^*\sh{O}(2), b^*\sh{O}(3), b^*\sh{O}(4)$ from the introduction.
In suitable coordinates, one computes the matrices as:
\begin{equation*}
L = 
\begin{pmatrix}
l_1\\ l_2 \\ l_3 \\ l_4
\end{pmatrix}
=
\begin{pmatrix}
1 & 0 \\
-1 & 1 \\
0 & -1 \\
1 & 0
\end{pmatrix}
\qquad \text{ and } \qquad \tilde{L} = 
\begin{pmatrix}
\tilde{l}_1\\ \tilde{l}_2 \\ \tilde{l}_3
\end{pmatrix}
=
\begin{pmatrix}
-1 & 1 \\
0 & -1 \\
1 & 0
\end{pmatrix}.
\end{equation*}
The corresponding toric system is given by data $E_1 = E, A_1, A_2, A_3, \phi$, with $\phi(1) = 3$.
In anticipation of Theorem \ref{maintheorem2}, Figure \ref{p2pic5} depicts the $\tilde{l}_i$ as primitive generators
for the fan of $\mathbb{P}^2$, where the presence of a second copy of $\tilde{l}_3$ in $L$ is indicated by a double arrow.
Here we express the fact that the exceptional sequence contains an object of rank zero at its
leftmost position by attaching the multiplicity $2$ to $l_3$.
\begin{figure}[ht]
\centering
\begin{minipage}{.35\linewidth}
\centering
\begin{tikzpicture}
\draw[style=help lines,dashed,very thin] (-1.3,-1.3) grid[step=.7cm] (1.3, 1.3);

\draw[->, thick, double] (0, 0) -- (.7, 0);
\draw[->, thick] (0, 0) -- (0, -.7);
\draw[->, thick] (0, 0) -- (-.7, .7);

\draw (1, -.3) node { $\tilde{l}_3 = l_1 = l_4$ };
\draw (-.25, -1) node { $\tilde{l}_2$ };
\draw (-.9, .9) node { $\tilde{l}_1$ };
\end{tikzpicture}
\caption{The fan for
$\sh{O}_E(E),$ $b^*\sh{O}(2), b^*\sh{O}(3), b^*\sh{O}(4)$.}\label{p2pic5}
\end{minipage}
\begin{minipage}{.2\linewidth}
\end{minipage}
\hspace{2cm}
\begin{minipage}{.42\linewidth}
\centering
\begin{tikzpicture}
\draw[style=help lines,dashed,very thin] (-1.3,-1.3) grid[step=.7cm] (1.3, 1.3);

\draw[->, thick] (0, 0) -- (.7, 0);
\draw[->, thick] (0, 0) -- (0, -.7);
\draw[->, thick, double] (0, 0) -- (-.7, .7);

\draw (1, -.3) node { $\tilde{l}_3$ };
\draw (-.25, -1) node { $\tilde{l}_2$ };
\draw (-.9, .9) node { $\tilde{l}_1 = \hat{l}_1 = \hat{l}_2$ };
\end{tikzpicture}
\caption{The fan for
$L_{\sh{O}_E(E)} b^*\sh{O}(2),$ $\sh{O}_E(E)[1], b^*\sh{O}(3), b^*\sh{O}(4)$.}\label{p2pic6}
\end{minipage}
\end{figure}
Now, if we move $\sh{O}_E(E)$ to the right by mutation, we obtain the sequence
$L_{\sh{O}_E(E)} b^*\sh{O}(2), \sh{O}_E(E), b^*\sh{O}(3), b^*\sh{O}(4)$ and, by applying
a shift to $\sh{O}_E(E)$, we get $L_{\sh{O}_E(E)} b^*\sh{O}(2), \sh{O}_E(E)[1], b^*\sh{O}(3), b^*\sh{O}(4)$.
If we denote the corresponding matrices by $L'$ and $\hat{L}$, respectively, we get:
\begin{equation*}
(L')^T =
\begin{pmatrix}
1 & -1 \\
-1 & 1 \\
0 & -1 \\
1 & 0
\end{pmatrix}
\qquad \text{ and }  \qquad
\hat{L}^T =
\begin{pmatrix}
\hat{l}_1\\ \hat{l}_2 \\ \hat{l}_3 \\ \hat{l}_4
\end{pmatrix}
=
\begin{pmatrix}
-1 & 1 \\
-1 & 1 \\
0 & -1 \\
1 & 0
\end{pmatrix},
\end{equation*}
where the first equality follows from Corollary \ref{zerojumpcorollary} and the application of the shift
functor translates to the second equality. Moreover, we have $\tilde{L}' = \tilde{\hat{L}} = \tilde{L}$,
so, up to multiplicities, the fan we associate to each of the three cases is the same. Figure \ref{p2pic6}
shows the effect that the multiplicity $2$ shifts counterclockwise from $\tilde{l}_3$ to $\tilde{l}_1$.
We may refer to this effect colloquially as ``hopping'' of multiplicities.

In general, if we apply, say, $L_k\mathbf{E}$ to $\mathbf{E} = \dots, \sh{E}_{k - 1}, \sh{E}_k, \sh{E}_{k + 1}, \dots$
with $e_k = 0$ and $e_{k + 1} \neq 0$, and we get $\dots, \sh{E}_{k - 1}, L_{\sh{E}_k} \sh{E}_{k + 1},
\sh{E}_k, \dots$, then by Corollary \ref{zerojumpcorollary} (\ref{zerojumpcorollaryi}) the matrices
$\tilde{L}$ and $\tilde{L}'$ (equivalently, the respective last $n - t$ rows of $L$ and $L'$) coincide.
The only difference between $L$ and $L'$ then is that the $i$-th row (where $1 \leq i \leq t$ such
that $E_i = c_1(\sh{E}_k)$) flips from $(1 + 2\delta_i)\tilde{l}_{\phi(i)}$ to $-(1 + 2 \delta_i)
\tilde{l}_{\phi(i) + 1}$. In other terms, one could think of the $\tilde{l}_j$ to be endowed with
``multiplicities'', i.e. if $A_j = s(\sh{E}_k, \sh{E}_{k'})$ for $k < k'$ and $e_k, e_{k'} \neq 0$
then $|\phi^{-1}(j)|$ is the number of rank zero objects in the exceptional sequence between
$\sh{E}_k$ and $\sh{E}_{k'}$ and correspondingly, $L$ contains as many many additional rows which are
collinear to $\tilde{l}_j = l_{t + j}$. If we disregard the factors $(1 + 2 \delta_i)$, we
should think of these additional rows as copies of $l_{t + j}$. The latter will be fully justified
once we have proven Theorem \ref{deltatheorem} which implies $1 + 2\delta_i = \pm 1$. Then by Corollary
\ref{deltaz} there will be no loss of generality to assume that $1 + 2\delta_i = 1$ for every $i$.
Ultimately, we can view the vectors $\tilde{l}_1, \dots, \tilde{l}_{n - t}$ as the essential
data associated to a toric system of an exceptional sequence, where every $\tilde{l}_i$
comes with a multiplicity for bookkeeping of the rank zero objects in the sequence. In particular,
by Corollary \ref{zerojumpcorollary}, moving rank zero objects around via mutations then is
reflected simply by a ``hopping'' of the corresponding multiplicities. So, as far as the
$\tilde{l}_i$ are concerned, there will in many situations be no harm to pretend that they have
multiplicity one. E.g. we can without loss of generality assume that $\phi$ is constant so
that the rank zero objects form an uninterrupted
sub-sequence anywhere in the sequence.
\end{sub}

\section{Local constellations}\label{localconstellations}

In order to simplify notation and to reduce the number of trivial caveats, in this section we will
make the assumption that  our exceptional sequence $\mathbf{E} = \sh{E}_1, \dots, \sh{E}_n$ contains
no objects of rank zero (and hence $L = \tilde{L}$). In the spirit
of \ref{zerojumpremark}, ``up to multiplicities'' the results extend trivially to the general case.

\begin{sub}\label{circuitintro}
We denote $a_i := e_i^2 e_{i + 1}^2 A_i^2 \in \Z$.
The intersection product in $\chnum^1(X)$ induces a bilinear form on $A$. For any $A_i$, we denote
$A_i^\bot$ the orthogonal complement of $A_i$ in $A$ with respect to this form. Clearly, $A_i^\bot$
contains all $A_j$ for $j \notin \{i - 1, i, i + 1\}$ and therefore the quotient
$A / A_i^\bot$ is isomorphic to $\Z$ and $A / \langle A_j \in A_i^\bot \rangle_\Z \simeq \Z
\oplus F_i$ where $F_i \simeq A_i^\bot / \langle A_j \in A_i^\bot \rangle_\Z$. As we have seen in
Section \ref{galedualsection} (see in particular the proof of Proposition \ref{toricsystemrank}),
the set $\{A_j \mid j \neq i, i + 1\}$ is linearly independent in $A$, hence the set
$\langle A_j \mid j \neq i - 1, i, i + 1\rangle \subseteq A_i^\bot$ has finite index in $A_i^\bot$
and thus $F_i$ is finite. Moreover, by the general properties of Gale duality, this implies that
$l_i$ and $l_{i + 1}$ are linearly independent in $N$.
So we have the following exact commutative diagram:
$$
\xymatrix{
0 \ar[r] & M \ar[rr]^L \ar@{=}[d] && \Z^n \ar[rr]^c \ar@{>>}[d] && A \ar[r] \ar@{>>}[d]  & 0\\
0 \ar[r] & M \ar[rr]^{L_{i - 1, i, i + 1}} && \Z^3 \ar[rr]^{\bar{c}_i} && \Z \oplus F_i \ar[r] & 0.
}
$$
Dualizing the lower row, we get the exact sequence:
$$
\xymatrix{
0 \ar[r] &  \Z \ar^{\bar{c}^T_i}[rr] & & \Z^3 \ar^{L^T_{i - 1, i, i + 1}}[rr] & &
N \ar[r] & F_i \ar[r] & 0
}
$$
where by slight abuse of notation we denote $\bar{c}^T_i$ the dual of $\bar{c}_i$ and we
identify $\Ext^1(\Z \oplus F_i, \Z) \simeq F_i$.
An elementary computation shows
$$ \mathbf{d}_i := (\det(l_i, l_{i + 1}), \det(l_{i + 1}, l_{i - 1}), \det(l_{i - 1}, l_i))^T
\in \ker(L^T_{i - 1, i, i + 1}).$$
and
$\ker(L^T_{i - 1, i, i + 1})$ is generated by $\frac{1}{g_i}\mathbf{d}_i$, where
$g_i = \gcd\{\det(l_i, l_{i + 1}), \det(l_{i + 1}, l_{i - 1}), \det(l_{i - 1}, l_i)\}$.
We leave it as an exercise for the reader to show that $L^T_{i - 1, i, i + 1}$ is surjective iff
$g_i = 1$. Applying this exercise to the exact sequence
$$0 \longrightarrow \Z \xrightarrow{\ \ \ \bar{c}_i^T \ \ \ } \Z^3 \xrightarrow{L^T_{i - 1, i, i + 1}} \im L^T_{i - 1, i, i + 1} \longrightarrow 0$$
we obtain by composition
with the inclusion $\im L^T_{i - 1, i, i + 1} \hookrightarrow N$ that $$g_i = |F_i|.$$
Now we want to determine $\mathbf{d}_i$.
\end{sub}

\begin{proposition}\label{nonflatcircuit}
Up to a choice of orientation in $N$, we have $\mathbf{d}_i = h(e_{i + 1}^2, a_i, e_i^2)$ for every $i$ and some $h \in \Z$ with $h > 0$.
In particular,
$$
\det(l_{i - 1}, l_i) = he_i^2 \quad \text{ and } \quad \det(l_{i + 1}, l_{i - 1}) = h a_i
$$
for every $i$.
\end{proposition}

\begin{proof}
We first fix some $i$.
The projection $A \rightarrow \Z \oplus F_i$ induces a $\Q$-valued linear form on $\Z \oplus F_i$
which is given by $( A_i \cdot - )$ and which vanishes on $F_i$.
We project further and obtain a $\Q$-valued linear form on $\Z$ which is completely determined
by a number $q \in \Q$ such that $A_i \cdot A_j = q \bar{A}_j$ for $j = i - 1, i, i + 1$,
where we denote $\bar{A}_j$ the image of $A_j$ in $\Z$. In particular, we have $q \bar{A}_{i - 1} = 1 / e_i^2$,
$q \bar{A}_i = a_i / e_i^2 e_{i + 1}^2$, and $q \bar{A}_{i + 1} = 1 / e_{i + 1}^2$ that we
can simply interpret as equations of rational numbers.

Now we set  $q =: x / e_i^2 e_{i + 1}^2$ for some $x \in \Q$ and obtain by rearranging the equations:
$$\bar{A}_{i - 1} = e_{i + 1}^2/x, \quad \bar{A}_i = a_i/x, \quad \bar{A}_{i + 1} = e_i^2/x.$$
As the $\bar{A}_j$'s are integral and generate $\Z$, we get $x = \pm g_i'$, where
$g_i' := \gcd\{e_{i + 1}^2, a_i, e_i^2\}$.
Up to the choice of a generator of $A / A_i^\bot \simeq \Z$, we can assume $x = g_i'$. By construction, the short
exact sequence corresponding to the surjection $\Z^3 \twoheadrightarrow \Z = \langle \bar{A}_{i - 1}, \bar{A}_i, \bar{A}_{i + 1} \rangle$
is dual to the short exact sequence $\Z \rightarrowtail \Z^3 \twoheadrightarrow \im L^T_{i - 1, i, i + 1}$ from  \ref{circuitintro},
hence from the discussion there it follows that $\bar{c}^T_i = \frac{1}{g_i'}(e_{i + 1}^2, a_i, e_i^2)^T$.
Morover, it follows that $\mathbf{d}_i = h_i(e_{i + 1}^2, a_i, e_i^2)$, where we
denote $h_i := g_i / g_i' = \det(l_{i - 1}, l_i) / e_i^2 = \det(l_i, l_{i + 1}) / e_{i + 1}^2 \in \Q$.

This implies $h_i = h_{i + 1}$ for every $i$, hence by induction there
exists $h \in \Q$ such that $h_i = h$ for every $i$. Moreover, because for an exceptional sequence of maximal
length the gcd of the $e_i$ is $1$, it follows that $h \in \Z$.
\end{proof}

\begin{corollary}\label{localfan}
\begin{enumerate}[(i)]
\item\label{localfani} We can choose an orientation of $N$ such that for every $i$ the pair $l_i, l_{i + 1}$ is positively oriented,
i.e. $\det(l_i, l_{i + 1}) > 0$.
\item\label{localfanii} For every $i$, the pair  $l_i, l_{i + 1}$ generates a strictly convex polyhedral cone in $N_\Q$.
\item\label{localfaniii} Every triple $l_{i - 1}, l_i, l_{i + 1}$ satisfies the relation $$e_{i + 1}^2 l_{i - 1} + a_i l_i + e_i^2 l_{i + 1} = 0$$
and generates a fan which contains two 2-dimensional cones which intersect in the common facet $\Q_{\geq 0} l_i$.
\end{enumerate}
\end{corollary}

\begin{proof}
(\ref{localfani}) Obvious. We will assume the choice of this orientation for the rest of the proof.

(\ref{localfanii}) Follows from $\det(l_i, l_{i + 1}) = h e_i^2 > 0$ for every $i$.

(\ref{localfaniii}) Because  $\det(l_i, l_{i + 1}) > 0$ and $\det(l_{i - 1}, l_i) > 0$,
the lattice vectors $l_{i - 1}$ and $l_{i + 1}$ lie in the opposite interiors of the half spaces which are bounded by the line
$\Q l_i$. Hence the cones generated by $l_{i - 1}, l_i$ and $l_i, l_{i + 1}$, respectively,
intersect at the common facet $\Q_{\geq 0} l_i$ and therefore form a fan.
\end{proof}

Note that for the rest of the paper we will always implicitly assume that we have chosen an orientation of $N$
which conforms to Corollary \ref{localfan} (\ref{localfani}) and will mention it no further.

The following lemma shows that the case $a_i = 0$ corresponds to a very special configuration.

\begin{lemma}\label{opposite1}
If $a_i = 0$ then $e_i^2 = e_{i + 1}^2$ and $l_{i - 1} = -l_{i + 1}$.
\end{lemma}

\begin{proof}
By \ref{eulersym} (\ref{eulersymii}) we have $\chi(\sh{E}_i, \sh{E}_{i + 1}) =
\frac{e_i^2 + e_{i + 1}^2}{e_i e_{i + 1}}$. If we denote $g := \gcd\{e_i, e_{i + 1}\}$ and
$e'_i := e_i / g$, $e'_{i + 1} := e_{i + 1} / g$, then  we get immediately
$\chi(\sh{E}_i, \sh{E}_{i + 1}) = \frac{(e'_i)^2 + (e'_{i + 1})^2}{e'_i e'_{i + 1}}$.
But then $e'_i$ divides $(e'_{i + 1})^2$ and $e'_{i + 1}$ divides $(e'_i)^2$. But because
$\gcd\{e'_i, e'_{i + 1}\} = 1$, this implies $(e'_i)^2 = (e'_{i + 1})^2 = 1$,
hence $e_i^2 = e_{i + 1}^2 = g^2$.
For the second assertion, observe that the relation $e_{i + 1}^2 l_{i - 1} + e_i^2 l_{i + 1} = 0$ holds.
\end{proof}

We want now describe what happens to the vectors $l_i$ if we perform a mutation of the sequence
$\mathbf{E}$. If we apply a mutation $L_i \mathbf{E}$ or $R_i \mathbf{E}$, then we can construct
by above procedure a sequence of vectors $l'_1, \dots, l'_n \in N'$, where $N'$ is the dual of
the kernel of
the structural morphism $c'$ corresponding to the new toric system. The following lemma shows
that in terms of the $l_i$, the effect of mutation is local.

\begin{proposition}\label{mutationlocal}
With above notation, we can naturally identify $N$ with $N'$ such that $l'_j = l_j$ for all $j \neq i$.
\end{proposition}

\begin{proof}
Without loss of generality, we only consider left mutations $L_i \mathbf{E}$;
the case of right mutations then follows analogously. On the level of toric systems, the effects of
such a mutation
can be described by the formulas of \ref{additivity} and \ref{mutlemma}. For this, we distinguish
two cases, depending on whether $R := \rk L_{\sh{E}_i} \sh{E}_{i + 1}$ is nonzero or not.

In the first case, we obtain a toric system $A_1', \dots, A_n'$, where $A_{i - 1}' = A_{i - 1} -
\frac{e_{i + 1}}{R} A_i$, $A_i' = \frac{e_{i + 1}}{R} A_i$, $A_{i + 1}' = A_i + A_{i + 1}$,
and $A_j' = A_j$ otherwise. In particular, $\langle A'_1, \dots, A'_n \rangle_\Z$ contains
all $A_j$ with $j \neq i, i + 1$. By Proposition \ref{toricsystemrank},
these $A_j$ are linearly independent and therefore form a $\Q$-basis of $\Ch^1(X)_\Q$. By rescaling
these basis vectors by a factor $\frac{1}{e_{i + 1}^2}$, we can represent the $\Q$-linear extension
of the maps $c$ and $c'$ in \ref{dualitydef} by the following matrices:
\begin{equation*}
c =
\begin{pmatrix}
e_3^2 & a_2 & e_2^2 & 0 & \cdots & 0 \\
0 & x_2 & y_2 & e_3^2 & & 0 \\
\vdots & \vdots & \vdots & & \ddots \\
0 & x_{n - 2} & y_{n - 2} & 0 & & e_3^2
\end{pmatrix}
\text{ and }
c' =
\begin{pmatrix}
e_3^2 - \frac{e_3}{R} a_2 & \frac{e_3}{R} a_2 & a_2 + e_2^2 & 0 & \cdots & 0 \\
-\frac{e_3}{R} x_2  & \frac{e_3}{R} x_2 & x_2 + y_2 & e_3^2 & & 0 \\
\vdots & \vdots & \vdots & & \ddots \\
-\frac{e_3}{R} x_{n - 2} & \frac{e_3}{R} x_{n - 2} & x_{n - 2} + y_{n - 2} & 0 & & e_3^2
\end{pmatrix}
\end{equation*}
where by cyclic renumbering we assume without loss of generality that $i = 2$. The
$x_j, y_j$ are determined by the relations $x_j l_2 + y_j l_3 + e_3^2 l_{j + 2} = 0$, in particular
we have $h x_j = \det(l_3, l_{j + 2})$ and $h y_j = \det(l_{j + 1}, l_2)$ for every $j$. Consider first
the case $a_2 \neq 0$. Then $l'_2 = \frac{-1}{a_2}(e_2^2 l'_1 + R^2 l'_3)$. Now it
follows from a direct calculation that we can represent the Gale transforms $l'_1, \dots, l'_n$
by $l_1, l'_2, l_3, \dots, l_n$, in particular, we have $l'_2 = \frac{-1}{a_2}(e_2^2 l_1 + R^2 l_3)$.
In the case $a_2 = 0$, we use any row with $x_j \neq 0$ in order to find the representation
$l'_2 = l_2 + 2 l_1$. As before, we check that we can represent $l'_1, \dots, l'_n$ by
$l_1, l_2', l_3, \dots, l_n$.

In the second case, our toric system is given by $E_1, A_1', \dots, A_{n - 1}', \phi$, where
$A_j' = A_j$ for $j < i$, $A'_i = A_i + A_{i + 1}$ $A_j' = A_{j - 1}$ for $j > i$, and
$\phi(1) = i - 1$. By similar arguments as in the proof of Proposition \ref{multiplicity}, we can
conclude that the effect of the mutation is that the vector $l_i$ ``hops'' onto $l_{i - 1}$.
\end{proof}

The following statement shows that we can get rid of the factor $h$ in Proposition \ref{nonflatcircuit}.

\begin{proposition}\label{primitivity}
\begin{enumerate}[(i)]
\item\label{primitivityi} Let $h$ be as in Proposition \ref{nonflatcircuit}. Then $h = 1$ and thus
$$
\det(l_{i - 1}, l_i) = e_i^2 \quad \text{ and } \quad \det(l_{i + 1}, l_{i - 1}) = a_i
$$
for every $i$.
\item\label{primitivityii} The $l_i$ are primitive lattice vectors.
\end{enumerate}
\end{proposition}

\begin{proof}
(\ref{primitivityi})
In Proposition \ref{nonflatcircuit} we have shown that $h$ divides $\det(l_i, l_{i + 1})$ and $\det(l_{i - 1}, l_{i + 1})$ for every $i$.
We will show that $h \neq 1$ implies that the $l_i$ generate a proper sublattice of $N$, which is a contradiction.
Obviously, for $n \leq 5$ we immediately have that $h$ divides $\det(l_i, l_j)$ for every $i \neq j$, hence
there is nothing to show. Soo without loss of generality, we can assume $n > 5$.
We start with the following claim: for every $i < j$
with $2 < j - i < n$, $h$ divides $\det(l_i, l_j) \cdot e_{i + 2}^2 \cdots e_{j - 1}^2$. To see this, we start with the
equality
$$
e_j^2 l_{j - 2} + a_{j - 1} l_{j - 1} + e_{j - 1}^2 l_j = 0.
$$
Applying $\det(l_i, -)$ to this equation, we get:
$$
e_j^2 \det(l_i, l_{j - 2}) + a_{j - 1} \det(l_i, l_{j - 1}) + e_{j - 1}^2 \det(l_i, l_j) = 0.
$$
Now the claim follows by induction starting with $j = i + 3$. As a corollary of this claim we conclude that
for any $i, j$ with $| i - j | > 2$, $h$ divides
$$
\det(l_i, l_j) \cdot \gcd\{e_{i + 1}^2 \cdots e_{j - 1}^2, e_{j + 2}^2 \cdots e_{i - 1}^2\}.
$$
Now assume $p$ is a prime factor of $h$, which is not a prime factor of any of the $e_i$. Then above
equation shows that $p$ divides $\det(l_i, l_j)$ for every $i \neq j$, hence the sublattice generated
by the $l_i$ in $N$ has an index which is a multiple of $p$, which is absurd, as the $l_i$ generate $N$.
Hence, such a prime factor cannot exist and thus every prime factor of $h$ must show up as a
prime factor of at least one of the $e_i$. To complete our argument, it suffices to show that for a prime
factor $p$ of $h$, $p$ divides $\det(l_i, l_j)$ for any pair $i \neq j$. As before, this is absurd, hence
such a prime factor cannot exist.

We now observe that by Proposition \ref{mutationlocal}, $h$ remains invariant under mutation.
There exists always at least one $e_i$ which is not divisible by $p$ and, by renumbering
we can always arrange that $p$ does not divide $e_3, \dots, e_k$ for some $k > 2$.
Then it follows from our claim above that $p$ divides $\det(l_i, l_j)$ for all $1 \leq i < j \leq k + 1$. If
$k > n - 3$, we are done. Otherwise, we will show that, possibly after mutation, $p$ divides
$\det(l_i, l_j)$ for all $1 \leq i < j \leq k + 2$. The statement is trivial if $p$ does not divide $e_{k + 1}$.
If $p$ is a divisor of $e_{k + 1}$, we perform a right mutation of the pair $\sh{E}_k, \sh{E}_{k + 1}$, resulting
in the pair $\sh{E}_{k + 1}, R_{\sh{E}_{k + 1}} \sh{E}_k$ and moving $l_k$ to $l_k'$ by Proposition \ref{mutationlocal}. Using
our claim on $l_1, \dots, l_{k - 1}, l_k'$ and the properties stated in the beginning of the proof, we see that
$p$ divides $\det(l_i, l_k')$ for every $1 \leq i \leq k + 1$, $i \neq k$, hence $l_1, \dots, l_{k - 1}, l_k', l_{k + 1}$ still
generate a sublattice of index divisible by $p$ in $N$. Now, because $\rk R_{\sh{E}_{k + 1}} \sh{E}_k =
\chi(\sh{E}_k, \sh{E}_{k + 1}) e_{k + 1} - e_k$, where $p$ divides $e_{k +1 }$ and does not divide $e_k$,
it follows that $p$ does not divide $\rk R_{\sh{E}_{k + 1}} \sh{E}_k$. Hence, as we did in the induction step
for above claim, we conclude that $p$ divides $\det(l_1, l_{k + 2})$. Similarly, we conclude that
$p$ divides $\det(l_j, l_{k + 2})$ for $1 < j < k + 2$. By induction it follows that $p$ divides $\det(l_i, l_j)$
for every $i \neq j$, which implies that the $l_i$ cannot generate $N$, which is a contradiction.
Therefore $h$ has no prime factors, i.e. $h = 1$, which completes our proof.

(\ref{primitivityii})
Assume there is one $l_i$ which is not primitive. Without loss of generality, we cyclically renumber
the sequence such that $i = 1$. Then $l_1 = p \tilde{l}_1$ for
some $p > 1$ and $\tilde{l}_1$ is a primitive lattice vector. Then $p$ divides both $e_1^2 =
\det(l_n l_1)$ and $e_2^2 = \det(l_1 l_2)$. Now for every $3 < i < n$ with $e_i \neq 0$, we can
perform right-mutations in order to move $\sh{E}_i$ to the left:
\begin{equation*}
\sh{E}_1, \sh{E}_i, R_{\sh{E}_i} \sh{E}_2, \dots, R_{\sh{E}_i} \sh{E}_{i - 1}, \sh{E}_{i + 1},
\dots, \sh{E}_n.
\end{equation*}
As these mutations do not alter $l_n$ and $l_1$, the exceptional pair $\sh{E}_1, \sh{E}_i$
corresponds to rays $l_n, l_1, l_2'$ and $e_i^2 = \det(l_1, l_2')$. Therefore $p$ divides
$e_i^2$ as well, and hence we get $\gcd\{e_1^2, \dots, e_n^2\} \neq 1$ and thus $\gcd\{e_1, \dots,
e_n\} \neq 1$. But this contradicts the fact that the rank morphism from $\knum(X)$ to $\Z$ is surjective.
\end{proof}

Another special configuration arises if $L_{\sh{E}_i} \sh{E}_{i + 1}$ or $R_{\sh{E}_{i + 1}} \sh{E}_i$ has rank zero.
Recall that in each case we have defect terms $\delta$ and $\delta'$, respectively, in the sense of Definition \ref{defectdef}.

\begin{lemma}\label{zeromutationconfiguration}
If $\rk L_{\sh{E}_i} \sh{E}_{i + 1} = 0$ then
$e_i^4 (1 + 2\delta)^2 l_{i - 1} - e_i^2 l_i + e_i^2 l_{i + 1} = 0$.
If $\rk R_{\sh{E}_{i + 1}} \sh{E}_i = 0$ then
$e_i^2 (1 + 2\delta')^2 l_{i - 1} - e_i^2 l_i + e_i^4 l_{i + 1} = 0$.
\end{lemma}

Note that once Theorem \ref{deltatheorem} is proved we can assume
$(1 + 2 \delta)^2 = 1$ and $(1 + 2 \delta')^2 = 1$, respectively.

\begin{proof}
We only prove the first equation. Observe that $0 = \rk L_{\sh{E}_i} \sh{E}_{i + 1} =
\chi(\sh{E}_i, \sh{E}_{i + 1}) e_i - e_{i + 1}$ by \ref{mutationrank} and
$\chi(\sh{E}_i, \sh{E}_{i + 1}) = \chi(L_{\sh{E}_i} \sh{E}_{i + 1}, \sh{E}_i) = -e_i(1 + 2\delta)$ by
Proposition \ref{chirank} (\ref{chirankiii}), hence $e_{i + 1} = -e_i^2(1 + 2\delta)$ and
with $a_i = \det(l_{i + 1}, l_{i - 1}) = -e_i^2$ the statement follows.
\end{proof}

\begin{remark}
As remarked at the beginning of this section, in order to simplify the presentation we considered only
the case where our exceptional sequence contains only objects of nonzero rank. Given an arbitrary
sequence of maximal length, we can always produce such a sequence by mutation. More precisely, by \ref{zerojumpremark},
we can assume that $e_i = 0$ for $1 \leq i \leq t$ for some $0 \leq t \leq n - 3$ and
$e_i \neq 0$ for $t < i \leq n$. Then by right mutation, we can produce a sequence
\begin{equation*}
\sh{E}_{t + 1}, R_{\sh{E}_{t + 1}} \sh{E}_1, \dots, R_{\sh{E}_{t + 1}} \sh{E}_t, \sh{E}_{t + 2},
\dots, \sh{E}_n
\end{equation*}
with $\rk R_{\sh{E}_{t + 1}} \sh{E}_i = -e_{t + 1}^2 (1 + \delta) \neq 0$ for $1 \leq i \leq t$. The local
configuration of the new $l_i$ then arises iteratively from Lemma \ref{zeromutationconfiguration}.
With Theorem \ref{deltatheorem}, in the case $e_i = \pm 1$ this corresponds (at least locally, so far)
to a series of smooth toric blow-ups.
\end{remark}

For any exceptional pair $\sh{E}_0, \sh{E}_1$, we define inductively
$\sh{E}_{i + 2} = R_{\sh{E}_{i + 1}} \sh{E}_i$ for $i \geq 0$. The Chern classes $s(\sh{E}_i, \sh{E}_{i + 1})$
are all collinear to $c_1(\sh{E}_0, \sh{E}_1)$ in $\Ch^1(X)_\Q$, where the proportionality
is successively given by the quotients of ranks $e_i / e_{i + 1}$. These ranks are determined
in the following
proposition which generalizes a similar statement by Rudakov \cite[\S 4]{Rudakov89}. It is not needed
in the remainder of this paper but it might be of some general interest.

\begin{proposition}\label{rankgeneration}
Let $\sh{E}_0, \sh{E}_1$ be an exceptional pair with $\chi(\sh{E}_0, \sh{E}_1)^2 \neq 4$ and
for $i \geq 0$ define inductively
$\sh{E}_{i + 2} = R_{\sh{E}_{i + 1}} \sh{E}_i$.
Moreover, denote $\alpha_\pm = \frac{1}{2}(\chi(\sh{E}_0, \sh{E}_1) \pm
\sqrt{\chi(\sh{E}_0, \sh{E}_1)^2 - 4})$ the roots of the polynomial
$x^2 - \chi(\sh{E}_0, \sh{E}_1) x + 1$. Then for $i \geq 2$ we get:
\begin{equation*}
e_i =
\frac{\alpha_+^{i + 1} - \alpha_-^{i + 1}}{\alpha_+ - \alpha_-} e_0 -
\frac{\alpha_+^i - \alpha_-^i}{\alpha_+ - \alpha_-} (\chi(\sh{E}_0, \sh{E}_1) e_0 - e_1).
\end{equation*}
\end{proposition}

\begin{proof}
Follows from standard arguments for solving the reccurence relation
\begin{equation*}
e_{i + 2} = \chi(\sh{E}_0, \sh{E}_1) e_{i + 1} - e_i.
\end{equation*}
for $i \geq 0$.
\end{proof}

\begin{remark}\label{rankgeneration2}
For the two cases with $\chi(\sh{E}_0, \sh{E}_1)^2 = 4$ we can directly use the first
equation in the proof of Proposition \ref{rankgeneration} end get by induction:
\begin{align*}
e_i & = i e_1 + e_0 (1 - i) & \text{ if } \chi(\sh{E}_0, \sh{E}_1) = 2,\\
e_i & = (-1)^{i + 1} (i e_1 - e_0 (1 - i)) & \text{ if } \chi(\sh{E}_0, \sh{E}_1) = -2.
\end{align*}
Moreover, note the periodic behaviour for the cases $\chi(\sh{E}_0, \sh{E}_1)^2 \leq 1$:
\begin{enumerate}[(i)]
\item If $\chi(\sh{E}_0, \sh{E}_1) = 0$ then $e_i =
\begin{cases}
(-1)^{i/2} e_0 & \text{ for $i$ even},\\
(-1)^{(i - 1)/2} e_1 & \text{ for $i$ odd.}
\end{cases}
$
\item If $\chi(\sh{E}_0, \sh{E}_1)^2 = 1$ then $e_i =
\begin{cases}
(-\chi(\sh{E}_0, \sh{E}_1))^{i/3} e_0 & \text{ for } i \equiv 0 (3), \\
(-\chi(\sh{E}_0, \sh{E}_1))^{(i - 1) / 3} e_1 & \text{ for } i \equiv 1 (3), \\
(-\chi(\sh{E}_0, \sh{E}_1))^{(i - 2) / 3} (\chi(\sh{E}_0, \sh{E}_1) e_1 - e_0) & \text{ for } i \equiv 2 (3).
\end{cases}
$
\end{enumerate}
\end{remark}

\section{Mutations}\label{mutationsection}

In this section we consider an exceptional pair $\sh{E}, \sh{F}$ with $e, f \neq 0$ and $a :=
c_1(\sh{E}, \sh{F})^2$. We assume that this pair can be extended to an exceptional sequence of
length $n$, $\sh{E}, \sh{F}, \sh{E}_3, \dots, \sh{E}_n$. Then the set of Gale duals of the associated
toric system contains primitive lattice vectors  $l_e, l, l_f \in N$ such that the following
relation holds:
\begin{equation*}
f^2 l_e + a l + e^2 l_f = 0.
\end{equation*}
Note that by Proposition \ref{primitivity} and Lemma \ref{triplelemma}, $l_e, l, l_f$ are essentially
uniquely determined by the integers $e^2, a, f^2$.

\begin{sub}\label{circumferenceformulas}
As in Definition \ref{circumferencesegmentdef}, we have circumference segments $p_e := l - l_e$, $p_f :=
l_f - l$ and it is convenient to define
\begin{equation*}
w_e := \frac{1}{e} p_e, \quad \text{ and } \quad w_f := \frac{1}{f} p_f.
\end{equation*}
Using \ref{eulersym} and \ref{mutlemma}, we immediately get the following formulas:
\begin{align*}
\det(w_e, w_f) & = \frac{1}{ef}(a + e^2 + f^2) = \chi(\sh{E}, \sh{F}),\\
\det(w_f, l_e) & = \frac{a + e^2}{f} = \rk L_\sh{E} \sh{F},\\
\det(w_e, l_f) & = \frac{a + f^2}{e} = \rk R_\sh{F} \sh{E}.
\end{align*}
Note that mutation can change the orientation of the $w$'s with respect to the $p$'s.
\end{sub}

In our subsequent analysis, the circumference segments will play a crucial role. We start with the
following observation.

\begin{lemma}\label{peri1}
Both $w_e$ and $w_f$ are integral.
\end{lemma}

\begin{proof}
After a choice of coordinates we can
assume without loss of generality that $l_e = (1, 0)$, $l = (x, e^2)$,
$l_f = (y, -a)$ for some integers $x, y$. Then $p_e = (x - 1, e^2)$ and
$p_f = (y - x, -a - e^2)$ and $ef \chi(\sh{E}, \sh{F}) = \det(p_e, p_f)
= e^2(1 - y) - a(x - 1)$. So, in particular, $e$ divides $a(x - 1)$ and
therefore $e$ divides $a \cdot p_e$. Moreover, as $e$ divides $a + f^2$,
it also divides $(a + f^2) p_e$ and therefore also $f^2 \cdot p_e$. Hence,
$e$ divides $\gcd\{a, e^2, f^2\} \cdot p_e$. Now, via mutation, we can
replace $\sh{F}$ by any $\sh{E}_i$ with $e_i \neq 0$. In the fan, this
leaves $l_e$, $l$ and $p_e$ unchanged, and we obtain analogously that
$e$ divides $\gcd\{c_1(\sh{E}, \sh{E}_i)^2, e^2, e_i^2\} \cdot p_e$.
So we get that
$e$ divides $\gcd\{e^2, f^2, e_i^2 \mid e_i \neq 0\} \cdot p_e$. But
$\gcd\{e^2, f^2, e_i^2 \mid e_i \neq 0\} = 1$ and the assertion follows for
$w_e$ and, by exchanging the roles of $\sh{E}$ and $\sh{F}$, also
for $w_f$.
\end{proof}

By Proposition \ref{mutationlocal}, the effect of mutation is local in the sense that
if we apply a mutation to the pair $\sh{E}, \sh{F}$, say, then the triple $l_e, l, l_f$ gets transformed
to a triple $l_e, l', l_f$ and the other $l_i$ remain constant. The transformation of $l$ to $l'$ can
be described nicely with help of the circumference segments.
We consider the mutated pairs $\sh{F}, R_\sh{F} \sh{E}$ and $L_\sh{E} \sh{F}, \sh{E}$, by which
$l$ gets transformed to some $l'$ and $l^{\prime\prime}$, respectively, which satisfy the following relations:
\begin{equation*}
\left(\frac{a + f^2}{e}\right)^2 l_e + a l' + f^2 l_f = 0 \quad \text{ and } \quad
e^2 l_e + a l^{\prime\prime} + \left(\frac{a + e^2}{f}\right)^2 l_f = 0.
\end{equation*}
We then get the following transformation formulas.

\begin{lemma}\label{peri2}
We have $l' = l_f + \frac{a + f^2}{e} w_e = l_e + f (\det(w_e, w_f) w_e + w_f)$ and
$l^{\prime\prime} = l_e - \frac{a + e^2}{f} w_f = l_f - e(\det(w_e, w_f) w_f + w_e)$.
\end{lemma}

\begin{proof}
We only prove the formulas for $l'$; the case $l^{\prime\prime}$ then follows analogously.
By \ref{triplelemma}, $l'$ is completely determined by $a$ and the volumes $f^2$ and
$\left(\frac{a + f^2}{e}\right)^2$ relative to $l_f$ and $l_e$, respectively.
Then the first equality follows from
\begin{equation*}
\det(l_f + \frac{a + f^2}{e} w_e, l_f) = \left(\frac{a + e^2}{f}\right)^2
\quad \text{ and } \quad
\det(l_e, l_f + \frac{a + f^2}{e} w_e) = f^2,
\end{equation*}
which both are immediate consequences of formulas \ref{circumferenceformulas}.
The second equality follows from a simple rearrangement of terms.
\end{proof}

For $\chi(\sh{E}, \sh{F})^2 \leq 1$ we see that the transformations of $w_e$ and $w_f$ reflect the
periodic behaviour under subsequent permutations which we have observed in Remark \ref{rankgeneration2}.
For $\chi(\sh{E}, \sh{F})^2 > 0$ we observe the following.

\begin{corollary}\label{mutationtransform}
With above notation, consider the pair $\sh{F}, R_\sh{F} \sh{E}$ and assume that
$\chi(\sh{E}, \sh{F}) \neq 0$. If $\rk R_\sh{F} \sh{E} \neq 0$, then $w_e$ and $w_f$
transform in the sublattice which they generate by the matrix
$\left(
\begin{smallmatrix}
\chi(\sh{E}, \sh{F}) & -1 \\
1 & 0 
\end{smallmatrix}
\right)$. If $\rk L_\sh{E} \sh{F} \neq 0$, $w_e$ and $w_f$ transform by
$\left(
\begin{smallmatrix}
0 & 1 \\
-1 & \chi(\sh{E}, \sh{F}) 
\end{smallmatrix}
\right)$
\end{corollary}

\begin{remark}
For $\chi(\sh{E}, \sh{F})^2 > 4$, Corollary \ref{mutationtransform} gives us another method for
computing the terms
$\frac{\alpha_+^i - \alpha_-^i}{\alpha_+ - \alpha_-}$ of Corollary \ref{mutationtransform} for
$\sh{E}_0 = \sh{E}$ and $\sh{E}_1 = \sh{F}$. It follows by induction that for $i \geq 1$,
$\frac{\alpha_+^i - \alpha_-^i}{\alpha_+ - \alpha_-}$ coincides
with the upper left entry of the $(i - 1)$-st power of the matrix $\left(
\begin{smallmatrix}
\chi(\sh{E}, \sh{F}) & -1 \\
1 & 0 
\end{smallmatrix}
\right)$
\end{remark}

The following statement is less nice but stronger, as it in particular implies that $w_e$ and $w_f$
do not change their lattice length in $N$ under mutation.

\begin{lemma}\label{peri3}
We have $\det(w_e, w_f) w_e + w_f = G w_f$ and $\det(w_e, w_f) w_f + w_e = G' w_e$ for
$G, G' \in \operatorname{GL}_\Z(N)$.
\end{lemma}

\begin{proof}
We consider only the first equality; the second equality follows analogously. By construction,
$\det(w_e, wf_f) w_e + w_f = \frac{1}{f}(l' - l_e)$ which is a lattice vector by Lemma \ref{peri1}.
We show that its lattice length coincides with that of $w_f$, which then implies the assertion.
For this, we consider the exceptional triples $\sh{E}_0, \sh{F}, R_\sh{F} \sh{E}$ and
$L_{\sh{E}_0} \sh{E}, \sh{E}_0, \sh{F}$ (for simplicity, we assume that $e_0 \neq 0$, which can
always be arranged). In the first case, the exceptional pair $\sh{E}_0, \sh{F}$ gives rise to
a relation $f^2 l_{n - 1} + b l_e + e_0^2 l' = 0$, and in the second case we get
$f^2 l_e' + b l + e_0^2 l_f = 0$, where $l_e'$ is the mutation of $l_e$ corresponding to
$L_{\sh{E}_0} \sh{E}$. As the triples $l_{n - 1}, l_e, l'$ and $l_e', l, l_f$ correspond to the
same triple of volumes $f^2, b, e_0^2$, it follows by Lemma \ref{triplelemma} that
$l' - l_e = G \cdot (l_f - l)$ for some $G \in \operatorname{GL}_\Z(N)$, which implies the assertion.
\end{proof}

\begin{sub}\label{convexpictures}
We can consider the value $ef \chi(\sh{E}, \sh{F}) = \det(p_e, p_f) = a + e^2 + f^2$ as
a lattice-geometric measure of the convexity of the configuration of lattice vectors $l_e, l, l_f$.
More precisely, if we consider the angle between $p_e$ and $p_f$ at $l$, then we have three possibilities:
\begin{align*}
a + e^2 + f^2 > 0 & & \text{(convex)}, \\
a + e^2 + f^2 < 0 & & \text{(concave)}, \\
a + e^2 + f^2 = 0 & & \text{(flat)}.
\end{align*}
Figure \ref{convexpic} schematically depicts these possibilities for the case $a < 0$.
\end{sub}

\begin{figure}[ht]
\centering
\begin{tikzpicture}

\draw[->, thick] (-5, 0) -- (-2, 1);
\draw[->, thick] (-5, 0) -- (-1.5, 3.5);
\draw[->, thick] (-5, 0) -- (-4, 3);

\draw[->, thick, red] (-2, 1) -- (-1.5, 3.5);
\draw[->, thick, red] (-1.5, 3.5) -- (-4, 3);

\draw (-1.8, .8) node { $l_e$ };
\draw (-1.3, 3.45) node { $l$ };
\draw (-4.3, 3) node { $l_f$};
\draw[red] (-1.51, 2.05) node { $p_e$ };
\draw[red] (-2.85, 3.5) node { $p_f$ };

\draw (-3, .2) node {convex};

\draw[->, thick] (0, 0) -- (3, 1);
\draw[->, thick] (0, 0) -- (2, 2);
\draw[->, thick] (0, 0) -- (1, 3);

\draw[->, thick, red] (3, 1) -- (2, 2);
\draw[->, thick, red] (2, 2) -- (1, 3);

\draw (3.2, .8) node { $l_e$ };
\draw (1.9, 1.65) node { $l$ };
\draw (.7, 3) node { $l_f$};
\draw[red] (2.85, 1.65) node { $p_e$ };
\draw[red] (1.85, 2.65) node { $p_f$ };

\draw (2, .2) node {flat};

\draw[->, thick] (5, 0) -- (8, 1);
\draw[->, thick] (5, 0) -- (6.5, 1.5);
\draw[->, thick] (5, 0) -- (6, 3);

\draw[->, thick, red] (8, 1) -- (6.5, 1.5);
\draw[->, thick, red] (6.5, 1.5) -- (6, 3);

\draw (8.2, .8) node { $l_e$ };
\draw (6.4, 1.15) node { $l$ };
\draw (5.7, 3) node { $l_f$};
\draw[red] (7.35, 1.55) node { $p_e$ };
\draw[red] (6.55, 2.35) node { $p_f$ };

\draw (7, .2) node {concave};

\end{tikzpicture}
\caption{Convex, flat, and concave configurations of $l_e$, $l$, $l_f$ for $a < 0$.}\label{convexpic}
\end{figure}
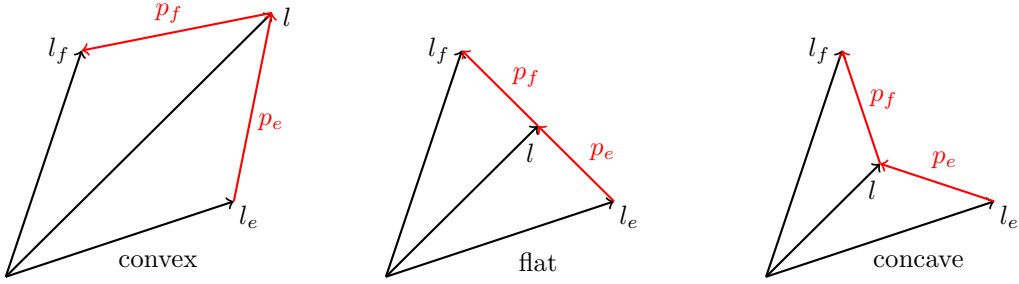

\begin{definition}
Corresponding to above inequalities, we call an exceptional pair $\sh{E}, \sh{F}$ with
$e, f \neq 0$ either
{\em convex}, {\em concave}, or {\em flat}. We call the value $a + e^2 + f^2$ the
{\em convexity} of the pair $\sh{E}, \sh{F}$.
\end{definition}

The following lemma shows that very often we can decrease the convexity of an exceptional pair
by mutation.

\begin{lemma}\label{concavemutation}
Let $\sh{E}, \sh{F}$ be a convex exceptional pair and assume that $a < 0$.
Then either $\max\{(\rk L_\sh{E} \sh{F})^2,$ $e^2\} < \max\{e^2, f^2\}$ or $\max\{(\rk R_\sh{F} \sh{E})^2, f^2\} < \max\{e^2, f^2\}$.
\end{lemma}

\begin{proof}
We can assume without loss of generality that $0 < e^2 \leq f^2$. Then we have to show that
both $\frac{a + e^2}{f} < f$ and $-\frac{a + e^2}{f} < f$, which both follow trivially from
the assumptions.
\end{proof}

\begin{sub}\label{flatnesspreserved}
Observe that flatness implies $a < 0$ and that mutation preserves flatness. In particular,
for a flat pair
we have $\rk L_\sh{E} \sh{F} = -f$ and $\rk R_\sh{F} \sh{E} = -e$ and it can never occur that iterative
mutations of a non-flat pair can result in a flat pair.
\end{sub}

\begin{sub}\label{convexdecrease}
So, the lemma implies that for any convex pair with $a < 0$, we can obtain by mutation
a sequence of exceptional pairs where both the maximal rank as well as the convexity strictly decrease
while $a$ remains constant,
until either we arrive at a concave pair or one of the exceptional objects acquires ranks zero.
\end{sub}

\begin{corollary}\label{algorithm}
Let $\sh{E}_1, \dots, \sh{E}_n$ be an exceptional sequence, then we can produce by mutation an
exceptional sequence  $\sh{Z}_1, \dots, \sh{Z}_t$, $\sh{F}_1, \dots, \sh{F}_{n - t}$ such that
the $\sh{Z}_i$ have rank zero and for any $1 \leq i \leq n - t$, if $\sh{F}_i, \sh{F}_{i + 1}$ 
is not a flat or concave exceptional pair then $\tilde{A}_i^2 \geq 0$, where $\tilde{A}_i$ is the corresponding
element of the contracted toric system.
\end{corollary}

\begin{proof}
We iterate the following procedure.
\begin{enumerate}[1)]
\item Choose any pair $\sh{E}_i, \sh{E}_j$ with $1 \leq i < j \leq n + 1$ (recall that we are working
with the cyclically extended sequence) and $e_i, e_j \neq 0$ such that $e_k = 0$
for all $i < k  <j$. If $\sh{E}_i, \sh{E}_j$ is a convex pair and $c_1(\sh{E}_i, \sh{E}_{i + 1})^2 < 0$
then continue with steps 2) and 3).
\item Move all $\sh{E}_k$ for $i < k < j$, to the position left of $\sh{E}_i$ via right mutation.
We denote $R_{\sh{E}_{j - 1}} \cdots R_{\sh{E}_{i + 1}} \sh{E}_i =: \sh{E}_i'$. The resulting pair
$\sh{E}_i', \sh{E}_j$ is still convex and $(\rk \sh{E}_i')^2 = e_i^2$.
\item We iterate mutations as in \ref{convexdecrease} and successively minimize the maximal
rank of the resulting mutated pairs until either the pair becomes concave or a mutation results
in an object of rank zero.
\end{enumerate}
We repeat these steps until the resulting sequence does not contain any convex pair $\sh{E}_i,
\sh{E}_j$ with $c_1(\sh{E}_i$, $\sh{E}_j)^2$ $< 0$. As we minimize the absolute value of ranks
it is guaranteed that this iteration will terminate.
We then finalize this procedure by moving all rank zero objects by right mutations to the leftmost range.
\end{proof}

\section{The global picture}\label{globalsection}

So far, we have established that the Gale transforms of a toric system at least locally represent
the data of a complete toric surface. That is, by Corollary \ref{localfan} and Proposition
\ref{primitivity}, every triple $l_{i - 1}, l_i, l_{i + 1}$ generates
a fan with two maximal cones which intersect in a facet such that $l_{i - 1}, l_i, l_{i + 1}$
are the primitive vectors which generate the $1$-dimensional cones. In this section we want to
show that all the $l_i$ indeed form the set of primitive vectors which generate the fan of a
complete toric surface. For this, we can start with the two cones $\sigma_1, \sigma_2$, generated by,
say, $l_1, l_2, l_3$.
Then clearly, we can try to add a third cone $\sigma_3$ generated by $l_3$ and $l_4$. By construction,
this cone
lies in counterclockwise direction from the first two cones, and we know that $\sigma_2$ and $\sigma_3$
again form a fan with two maximal cones. However, so far we do not have any information on whether
$\sigma_1, \sigma_2, \sigma_3$ fit together to form a proper fan. For this, we would have to prove
that either $\sigma_1 \cap \sigma_3 = \{0\}$ or $\sigma_1 \cap \sigma_3 = \Q_{\geq 0} l_1$ (which
implies $l_1 = l_4$). Similarly, if we successively
add $\sigma_i = \langle l_i, l_{i + 1} \rangle_{\Q_{\geq 0}}$ in counterclockwise fashion, we have to
show that $\sigma_i$ obeys the correct intersection properties with the previously added cones.

The only possibility that this construction can violate these intersection properties is that
for some $3 \leq i < n$, the intersection $\sigma_1 \cap \sigma_i$ is a two-dimensional cone by itself.
Then $\sigma_1, \dots, \sigma_i$ cover all of $N_\Q$ for some $i < n$ without closing up to a
proper fan. Continuing this way, we would end up with a sequence of cones which in counter-clockwise
order cover $N_\Q$ several times by ``winding'' around the origin until finally $\sigma_n$ and
$\sigma_1$ close up these windings via the triple $l_n, l_1, l_2$. We are going to show that there
indeed can only be one winding.

\begin{sub}\label{zeroobjectwinding}
Recall that in Section \ref{localconstellations} we used the simplified assumption that $e_i \neq 0$
for all $i$. We cannot make this simplification in this section, but instead we will make use of
the fact that the number of windings does not depend on the existence of objects of rank
zero among the $\sh{E}_i$. Indeed, as already remarked in \ref{zerojumpremark}, if $e_i = 0$ for
some $i$ then we can use mutation to produce a
sequence which only contains objects of nonzero rank. By Lemmas
\ref{zeromutationconfiguration} and \ref{peri2}, this is a completely local operation which cannot
change the number of windings. This is also true for the converse process, where rank zero objects
are created by mutation. Our strategy will be to use Corollary \ref{algorithm}
in order to create a sequence of the form $\sh{Z}_1, \dots, \sh{Z}_t, \sh{E}_1, \dots, \sh{E}_{n - t}$
to show that the corresponding $\tilde{l}_1, \dots, \tilde{l}_{n - t}$ comprise only one winding.
\end{sub}

\begin{sub}\label{p2example}
We consider first the case $n - t = 3$, i.e. an exceptional sequence
$\sh{Z}_1, \dots, \sh{Z}_{n - 3}, \sh{E}_1, \sh{E}_2, \sh{E}_3$ with $z_i = 0$ and $e_j \neq 0$.
Without loss of generality we can assume that $e_i > 0$ for all $i$.
Then with the notation of the previous
section and Section \ref{galedualsection}, we have one single
relation $e_3^2 \tilde{l}_1 + e_1^2 \tilde{l}_2 + e_2^2 \tilde{l}_3 = 0$. The problem of additional
windings does not occur in this case and we obtain the fan of a weighted projective
space $\mathbb{P}(e_1^2, e_2^2, e_3^2)$ and we would like to determine the possible values for the
$e_i$.

The $E_i = c_1(\sh{Z}_i)$ form the
basis of a maximal negative definite subspace of $\chnum^1(X)$. As we have seen in Section
\ref{rankzerosection}, possibly after twisting the sequence with an appropriate line bundle
we can assume that $E_i \cdot c_1(\sh{E}_j) = 0$ for all $i, j$. Moreover, because the exceptional
sequence generates $\knum(X)$, the $E_i$ together with the $c_1(\sh{E}_j)$ form a generating set of
$\chnum^1(X)$. Now, as $E_1, \dots, E_{n - 3}$ together with $A_1, A_2, A_3$ generate $A \supseteq
\chnum^1(X)$,
so do $E_1, \dots, E_{n - 3}$ together with $\tilde{A}_1, \tilde{A}_2, \tilde{A}_3$. Moreover, the
$\Q$-span of $\tilde{A}_1, \tilde{A}_2, \tilde{A}_3$ is isomorphic to $\Q$ and contained in the orthogonal
complement of the $E_i$. If we denote $H  \in \langle \tilde{A}_1, \tilde{A}_2, \tilde{A}_3 \rangle_\Q
\cap \chnum^1(X)$ a minimal integral element, then by the unimodularity of the intersection pairing,
$E_1, \dots, E_n$ and $H$ necessarily form an
orthogonal basis of $\chnum^1(X)$, in particular we have $H^2 = 1$ and $\tilde{A}_i = \alpha_i H$ with
$\alpha_1, \alpha_2, \alpha_3 \in \Q$. Now we denote $J := -K_X
+ \sum_{i = 1}^{n - 3} (1 + 2 \delta_i) E_i = \tilde{A}_1 + \tilde{A}_2 + \tilde{A}_3 = \gamma H$
for some $\gamma \in \Z$. Then from
$$
\tilde{A}_{i - 1} \cdot \tilde{A}_i = \frac{1}{e_i^2}
$$
for $i = 1, 2, 3$ we compute
\begin{equation*}
\alpha_i^2 = \frac{e_{i + 2}^2}{e_i^2 e_{i + 1}^2}.
\end{equation*}
If $\alpha_i = -e_{i + 2}/(e_i e_{i + 1})$ for one $i$, it follows that $\alpha_i = -e_{i + 2} / (e_i e_{i + 1})$ for all $i$, hence,
after possibly exchanging $H$ with $-H$, we can assume without loss of generality that $\alpha_i = e_{i + 2}/(e_i e_{i + 1})$
for $i = 1, 2, 3$.
Now with the identities \ref{eulersym} (\ref{eulersymi}) \& (\ref{eulersymii}):
\begin{equation*}
\chi(\sh{E}_1, \sh{E}_2) = \frac{1}{e_1 e_2} (c_1(\sh{E}_1, \sh{E}_2)^2
+ e_1^2 + e_2^2) = \frac{1}{e_1 e_2} (e_3^2 + e_1^2 + e_2^2) =
 -K_X \cdot c_1(\sh{E}_1, \sh{E}_2) = J \cdot c_1(\sh{E}_1 \sh{E}_2) = \gamma e_3.
\end{equation*}
So, the $e_i$ must satisfy the following equation:
\begin{equation*}
e_1^2 + e_2^2 + e_3^2 = \gamma e_1 e_2 e_3.
\end{equation*}
It is well known that this equation admits integer solutions if and only if $\gamma \in \{1, 3\}$
(see e.g. \cite[\S 2.1]{Aigner13}) and $(e_1, e_2, e_3)$ is a solution for the case $\gamma = 1$
iff $(e_1 / 3, e_2 / 3, e_3 / 3)$ is a solution for the case $\gamma = 3$. As necessarily
$\gcd\{e_1, e_2, e_3\} = 1$, the result is that $\gamma = 3$ and the $e_i$ satisfy the {\em Markov
equation}:
$$
e_1^2 + e_2^2 + e_3^2 = 3 e_1 e_2 e_3.
$$
For the case $n = 3$ this reproduces a well-known result of Rudakov \cite{Rudakov89a} for
$\mathbb{P}^2$ from a purely combinatorial perspective. Also, as in the case for $\mathbb{P}^2$
we can use mutations such that after finitely many steps we obtain $e_1^2 = e_2^2 = e_3^2 = 1$.
Above observation is also a combinatorial variant of results of Hacking
\cite{HackingProkhorov10, Hacking13} which
gives a correspondence of the construction of exceptional sequences on $\mathbb{P}^2$ to
$\Q$-Gorenstein degenerations of $\mathbb{P}^2$ whose exceptional fiber is a weighted projective
spaces $\mathbb{P}(e_1^2, e_2^2, e_3^2)$, where the $e_i$ satisfy the Markov equation.
\end{sub}

We are now finally able to deal with the second Chern classes of rank zero objects. The following
result is an important step towards proving Theorem \ref{deltatheorem}.

\begin{proposition}\label{p2nodefects}
As in \ref{p2example}, let $\sh{Z}_1, \dots, \sh{Z}_{n - 3}, \sh{E}_1, \sh{E}_2, \sh{E}_3$ be an
exceptional sequence with $z_i = 0$ and $e_j \neq 0$. Then $\delta_i \in \{0, 1\}$
for every $i$.
\end{proposition}

\begin{proof}
With the notation of \ref{p2example} and our general assumption that $K_X^2 = 12 - n$, we have
\begin{align*}
12 - n = K_X^2 & = (J - \sum_{i = 1}^{n - 3} (1 + 2\delta_i)E_i)^2
= 9 - (n - 3) - 4 \sum_{i = 1}^{n - 3}\delta_i(\delta_i + 1)\\
& = 12 - n - 4 \sum_{i = 1}^{n - 3}\delta_i(\delta_i + 1),
\end{align*}
therefore $\sum_{i = 1}^{n - 3}\delta_i (\delta_i + 1) = 0$ which implies $\delta_i \in \{0, -1\}$
for every $i$.
\end{proof}

\begin{sub}\label{nequals4}
Another special case which we have to settle is $n - t = 4$. This case exhibits a nice symmetry with
$a_{i + 1} = \det(\tilde{l}_{i + 2}, \tilde{l}_i) = -\det(\tilde{l}_i, \tilde{l}_{i + 2}) = -a_{i - 1}$ for
every $i$.

If no opposing pair, i.e. an $\tilde{l}_i$ such that $\tilde{l}_{i + 2} = -\tilde{l}_i$ exists
then it is elementary to see that the $\tilde{l}_i$
can produce only one winding and for some $i$ we have $a_{i - 2}, a_{i - 1} > 0$
and $a_i, a_{i + 1} < 0$.
If we move the rank zero objects such that our exceptional sequence is of the form
$\sh{Z}_1, \dots, \sh{Z}_{n - 4}, \sh{E}_i, \sh{E}_{i + 1}, \sh{E}_{i + 2}, \sh{E}_{i + 3}$, then
the two pairs $\sh{E}_i, \sh{E}_{i + 1}$ and $\sh{E}_{i + 1}, \sh{E}_{i + 2}$ cannot be
both concave, because otherwise the $\Q_{\geq 0}$-span of the
$\tilde{l}_i$ could not generate $N_\Q$. This means that
whenever there is no opposing pair $\tilde{l}_{i + 2} = -\tilde{l}_i$, we can use Lemma
\ref{concavemutation} in order to decrease the maximal rank of a concave pair.
By iteration we will eventually end up in one of two
cases:
\begin{enumerate}[1)]
\item We produce a pair $\tilde{l}_{i + 2} = -\tilde{l}_i$
\item We produce a sequence of the form $\sh{Z}_1, \dots, \sh{Z}_t, \sh{Z}, \sh{F}_1, \sh{F}_2, \sh{F}_3$ with $z = 0$ and
$f_i^2 > 0$ for $i = 1, 2, 3$. Then it follows from \ref{p2example} that the $f_i$ must satisfy
the Markov equation $f_1^2 + f_2^2 + f_3^2 = 3 f_1 f_2 f_3$.
\end{enumerate}
\end{sub}

\begin{proposition}\label{hirzebruchnodefects}
Let $\sh{Z}_1, \dots, \sh{Z}_{n - 4}, \sh{E}_1, \sh{E}_2, \sh{E}_3, \sh{E}_4$ be an exceptional
sequence with $z_i = 0$ and $e_j \neq 0$ and assume that $\tilde{l}_j = -\tilde{l}_{j + 2}$ for some
$1 \leq j \leq 4$. Then $\delta_i \in \{0, -1\}$ for every $i$ and $e_j^2 = 1$ for every $j$.
\end{proposition}

\begin{proof}
If $\tilde{l}_j = -\tilde{l}_{j + 2}$, then $a_{j- 1} = a_{j + 1} = \det(\tilde{l}_j, \tilde{l}_{j + 2}) = 0$. So, by
Lemma \ref{opposite1} we get that $e_{j - 1}^2 = e_j^2$ and $e_{j + 1}^2 = e_{j + 2}^2$.
As in \ref{p2example}, we have an orthogonal decomposition
$\chnum^1(X)_\Q \supset A = \langle E_1, \dots, E_{n - 4}\rangle_\Z + \langle \tilde{A}_1, \dots
\tilde{A}_4\rangle_\Z$ with $\sum_{k = 1}^4 \tilde{A}_j = J = -K_X - \sum_{i = 1}^{n - 4}
(1 + 2 \delta_i)E_i$ and get:
\begin{align*}
K_X^2 = 12 - n = J^2 - (n - 4) - 4 \sum_{i = 1}^{n - 4} \delta_i (\delta_i + 1).
\end{align*}
Moreover, with $a_{j - 1} = a_{j + 1} = 0$, $a_j = -a_{j + 2}$, $e_{j - 1} = e_{j + 1}$, and $e_j = e_{j + 2}$, we
get $\tilde{A}_{j - 1}^2 = \tilde{A}_{j + 1}^2 = 0$, $\tilde{A}_j^2 = -\tilde{A}_{j + 2}^2$, and:
\begin{align*}
J^2 = \left(\sum_{k = 1}^4 \tilde{A}_k\right)^2 =
 4(\frac{1}{e_j^2} + \frac{1}{e_{j + 1}^2}).
\end{align*}
Substituting this into the previous formula, we get:
$$
\sum_{i = 1}^{n - 4} \delta_i (\delta_i + 1) = \frac{1}{e_j^2} + \frac{1}{e_{j + 1}^2} - 2,
$$
where the right hand side is integral iff $e_j^2 = e_{j + 1}^2 = 1$. So,
$$
\sum_{i = 1}^{n - 4}\delta_i (\delta_i + 1) = 0,
$$
hence $\delta_i \in \{0, -1\}$ for every $i$.
\end{proof}

By Proposition \ref{hirzebruchnodefects}, the $\tilde{l}_i$ generate the fan
of a Hirzebruch surface and we recover a special case of the corresponding result for exceptional
sequences of line bundles as was proven in \cite[Theorem 3.5]{HillePerling11}
We also have reproduced a result of Nogin \cite[\S 3]{Nogin91} which states that every exceptional
sequence on a Hirzebruch surface can be mutated to a sequence consisting of objects of rank one
(see also Example \ref{hirzebruchexample}).

The following Lemma helps to separate the cases $n - t \leq 4$ and $n - t > 4$.

\begin{lemma}\label{positiveselfintersection}
Let $n - t \geq 4$.
\begin{enumerate}[(i)]
\item\label{positiveselfintersectioni}
Assume that $\tilde{A}_i^2 \geq 0$ for some $i$, then there exists at most one other $\tilde{A}_j$
such that $\tilde{A}_j^2 > 0$. If so, then $j$ is either $i - 1$ or $i + 1$.
\item\label{positiveselfintersectionii}
If $i \neq j$, $\tilde{A}_i^2 = 0$, $\tilde{A}_j^2 = 0$, and $\tilde{A}_i \cdot \tilde{A}_j = 0$ then $j = i + 2$
and $n - t = 4$.
\item\label{positiveselfintersectioniii}
If $\tilde{A}_i^2 = \tilde{A}_{i + 1}^2 = 0$ for some $i$, then either $n - t = 4$ and
$\tilde{A}_j^2 = 0$ for all $j$, or $n - t > 4$ and $\tilde{A}_j^2 < 0$ for all $j \neq i, i + 1$.
\end{enumerate}
\end{lemma}

\begin{proof}
$\chnum^1(X)_\Q$ is a metric space with respect to the intersection product
which has an orthogonal decomposition $\chnum^1(X)_\Q \simeq C^+ \bot\ C^-$ into strictly positive and
strictly negative definitive part. This induces a decomposition of $\tilde{A}_\Q = \langle \tilde{A}_1,
\dots, \tilde{A}_{n - t} \rangle_\Q = \tilde{A}^+ \bot \tilde{A}^-$ where $\tilde{A}^- = \tilde{A} \cap
C^-$. By the Hodge Index Theorem we have $\dim_\Q \tilde{A}^+ \leq  1$.
Up to cyclic renumbering let us now assume that $\tilde{A}_1^2 \geq 0$. Then clearly $\tilde{A}_1$ and
$\tilde{A}^-$ generate $\tilde{A}_\Q$ as a $\Q$-vector space and for every $i$ we can write
$\tilde{A}_i = \alpha_i \tilde{A}_1 + n_i$ for some $\alpha_i
\in \Q$ and $n_i \in \tilde{A}^-$. Now, for any $2 < i < n$ we have $0 = \tilde{A}_1 \cdot \tilde{A}_i
= \tilde{A}_1 \cdot (\alpha_i A_1 + n_i) = \alpha_i \tilde{A}_1^2 + \tilde{A}_1 n_i$, hence
$\alpha_i \tilde{A}_i^2 = - \tilde{A}_1 n_i$. Now, for $2 < i < n$
we compute $\tilde{A}_i^2 = (\alpha_i \tilde{A}_1 + n_i)^2 = \alpha_i \tilde{A}_1 \cdot
(\alpha_i \tilde{A}_1 + 2 n_i) + n_i^2 =
\alpha_i \tilde{A}_1 \tilde{A}_i + \alpha_i \tilde{A}_1 n_i + n_i^2$
$= \alpha_i \tilde{A}_1 n_i + n_i^2 = -\alpha_i^2 \tilde{A}_i^2 + n_i^2 \leq 0$, as both
$-\alpha_i^2 \tilde{A}_i^2 \leq 0$
and $n_i^2 \leq 0$. If $\tilde{A}_1^2 > 0$ then this inequality is strict for every $i$. This
leaves only $\tilde{A}_2$ and $\tilde{A}_n$ which can have positive self-intersection number. But
because $n \geq 4$, we also have $\tilde{A}_2 \cdot \tilde{A}_n = 0$. So if one of $\tilde{A}_2,
\tilde{A}_n$ has positive self-intersection, then we can argue as for $\tilde{A}_1$ and conclude
that the other must have negative self-intersection and (\ref{positiveselfintersectioni}) follows.

(\ref{positiveselfintersectionii})
The same computation leads to $\tilde{A}_i^2 = n_i^2 = 0$, hence $n_i = 0$ and thus $\tilde{A}_i =
\alpha_i \tilde{A}_1$. Then $\tilde{A}_i \cdot \tilde{A}_2 \neq 0$ and $\tilde{A}_i \cdot A_n \neq 0$,
which is only possible if $n = 4$ and $i = 3$.

(\ref{positiveselfintersectioniii}) We know from (\ref{positiveselfintersectioni}) that there
cannot be any $\tilde{A}_j$ with $\tilde{A}_j^2 > 0$. If there exists a third $\tilde{A}_j$ with
$\tilde{A}_j^2 = 0$, then $\tilde{A}_j \cdot \tilde{A}_i = 0$ or $\tilde{A}_j \cdot \tilde{A}_{i + 1}
= 0$ and by (\ref{positiveselfintersectionii}) this implies $n - t = 4$ and $l_1 = -l_3$, $l_2 = -l_4$
by Lemma \ref{opposite1}. Note that $\det(l_{i - 1}, l_{i + 1}) = 0$ for any $i$ implies $a_i = 0$ and thus
$\tilde{A}_i^2 = 0$.
\end{proof}

\begin{proposition}\label{winding}
Let $\sh{E}_1, \dots, \sh{E}_n$ be a numerically exceptional sequence with contracted toric
system
$\tilde{A}_1, \dots$, $\tilde{A}_{n - t}$. Then the Gale duals $\tilde{l}_1, \dots, \tilde{l}_{n - t}$
generate the fan corresponding to a complete toric surface such that the $\tilde{l}_i$ are primitive
vectors of the rays in this fan and the maximal cones are generated by $\tilde{l}_i, \tilde{l}_{i + 1}$
for $1 \leq i < n - t$ and $\tilde{l}_{n - t}, \tilde{l}_1$.
\end{proposition}

\begin{proof}
By Corollary \ref{algorithm}, we can produce by mutation an exceptional sequence $\sh{Z}_1, \dots,
\sh{Z}_s$, $\sh{F}_1, \dots, \sh{F}_{n - s}$ with $s \geq t$ and contracted toric system
$B_1, \dots, B_{n - s}$, such that any exceptional pair $\sh{F}_i,
\sh{F}_{i + 1}$ (where we read the indices cyclically modulo $n - s$) with
$B_i^2 < 0$ is concave or flat. We denote $k_1, \dots, k_{n - s}$
the Gale duals of the $B_i$.
The cases $n - s \leq 4$ have been covered in \ref{p2example} and \ref{nequals4}, so we assume
$n - s \geq 5$. It follows from Lemma \ref{positiveselfintersection} that there exist at most two
$B_i$ with $B_i^2 > 0$ which then must be adjacent.

Now consider any subsequent $B_j, B_{j + 1}, \dots, B_{j + r}$ such that $B_{j + j'}^2 < 0$ for
all $0 \leq j' \leq r$. Then the sequence of circumference segments $p_j, \dots, p_{j + r}$ is
non-convex with respect to the origin, i.e. for any $0 \leq j' < r$, the vector $k_{j + j'}$ is
contained in the convex hull of $0, k_{j + j' - 1}, k_{j + j' + 1}$. But this implies that
all the $k_{j + j'}$ for $-1 \leq j' \leq r$ are contained in the same half space whose boundary
is given by $\Q k_{j - 1}$. In particular, $k_{j - 1}$ is the only one which is contained in the
boundary.

With this observation it follows that there must be at least one $B_i$ with $B_i^2 > 0$, say $B_1$,
and $B_j, \dots, B_n$ is the maximal sequence with $B_i^2 < 0$, where $j = 2$ or $j = 3$. This implies
that all cones except for possibly two (if $j = 2$) or three (if $j = 3$) are contained in a
half space and it follows from elementary geometric arguments that we cannot produce more than
one winding this way.

Now, as we have argued \ref{zeroobjectwinding}, we can conclude that also the original contracted
toric system does not produce more than one winding and the assertion follows.
\end{proof}

\begin{corollary}\label{rankonezero}
Let $\sh{E}_1, \dots, \sh{E}_n$ and $\sh{Z}_1, \dots, \sh{Z}_t$, $\sh{F}_1, \dots, \sh{F}_{n - t}$
be numerically exceptional sequences as in Corollary \ref{algorithm}. Then either $t = n - 3$ and
$f_1, f_2, f_3$
satisfy the Markov equation of \ref{p2example} or $t = n - 4$ and the $\sh{F}_i$ are objects of ranks
$\pm 1$. In particular, any exceptional sequence can by mutation be transformed into an exceptional
sequence consisting only of objects of ranks $\pm 1$ and $0$.
\end{corollary}

\begin{proof}
The assertions follow from Proposition \ref{winding} and the elementary observation that
a fan of a complete toric surface with at least $5$ rays always contains a configuration of
adjacent primitive vectors, $k_1, k_2, k_3$, say, with $\det(k_3, k_1) < 0$ and which are convex
in the sense of definition \ref{convexpictures}. Hence, the procedure of Corollary \ref{algorithm} always
results in an exceptional sequence with $n - t \leq 4$. Then as in the last part of \ref{nequals4},
if $n - t = 4$ then by Proposition \ref{hirzebruchnodefects} we know that the $\sh{F}_i$ have ranks
$\pm 1$. If $n - t = 3$, the $\sh{F}_i$ might have ranks different from $\pm 1$, but as in
the case of $\mathbb{P}^2$ we can conclude that by further mutation we can transform the
sequence to $\sh{Z}_1, \dots \sh{Z}_{n - 3}, \sh{F}_1', \sh{F}_2', \sh{F}_3'$ with $\rk \sh{F}_i' = \pm 1$
for $i = 1, 2, 3$.
\end{proof}

We are now ready to proof Theorem \ref{deltatheorem}:
\smallskip

\paragraph{\bf Proof of Theorem \ref{deltatheorem}}
Let $\sh{Z}, \sh{E}_2, \dots, \sh{E}_n$ be an exceptional sequence with $z = 0$.
If we follow the procedure described in the proof of Corollary \ref{algorithm}, we can produce
an exceptional sequence of the form $\sh{Z}_1 = \sh{Z}, \sh{Z}_2, \dots, \sh{Z}_{t}, \sh{E}_1, \dots,
\sh{E}_{n - t}$, where $z_i = 0$ and $e_j^2 \neq 0$. By Corollary \ref{rankonezero} we can even
assume that $n - t = 3$ or $n - t = 4$ and $e_j^2 = 1$ for all $j$. Then we can apply either
Proposition \ref{p2nodefects} or Proposition \ref{hirzebruchnodefects} to show that $\delta_i \in
\{0, -1\}$ for every $i$ and in particular $\delta_1 = \delta_\sh{Z} \in \{0, -1\}$.

\smallskip

The following theorem can also be considered as a corollary of Proposition \ref{winding} and
Corollary \ref{rankonezero}.

\begin{theorem}\label{rankone}
Let $X$ be a numerically rational surface. Then any numerically exceptional sequence on $X$ can be
transformed by mutation into a numerically exceptional sequence consisting only of objects of rank one.
\end{theorem}

\begin{proof}
By Corollary \ref{rankonezero}, every sequence can be transformed into an exceptional sequence of
objects of ranks one and zero. To this sequence is associated the fan of a toric surface which is
given by the Gale duals $\tilde{l}_1, \dots, \tilde{l}_{n - t}$ of the contracted toric system.
By assumption, we have $\det(\tilde{l}_i \tilde{l}_{i + 1}) = 1$ for all $i$, hence the toric surface
is smooth. If we instead consider the Gale duals of the uncontracted toric system, we obtain by
Proposition \ref{multiplicity} the same fan, but where the rays come with possible multiplicities, i.e.
if $\sh{E}_i, \sh{E}_{i + 1}, \dots, \sh{E}_j$ is a sub-sequence for $i < j$ such that $e_k = 0$ for
all $i < k < j$, then the multiplicity of the Gale dual of the contracted element
$s(\sh{E}_j) - s(\sh{E}_i) + \sum_{k = i + 1}^{j - 1} c_1(\sh{E}_k)$ is $j - i$. For any pair
$\sh{E}_i, \sh{E}_{i + 1}$ with $e_i = 0$, we have $\rk L_{\sh{E}_i} \sh{E}_{i + 1} = -e_i$
and $\rk R_{\sh{E}_{i + 1}} \sh{E}_i = -e_{i + 1}$ (and similarly if $e_{i + 1} = 0$).
We have seen in Section \ref{rankzeromoving} the mutation from $\sh{E}_i, \sh{E}_{i + 1}$ to
$L_{\sh{E}_i} \sh{E}_{i + 1}, \sh{E}_i$ leaves the associated fan unchanged, except for a possible
exchange of multiplicities. On the other hand, if $e_{i + 1} \neq 0$, then by Theorem
\ref{deltatheorem} and Lemma \ref{zeromutationconfiguration} the mutation to
$\sh{E}_{i + 1}, R_{\sh{E}_{i + 1}} \sh{E}_i$ yields a new object of rank $\pm 1$, and the
associated fan obtains a new primitive vector $l$, whose position is given by the relation
$e_{i + 1}^2 \tilde{l}_k - e_{i + 1}^2 l + e_{i + 1}^4 \tilde{l}_{k + 1} = 0$ for the corresponding
$1 \leq k \leq n - t$, and with $e_{i + 1}^2 = 1$ we get $l' = \tilde{l}_k + \tilde{l}_{k + 1}$.
That is, the right
mutation yields a bigger fan which corresponds to a toric blow-up of the original fan. Similarly, if
$e_i \neq 0$ and $e_{i + 1} = 0$, then left mutation results in a blow-up as well. Now we can iterate
the following two steps of mutations $\sh{E}_i, \sh{E}_{i + 1}$ to $L_{\sh{E}_i} \sh{E}_{i + 1},
\sh{E}_i$ (respectively $\sh{E}_{i + 1}, R_{\sh{E}_{i + 1}} \sh{E}_i$):
\begin{enumerate}[1)]
\item If $e_i = 0$ (respectively $e_{i + 1} = 0$), which only affect multiplicities of the primitive
vectors.
\item If $(e_i^2, e_{i + 1}^2) = (1, 0)$ (respectively $(e_i^2, e_{i + 1}^2) = (0, 1)$), which correspond
to smooth blow-ups.
\end{enumerate}
Iterating these steps at will we can realize every smooth toric surface which can be obtained from
the original $\tilde{l}_1, \dots, \tilde{l}_{n - t}$ by at most $t$ blow-ups.
\end{proof}

\begin{remark}\label{delayedremark}
By Theorem \ref{rankone}, any toric system coming from a numerically exceptional sequence indeed
can be constructed by mutation from a toric system associated to rank one objects as were considered
in \cite{HillePerling11} (see also Remark \ref{delayremark}).
Note that the relevant analysis of \cite[\S 2]{HillePerling11} is strictly on the numerical level
and therefore also applies to any numerical exceptional sequence whose elements' classes
in $\knum(X)$ coincide with that of invertible sheaves. Moreover, note that any numerically exceptional
object $\sh{L}$ of rank one has the same class in $\knum(X)$ as an invertible sheaf $\sh{O}(D)$,
where the class of $D$ in $\Ch^1(X)$ coincides with $c_1(\sh{L})$. This is easy to see from
$c_2(\sh{L}) = 0$ by \ref{eulerchar1} (\ref{eulerchar1ii}) and by the injectivity of the
Chern-isomorphism with $\ch(\sh{L}) = \ch(\sh{O}(D)) \in \Ch^1(X)_\Q$.
\end{remark}

\begin{example}\label{hirzebruchexample}
We construct examples of exceptional objects of rank one which are not isomorphic to
invertible sheaves.
Consider an even Hirzebruch surface $\mathbb{F}_a$, where $a = 2b$ for $b >= 0$ and denote
$P, Q$ the generators of its nef cone, where $P^2 = 0, Q^2 = a$ and $P \cdot Q = 1$. In
\cite[Proposition 5.2]{HillePerling11} it was shown that $\mathbb{F}_a$ admits two families
of toric systems for exceptional sequences consisting of objects of rank one, which are of
the following form:
\begin{align*}
F_1(s): & &  (- s - b) P + Q, P, (s - b) P + Q, P & \text{ for } s \in \Z,\\
F_2(s): & & -b P + Q, P + s(-b P + Q), -b P + Q, P - s(-b P + Q) & \text{ for } s \in \Z.
\end{align*}
Note that both families meet at $s = 0$. Any toric system of type $F_1(s)$ corresponds to an
exceptional sequence of invertible sheaves $\mathbf{L} = \sh{L}_1, \dots, \sh{L}_4$. Assume that,
say, $c_1(\sh{L}_2, \sh{L}_3) = P$, then $\rk L_{\sh{L}_2} \sh{L}_3 = \rk R_{\sh{L}_3} \sh{L}_2 = 1$
and both mutations are isomorphic to invertible sheaves, and
$c_1(L_{\sh{L}_2} \sh{L}_3, \sh{L}_2) = c_1(\sh{L}_3$, $R_{\sh{L}_3} \sh{L}_2) = P$.
Moreover, $c_1(\sh{L}_1, L_{\sh{L}_2} \sh{L}_3) = c_1(\sh{L}_1, \sh{L}_2) - 2P$ and
$c_1(R_{\sh{L}_3} \sh{L}_2) = c_1(\sh{L}_3, \sh{L}_4) - 2P$. That is, from a sequence of type
$F_1(s)$ we obtain by left mutation of the pair $\sh{L}_2, \sh{L}_3$ a sequence of type
$F_1(s - 1)$ and by right mutation a sequence of type $F_1(s + 1)$ and we can conclude that
the exceptional sequences of invertible sheaves of type $F_1(s)$ can be transformed into each other
via mutations (note 2 facts: 1. that this is strictly true only up to overall twist by an invertible
sheaf; 2. $Y(\mathbf{L}) \simeq \mathbf{F}_{2s}$, which due to the type of the intersection
form is the only allowed case).
We can argue similarly that exceptional sequences of type $F_2(s)$ can be transformed into each
other via mutations.

Both families intersect in the particular case $s = 0$, where
$Y(\mathbf{L}) \simeq \mathbf{P}^1 \times \mathbf{P}^1$. Then, starting with an exceptional sequence
of invertible sheaves representing $F_1(0) = F_2(0)$, we can produce both families by either mutating
the pair $\sh{L}_2, \sh{L}_3$ (to obtain $F_1(s)$) or $\sh{L}_1, \sh{L}_2$ (for $F_2(s)$). However,
by \cite[Proposition 5.2]{HillePerling11}, for $b > 0$ and $s \neq 0$, the resulting sequences of type
$F_2(s)$ cannot consist entirely of invertible sheaves.
\end{example}

\begin{remark}\label{relaxconditions}
Recent work by Vial \cite{Vial15} suggests that it should be possible to relax the conditions on our
surface $X$ by dropping the requirement that $K_X^2 = 12 - n$. We will give a brief outline on how this
condition can be dropped in many cases, though we are not able to remove it completely.

First, we point
out that Corollary \ref{rankonezero} can indeed be proved without $K_X^2 = 12 - n$. An analysis of the
proof shows that this condition enters only in Proposition \ref{hirzebruchnodefects} in order to
derive the integrality of $\frac{1}{e_j^2} + \frac{1}{e_{j + 1}^2}$ (and thus $e_j^2 = e_{j + 1}^2 = 1$).
Without the condition, we get only that $4\left(\frac{1}{e_j^2} + \frac{1}{e_{j + 1}^2}\right)$ is integral,
so that we have, up to order, the additional possibilities $e_j^2 = e_{j + 1}^2 = 4$ and $e_j^2 = 1, e_{j + 1}^2 = 4$.
We can exclude the first case right away, because necessarily $\gcd\{e_1, e_2, e_3, e_4\} = 1$.
For second case, we can assume without loss of generality that $e_1^2 = e_2^2 = 1, e_3^2 = e_4^2 = 4$
and $\tilde{l}_4 = -\tilde{l}_2$. Moreover, $\tilde{l}_1$ and $\tilde{l}_2$ form a basis
and we can assume without loss of generality that $\tilde{l}_4 = -4\tilde{l}_1 + (4k + 1) \tilde{l}_2$ for
some integer $k$. By successive mutation of the pair $\sh{E}_3, \sh{E}_4$, we can arrange that $k = 0$
(see Section \ref{mutationsection}). Then the vectors $\tilde{l}_1, \tilde{l}_2, \tilde{l}_3$ are in
a convex configuration with $a_2 < 0$ in the sense of Lemma \ref{concavemutation}. Then by mutating
the pair $\sh{E}_2, \sh{E}_3$, we obtain $\sh{Z}_1, \dots, \sh{Z}_{n - 4}, \sh{E}_1, L_{\sh{E}_2} \sh{E}_3, \sh{E}_2, \sh{E}_4$,
where we compute $\rk  L_{\sh{E}_2} \sh{E}_3 = 0$, which puts us into the position of Proposition
\ref{p2nodefects}. By the first displayed formula in the proof
of Proposition \ref{hirzebruchnodefects}, we have:
$$
K_X^2 = 9 - n - 4 \sum_{i = 1}^{n - 4} \delta_i (\delta_i + 1) \equiv 1 - n \quad (8).
$$
However, for the case $n - t = 3$ we have by the formula in the proof
of Proposition \ref{p2nodefects}:
$$
K_X^2 = 12 - n -  4 \sum_{i = 1}^{n - 3} \delta_i (\delta_i + 1) \equiv 4 - n  \quad(8).
$$
With this contradiction, we can exclude the case $e_j^2 = 1, e_{j + 1}^2 = 4$ and conclude that Corollary
\ref{rankonezero} indeed holds without the assumption $K_X^2 = 12 - n$.

Second, with the previous remarks we have established that
$$
4 \sum_i \delta_i (\delta_i + 1) = (10 - K_X^2) - \rk \chnum^1(X),
$$
in particular both sides of this equation are divisible by $8$ (this was pointed out to me by C. Vial).
Hence, unless $(10 - K_X^2) - \rk \chnum^1(X)$ is nonzero and divisible by $8$, the condition $\chi(\sh{O}_X) = 1$
and the existence of an exceptional sequence of maximal length already imply $K_X^2 = 12 - n$.
All our main results then remain true under these weaker assumptions.
\end{remark}

\section{The main theorem}\label{maintheorem}

We are now in possession of everything we need in order to show our main theorem.

\begin{theorem}\label{maintheorem1}
Let $X$ be a numerically rational surface with $\rk \knum(X) = n$ and
$\mathbf{E} = \sh{E}_1,\dots, \sh{E}_n$
a numerically exceptional sequence. Then to the maximal sub-sequence of objects of nonzero rank
$\sh{E}_{i_1}, \dots, \sh{E}_{i_{n - t}}$ there is associated in
a canonical way a complete toric surface $Y(\mathbf{E})$ with $n - t$ torus fixpoints. These fixpoints
are either smooth (if $e_{i_j}^2 = 1$) or $T$-singularities of type
$\frac{1}{e_{i_j}^2}(1, k_{i_j} e_{i_j} - 1)$ with $\gcd\{k_{i_j}, e_{i_j}\} = 1$.
\end{theorem}

\begin{remark}
To justify ``canonical'', we should, strictly speaking, also incorporate the multiplicities of the
rays of $Y(\mathbf{E})$ coming from the rank zero objects. However, in light of our discussion in
Section \ref{rankzeromoving}, at least on the combinatorial level
there seems not much to be gained from this.
\end{remark}

In absence of rank zero objects, we can state Theorem \ref{maintheorem1} in a more convenient
form.

\begin{theorem}\label{maintheorem2}
Let $X$ be a numerically rational surface and let $\mathbf{E} = \sh{E}_1,\dots, \sh{E}_n$
be a numerically exceptional sequence of maximal length such that $(\rk \sh{E}_i)^2 = e_i^2 > 0$ for
every $i$. Then to this sequence there is associated in a canonical way a complete toric surface
$Y(\mathbf{E})$
with $n$ torus fixpoints which are either smooth (if $e_i^2 = 1$) or $T$-singularities of type
$\frac{1}{e_i^2}(1, k_i e_i - 1)$ with $\gcd\{k_i, e_i\} = 1$.
Moreover, this correspondence induces a natural isomorphism of Chow rings
$\Ch^*(Y(\mathbf{E}))_\Q \rightarrow \chnum^*(X)_\Q$ which maps $K_{Y(\mathbf{E})}$ to $K_X$.
\end{theorem}

\begin{remark}\label{divisorcorrespondence}
Note that the ring isomorphism exists by construction and is given by mapping the class of the
$i$-th torus invariant prime divisor $D_i$ of $Y(\mathbf{E})$ to $A_i$ of the associated toric system
(see \ref{completecase}).
It follows that $K_{Y(\mathbf{E})} = -\sum_{i = 1}^n D_i$ maps to $K_X = -\sum_{i = 1}^n A_i$.
\end{remark}

\begin{proof}[Proof of Theorem \ref{maintheorem1}]
By Proposition \ref{winding}, we have constructed our toric variety as the Gale duals of the
contracted toric system. It only remains to show that this surface indeed can only have
at most $T$-singularities. Consider the vectors $l_1, \dots, l_{n - t}$ generating the fan of
$Y(\mathbf{E})$ and their corresponding $n - t$ circumference segments $p_i = l_i - l_{i - 1}$. For
any given $i$ with $e_i^2 > 1$, with a convenient choice of coordinates we can represent the vectors
$l_{i - 1}, l_i$ as $l_{i - 1} = (1, 0)$ and $l_i = (-x, e_i^2)$ for some $0 < x < e_i^2$. We have
seen in Lemma \ref{peri1} that the vector $w_i = \frac{1}{e_i} p_i$ is integral, hence we get
$x = k_i e_i - 1$ for some $1 \leq k_i \leq e_i$. So, by Lemma \ref{peri0}, in order to show that
$\gcd\{k_i, e_i\} = 1$ it suffices to show that $w_i$ is primitive.
By Lemmas \ref{peri1} and \ref{peri3} we know that any left or right mutation
of a pair $\sh{E}_i, \sh{E}_{i + 1}$ with $e_i, e_{i + 1} \neq 0$ which results in a pair
$\sh{E}_i', \sh{E}_{i + 1}'$ such that $e_i', e_{i + 1}' \neq 0$
leaves the lattice lengths of the involved $w_i$ invariant. By Theorem \ref{rankone} we can
transform any sequence by mutation into a sequence of length $n$ of objects of rank one,
where obviously the $w_j$ are primitive for every $1 \leq j \leq n$. However, in this process we
might destroy or create new circumference segments by creating or destroying objects of rank zero,
and it remains to check whether we obtain non-primitive $w_j$ this way. By Theorem \ref{deltatheorem}
and Lemma
\ref{zeromutationconfiguration}, for every pair $\sh{E}_i, \sh{E}_{i + 1}$ with
$\rk L_{\sh{E}_i} \sh{E}_{i + 1} = 0$ we have the relation $e_i^4 l_{i - 1} - e_i^2 l_i +
e_i^2 l_{i + 1} = 0$, respectively we have $l_i = e_i^2 l_{i - 1} + l_{i + 1}$. Then mapping
$l_i$ to $l_{i - 1}$ extends to a linear map on $N$ which leaves $l_{i - 1}$ invariant and thus
the  pairs $l_{i - 1}, l_i$ and $l_{i - 1}, l_{i + 1}$ both correspond to cones whose associated
affine toric surfaces are isomorphic, hence $l_{i + 1} - l_{i - 1}$ has lattice length $e_i$ iff
$l_i - l_{i - 1}$ has. The difference $l_{i + 1} - l_i = e^2 l_{i - 1}$ clearly corresponds to
the circumference segment of a $T$-singularity of order $e_i^4$. This concludes the proof of the theorem.
\end{proof}

\begin{sub}
We conclude this section with an observation on cyclic strongly exceptional sequences. In
\cite[Theorems 5.13, 5.14]{HillePerling11} it was shown that any rational surface which admits a
cyclic strongly exceptional sequence necessarily has Picard-rank at most $7$ and that every del Pezzo
surface meeting this conditions indeed does admit a cyclic strongly exceptional sequence of invertible
sheaves. More generally, in \cite{vandenBergh02} van den Bergh constructed cyclic strongly exceptional
sequences for all del Pezzo surfaces, which for the Picard-ranks 8 and 9 cannot consist
of line bundles only. The crucial observation in \cite{HillePerling11} was the fact that for the rank
one case the associated toric surface must be a weak del Pezzo surface or,
equivalently, that the fan must be convex. This is easily seen to be true also in the general case,
as by \ref{circumferenceformulas}, $\chi(\sh{E}_i, \sh{E}_{i + 1}) \geq 0$ implies the convexity of the
triple $l_{i - 1}, l_i, l_{i + 1}$. The following implication for the Picard-rank then is quite
straightforward. Note that $K_X^2 > 0$ implies that $\rk \chnum^1(X) < 10$.
\end{sub}

\begin{theorem}
Let $X$ be a numerically rational surface.
Then the length of a cyclic strongly exceptional sequence on $X$ is at most $11$. In particular,
if $X$ admits a cyclic strongly exceptional sequence of maximal length, then
$K_X^2 > 0$.
\end{theorem}

\begin{proof}
By Proposition \ref{rankzerosequences}, we can choose $\sh{E}_i, \sh{E}_j$ with
$e_i e_j > 0$ and $i < j$. From strong cyclicity it follows that $\chi(\sh{E}_i, \sh{E}_j),
\chi(\sh{E}_j, \sh{E}_i \otimes \omega_X^{-1}) \geq 0$. 
From the fact that the fan for $Y(\mathbf{E})$ necessarily has a convex configuration of primitive vectors
it follows that at least one inequality is strict, where up to
cyclic renumbering we can assume without loss of generality that $\chi(\sh{E}_i, \sh{E}_j)
> 0$. By Lemma \ref{paircommute}, we have $\chi(\sh{E}_j, \sh{E}_i \otimes \omega_X^{-1}) =
e_i e_j K_X^2 - \chi(\sh{E}_i, \sh{E}_j) \geq 0$, hence
\begin{equation*}
e_i e_j K_X^2 \geq \chi(\sh{E}_i, \sh{E}_j) > 0.
\end{equation*}
from which our assertion follows.
\end{proof}

\section{Some remarks on the singularities}\label{singularitiessection}

Consider a numerically exceptional sequence $\mathbf{E} = \sh{E}_1, \dots, \sh{E}_n$, where for
simplicity we will assume that $e_i > 0$ for all $i$. Then by the classification of Section
\ref{maintheorem}, $Y(\mathbf{E})$ belongs to a very nice class of toric surfaces with the
following properties:
\begin{enumerate}[1)]
\item $\det(l_i, l_{i + 1}) = e_i^2$ for every $i$.
\item The circumference segments $p_i = l_i - l_{i - 1}$ have lattice length $e_i$ for every $i$.
\item $K_{Y(\mathbf{E})}^2 = K_X^2 = 12 - n$.
\end{enumerate}
For $K_{Y(\mathbf{E})}^2$, we have by \ref{ksquare}, \ref{eulersym}, and \ref{circumferenceformulas}
the formulas (where $q_i = p_i / e_i^2$ in the notation of \ref{ksquare}):
\begin{equation*}
K_{Y(\mathbf{E})}^2 = \sum_{i = 1}^n \det(q_i q_{i + 1}) = \sum_{i = 1}^n \frac{1}{e_i e_{i + 1}}
\chi(\sh{E}_i, \sh{E}_{i + 1}) = \sum_{i = 1}^n
\frac{c_1(\sh{E}_i, \sh{E}_{i + 1})^2 + e_i^2 + e_{i + 1}^2}{e_i^2 e_{i + 1}^2}.
\end{equation*}

As we have already stated in the introduction, this kind of surface can be
classified in terms of mutations, for instance by starting from the fan of a smooth toric surface.
However, a classification of such surfaces independent of exceptional sequences might be of interest
as well (e.g. one could relax property 3) above and just require that $K_{Y(\mathbf{E})}^2$ be
integral). This is beyond the scope of this article, but with this perspective we want
conclude with some general remarks on the singularities which occur in our setting.

Cyclic $T$-singularities of the form $\frac{de^2}{kde- 1}$, where $e, k, d > 0$ and $\gcd\{k, e\} = 1 $, have been
classified Kollar and Shepherd-Barron. Their statement is:

\begin{proposition}{\cite[Proposition 3.11]{KollarShepherdbarron88}}\label{tsing}
A cyclic singularity is of class $T$ if and only if its continued fraction expansion
$[b_1, \dots, b_r]$ is of one of the following forms:
\begin{enumerate}[(i)]
\item\label{tsingi} $[4]$ and $[3, 2, \dots, 2, 3]$ are of class T,
\item\label{tsingii} If $[b_1, \dots, b_r]$ is of class $T$ then so are
$[2, b_1, \dots, b_r + 1]$ and $[b_1 + 1, \dots, b_r, 2]$.
\item\label{tsingiii} Every singularity of class T that is not a rational double point
can be obtained by starting with one of the singularities described in (\ref{tsingi}) and
iterating the steps described in (\ref{tsingii}).
\end{enumerate}
\end{proposition}

\begin{sub}\label{resolutionshearterm}
It follows from the analysis in the proof of \cite[Proposition 3.11]{KollarShepherdbarron88} that
$d = 1$ implies that the $T$-singularities of type $\frac{e^2}{ke - 1}$ are precisely those obtained
from $[4]$ and iterating step (\ref{tsingii}). Consider the minimal desingularization
of $Z \rightarrow Y = Y(\mathbf{E})$. Denote $I := \{1 \leq k \leq n \mid e_k^2 \neq 1\}$, then
the fan of $Z$ is obtained from $Y$ by inserting $t_k$ new primitive vectors corresponding to
the continued fractions $[b_1^k, \dots, b_{t_k}^k]$. We have
\begin{equation*}
K_Z^2 = \sum_{i = 1}^n a_i + 2n - \sum_{k \in I} (\sum_{j = 1}^{t_k} b_j^k - 2t_i) = 12 - n -
\sum_{k \in I} t_k,
\end{equation*}
where the $a_i$ are the self-intersection of the pullbacks of the original $n$ divisors to the
surface $Z$.
A simple induction using Proposition \ref{tsing} shows that $\sum_{j = 1}^{t_k} b_j^k = 1 + 3t_k$
for every $k \in I$, hence
\begin{equation*}
\sum_{i = 1}^n a_i = 12 - 3n + \vert I \vert.
\end{equation*}
This is almost the same formula as for the sum of self-intersection numbers of the toric prime divisors
on a smooth toric surface with $n$ rays, except that we obtain the number of singularities as an extra
term.
\end{sub}

\begin{sub}\label{shearterm}
Consider $e, f > 1$ and a triple of primitive lattice vectors $l_e, l, l_f$ such that the cones
generated by $l_e, l$ and $l, l_f$, respectively, correspond to $T$-singularities of types
$\frac{1}{e^2}(1, \alpha e - 1)$ and $\frac{1}{f^2}(1, \beta f - 1)$, respectively.
We can choose coordinates such that $l = (0, 1)$ and $l_e = (e^2, -\alpha e + 1 + \lambda_1 e^2)$,
$l_f = (-f^2, -\beta f + 1 + \lambda_2 f^2)$. It is easy to verify that the term $\lambda := \lambda_1
+ \lambda_2$ does not depend on any choice of coordinates with $l = (0, 1)$.
Using \ref{circumferenceformulas}, we compute:
\begin{equation*}
\det(l_f, l_e) = e^2 f^2 (\frac{\alpha}{e} + \frac{\beta}{f} + \lambda) - e^2 - f^2 =
e^2 f^2 \frac{\chi(\sh{E}, \sh{F})}{e f} - e^2 - f^2,
\end{equation*}
hence $\frac{1}{ef}\chi(\sh{E}, \sh{F}) = \frac{\alpha}{e} + \frac{\beta}{f} + \lambda$.
If $e = 1, f > 1$ (or $e > 1, f = 1$ or $e = f = 1$, respectively), then the same calculation yields
$\frac{1}{f}\chi(\sh{E}, \sh{F}) = 1 + \frac{\beta}{f} + \lambda$ (respectively
$\frac{1}{e}\chi(\sh{E}, \sh{F}) = \frac{\alpha}{e} + 1 + \lambda$ and
$\chi(\sh{E}, \sh{F}) = 2 + \lambda$). In the case $e = f = 1$, $\lambda$ coincides with the
self-intersection number of the toric divisor associated to $l$.
\end{sub}

\begin{sub}
Now, if we transfer above considerations to $Y(\mathbf{E})$, we have in terms of reduced circumference
segments for every $1 \leq i \leq n$ that $\frac{1}{e_i e_{i + 1}}\chi(\sh{E}_i, \sh{E}_{i + 1})
= \det(q_i, q_{i + 1})$, where
\begin{equation*}
\det(q_i, q_{i + 1}) =
\begin{cases}
\frac{\alpha_i}{e_i} + \frac{\beta_i}{e_{i + 1}} + \lambda_i & \text{ if } e_i, e_{i + 1} > 1\\
1 + \frac{\beta_i}{e_{i + 1}} + \lambda_i & \text{ if } e_i = 1, e_{i + 1} > 1\\
\frac{\alpha_i}{e_i} + 1 + \lambda_i & \text{ if } e_i > 1, e_{i + 1} = 1\\
2 + \lambda_i \text{ if } e_i^2, e_{i + 1}^2 = 1.
\end{cases}
\end{equation*}
It is a consequence of local change of coordinates that $\alpha_i = e_i - \beta_{i - 1}$ for every $i$
and in particular, for any $i$ with $e_i > 1$ we have $\frac{\alpha_i}{e_i} + \frac{\beta_{i - 1}}{e_i}
= 1$. Thus we get:
\begin{equation*}
K_{Y(\mathbf{E})}^2 = 12 - n = \sum_{i = 1}^n \det(q_i, q_{i + 1}) = \sum_{i = 1}^n \lambda_i +
2 n - \vert I \vert,
\end{equation*}
where $I = \{1 \leq i \leq n \mid e_i > 1\}$ as in \ref{resolutionshearterm}, hence
\begin{equation*}
\sum_{i = 1}^n \lambda_i = 12 - 3n + \vert I \vert.
\end{equation*}
That is, the sum of the $\lambda_i$ coincides with the sum of the $a_i$ from \ref{resolutionshearterm}.
Using mutation, one can trace the $a_i$ and the $\lambda_i$ to see that indeed $a_i = \lambda_i$
for every $i$, though we will leave the verification to the reader.
\end{sub}

\section*{Appendix: Toric surfaces}

\setcounter{theorem}{0}
\renewcommand{\thesection}{A}

A toric surface is a normal algebraic surface $X$ on which an algebraic torus $T \simeq
\mathbb{G}_m(\K)^2$
acts such that $T$ embeds into $X$ as an open dense orbit and the group action extends the
multiplication on $T$. In this appendix we want to remind the reader on standard material on
toric surfaces as it can be found in standard textbooks such as \cite{CoxLittleSchenck}.
However, we will need to paraphrase some of the material in order to suit its use in the main text.

Let us denote $M \simeq \Z^2$ the character group of $X$ and $N$ its dual module. We denote $M_\Q :=
M \otimes_\Z \Q$ and $N_\Q := N \otimes_\Z \Q$. The complement of $T$ in $X$ (if nonempty) is given
by a union of divisors $D_1 \cup \dots \cup D_n$. We are interested in three cases:
\begin{enumerate}
\item $X$ is affine and has a fixpoint with respect to the torus action.
\item $X$ can be covered by two affine toric varieties and has two fixpoints.
\item $X$ is complete.
\end{enumerate}

Each of these cases is completely determined by a collection of primitive lattice vectors
$l_1, \dots, l_n$ in $N$. In every case, we will write $L$ for the matrix whose rows are given
by the $l_i$ which we will interpret as $\Z$-linear map from $M$ to $\Z^n$.

\begin{sub}[\bf The affine case]\label{affinetoricsurfaces}
In the affine case, we have $n = 2$ and the $T$-fixpoint is $D_1 \cap D_2$, where $D_1$ and $D_2$
both are isomorphic to $\mathbb{A}_\K^1$. The vectors $l_1$ and $l_2$ generate over $\Q_{\geq 0}$ a strictly
convex polyhedral cone in $N_\Q$. In general, the fixpoint of $X$ is a quotient singularity,
i.e. $X \simeq \mathbb{A}_\K^2 / \mu_v$, where $\mu_v = \spec \K[G]$ the abelian group scheme over $\K$
corresponding to the cyclic group $G \simeq N / (\Z l_1 + \Z l_2) \simeq \Z / v \Z$.

As $l_1$ and $l_2$ are primitive, one can choose a suitable basis for $N$ such that
$l_1 = (1, 0)$ and $l_2 = (-k, v)$, where $\gcd\{k, v\} = 1$ and $0 < k < v$. In the case
that $\K$ is algebraically closed, and $\characteristic \K$ and
$v$ are coprime, this yields a customary representation for the action of $\mu_v$ on $\mathbb{A}_\K^2$
as $\diag(\eta, \eta^k)$, where $\eta \in \K$ is a $v$-th root of unity. We also use the
notation $\frac{1}{v}(1, k)$ to denote a toric singularity which up to choice of coordinates
can be represented in this way.
Note that $v = \det(l_1, l_2)$, which equals the lattice volume of the parallelogram spanned by
$l_1$ and $l_2$. We will therefore very often refer to $\det(l_1, l_2)$ as the {\em volume} of
$l_1$ and $l_2$.
A useful invariant for us will be the lattice vector $l_2 - l_1$:

\begin{definition}\label{circumferencesegmentdef}
We call $l_2 - l_1$ the {\em circumference segment} of the pair $l_1, l_2$. We call
$\frac{1}{\det(l_1, l_2)}(l_2 - l_1)$ the {\em reduced} circumference segment.
\end{definition}

Note that the circumference segment $l_2 - l_1$ in general is not primitive and the reduced segment
in general is not integral.

\begin{lemma}\label{peri0}
Assume that $\mu_v$ acts on $\mathbb{A}_\K^2$ with weights $\frac{1}{v}(1, k)$. Then the lattice length of
$l_2 - l_1$ is $\gcd\{v, k + 1\}$.
\end{lemma}

\begin{proof}
With above coordinate representation, we get $l_2 - l_1 = (-k - 1, v) = g ((-k - 1) / g, v/g) =: g P$,
where $g = \gcd\{k + 1, v\}$ and $P$ is a primitive lattice vector.
\end{proof}
\end{sub}

\begin{sub}[\bf The minimal linearly dependent case]\label{minlindep}
Let $l_1, l, l_2$ be three primitive lattice vectors such that $l_1$ and $l_2$ lie in opposite half
spaces with respect to the line $\Q l$. Then $l_1, l$ and $l, l_2$ generate strictly convex polyhedral
cones in $N_\Q$ which intersect in the common facet $\Q_{\geq 0} l$ and therefore generate a fan
corresponding to a two-dimensional toric surface which is covered by two affine toric surfaces each of
which contains a single torus fixpoint. We choose an orientation on $N_\Q$ such that $l_1, l, l_2$
are ordered counter-clockwise. Then they satisfy a relation
\begin{equation*}
w l_1 + a l + v l_2 = 0,
\end{equation*}
where $v = \det(l_1, l)$, $w = \det(l, l_2)$, $a = \det(l_2, l_1)$, and $v, w > 0$. This relation is unique up
to a common scalar multiple of the coefficients. If $a = 0$, then $v = w$ and $l_1 = -l_2$.
We have the two circumference segments $p_1 = l - l_1$ and $p_2 = l_2 - l$ and the corresponding reduced
circumference segments $q_1 = \frac{1}{v} p_1$ and $q_2 = \frac{1}{w} p_2$. We observe:
\begin{equation*}
\det(p_1, p_2) = \det(l - l_1, l_2 - l) = a + v + w, \qquad \det(q_1, q_2) = \frac{1}{vw}(a + v + w).
\end{equation*}
It follows that the sum $a + v + w$ determines the convexity of the configuration of lattice vectors.
That is, $l$ is not contained in the convex hull of $l_1, l_2$ and the origin if and only if
$a + v + w > 0$. The vectors $l_1, l, l_2$ lie on a line in $N_\Q$ if and only if $a + v + w = 0$.
\end{sub}

\begin{lemma}\label{triplelemma}
Let $a_1, a_3 > 0$ and $a_2$ be integers and $l_1, l_2, l_3 \in N$ primitive such that $a_{\pi(1)}
= (\operatorname{sgn} \pi) \cdot \det(l_{\pi(2)}, l_{\pi(3)})$ for any permutation $\pi \in
S_3$. Then
\begin{enumerate}[(i)]
\item\label{triplelemmai} $a_1 l_1 + a_2 l_2 + a_3 l_3 = 0$,
\item\label{triplelemmaii} $\gcd\{a_i, a_j\} = \gcd\{a_1, a_2, a_3\}$ for any $1 \leq i \neq j \leq 3$,
\item\label{triplelemmaiii}
The $l_i$ are uniquely determined up to a transformation by $\operatorname{GL}_\Z(N)$.
\end{enumerate}
\end{lemma}

\begin{proof}
We have seen (\ref{triplelemmai}) already in \ref{minlindep}.

(\ref{triplelemmaii})
Without loss of generality we restrict to the case $i = 1$, $j = 2$. Clearly, $\gcd\{a_1, a_2, a_3\}$ divides
$\gcd\{a_1,a_2\} =: g$. Now we write $a_1 l_1 + a_2 l_2 = -a_3 l_3$. Clearly, $g$ divides the left hand
side and hence the right hand side. But $l_3$ is primitive, so $g$ cannot divide $l_3$, hence
$g$ divides $a_3$ and thus $g$ divides $\gcd\{a_1, a_2, a_3\}$.

(\ref{triplelemmaiii})
Up to a choice of basis for $N$, we can write $l_1 = (1, 0)$, $l_2 = (x, a_3)$, $l_3 = (y, -a_2)$,
where $-a_3 < x < 0$ and $y$ is determined by the relation $a_1 + a_2 x + a_3 y = 0$. We set
$a_i =: a'_i g$, where $g := \gcd\{a_1, a_2, a_3\}$. Then the $a'_i$ are pairwise coprime by
(\ref{triplelemmaii}) and equation $a'_1 + a'_2 x + a'_3 y = 0$ holds as well. Then the set of
solutions $(x, y)$ is given by the set $(x'_0, y'_0) + k (-a'_3, a'_2)$, where $k \in \Z$ and
$(x_0, y_0)$ is some special solution. Multiplying by $g$, the solutions are given by the set
$(x_0, y_0) + k(-a_3, a_2)$ or $k \in \Z$ and $(x_0, y_0)$ some special solution. It follows that
the condition $-a_3 < x < 0$ determines $x$ (and therefore $y$) uniquely.
\end{proof}

\begin{sub}[\bf The complete case]\label{completecase}
In this case the $D_i$
form a cycle of $\mathbb{P}^1$'s and there are $n$ torus fixpoints which are given by the
intersections $D_i \cap D_{i + 1}$. Here it is customary to consider in this case the integers
$1, \dots, n$ as a system of representatives for $\Z / n \Z$, i.e. we read the indices modulo $n$.
The Chow group of $X$ is determined by the following standard short exact sequence
\begin{equation*}
0 \longrightarrow M \overset{L}{\longrightarrow} \Z^n \longrightarrow \Ch^1(X) \longrightarrow 0,
\end{equation*}
such that the $i$-th standard basis vector of $\Z^n$ maps to the class of $D_i$ in $\Ch^1(X)$.
The canonical divisor on $X$ can be represented by $K_X = -\sum_{i = 1}^n D_i$.
On $\Ch^1(X)$, there exists a $\Q$-valued intersection product which is completely
determined by triple relations as in \ref{minlindep}. That is, if for every $i$ we denote
$v_i := \det(l_{i - 1}, l_i)$ and $a_i := \det(l_{i + 1}, l_{i - 1})$, then we have for every
$i$ the following relation:
\begin{equation*}
v_{i + 1} l_{i - 1} + a_i l_i + v_i l_{i + 1} = 0, \quad \text{ respectively } \quad
\frac{1}{v_i} l_{i - 1} + \frac{a_i}{v_i v_{i + 1}} l_i + \frac{1}{v_{i + 1}}l_{i + 1} = 0,
\end{equation*}
which translates to intersection products:
\begin{equation*}
D_{i - 1} D_i = \frac{1}{v_i}, \quad D_i^2 = \frac{a_i}{v_i v_{i + 1}}, \quad
D_i D_{i + 1} = \frac{1}{v_{i + 1}} \qquad \text{ for every } i
\end{equation*}
and $D_i D_j = 0$ else. Note that we choose an orientation on $N_\Q$ such that
$v_i > 0$ for every $i$.

For any $i$, we have a circumference segments $p_i := l_i - l_{i - 1}$ and its reduction $q_i :=
\frac{p_i}{v_i}$. As in \ref{minlindep}, we have equations
\begin{equation*}
\det(p_i, p_{i + 1}) = a_i + v_i + v_{i + 1} \quad \text{ and } \quad
\det(q_i, q_{i + 1}) = \frac{1}{v_i v_{i + 1}}(a_i + v_i + v_{i + 1}) =
D_i (D_{i - 1} + D_i + D_{i + 1})
\end{equation*}
for every $i$.
\end{sub}

\begin{lemma}\label{ksquare}
We have $K_X^2 = \sum_{i = 1}^n \det(q_i, q_{i + 1}) =
\sum_{i = 1}^n \frac{a_i + v_i + v_{i + 1}}{v_i v_{i + 1}}$.
\end{lemma}

\begin{proof}
By above discussion and $K_X = -\sum_{i = 1}^n D_i$.
\end{proof}

Note that in the case that $X$ is smooth, we have $v_i = 1$ for every $i$ and this formula specializes
to $K_X^2 = 12 - n = 2n + \sum_{i = 1}^n a_i$ or equivalently, $\sum_{i = 1}^n a_i = 12 - 3n$.

\begin{sub}\label{desingularization}
The minimal resolution of a toric surface singularity can be described with help of Hirzebruch-Jung
continued fractions. That is, in the situation of \ref{affinetoricsurfaces}, we have
$l_1 = (1, 0)$ and $l_2 = (-k, v)$ and we write the quotient $\frac{v}{k} = [b_1, \dots, b_r] :=
b_1 - 1 / (b_2 - 1 / (b_3 - \cdots / (b_{r - 1} - 1 / b_r) \cdots ))$
and the $b_i$ are integers $\geq 2$.
Now, in the first step of the resolution, we introduce a new primitive vector $(0, 1)$ which
subdivides the cone into smaller cones of volumes $1$ and $k$, respectively. Ultimately, this
new vector will correspond to a component with self-intersection number $-b_1$ of the the exceptional
divisor in the minimal resolution. After this first step we see that the fraction $\frac{v}{k}$ can be
interpreted as the ratio of the volume of the original cone and the
volume of the (possibly still) singular cone after the first resolution step.
Iterating this, we see that the continued fraction $[b_i, \dots, b_r]$ equals the ratio of volumes
$V_{i - 1} / V_i$, where $V_i$ is the volume of the remaining singular cone after the $i$-th
resolution step (we set $V_0 := v$ and $V_1 := k$, and it follows that $V_r = 1$).
\end{sub}

We now are interested in the behaviour of the canonical self-intersection number under minimal
resolutions. That is, we consider a complete toric surface $X$ which has a singular point and
$Y \rightarrow X$ a toric morphism which is the result of $s \leq r$ steps of the minimal
desingularization procedure for this point as described in \ref{desingularization}.

\begin{proposition}\label{ksquaredesing}
With notation of \ref{desingularization}, we have $K_X^2 - K_Y^2 = \sum_{i = 1}^s
\frac{(V_i - V_{i - 1} + 1)^2}{V_{i - 1} V_i}$.
\end{proposition}

\begin{proof}
We consider the very first step of of the resolution, i.e. we factor $Y = Z_s \rightarrow
\cdots \rightarrow Z_1 \rightarrow Z_0 = X$, where every map $Z_i \rightarrow Z_{i - 1}$ is a partial
resolution step which precisely adds one more toric divisor according to the procedure of
\ref{desingularization}. For the first step, we use the presentation of $K_X^2$ and $K_{Z_1}^2$
of Lemma \ref{ksquare}. Up to cyclic renumbering and
a choice of basis of $N$, we can assume that the singular point on $X$ being resolved is represented
by the cone generated by $l_2 = (1, 0)$ and $l_3 = (-k, v_2)$ and one primitive vector
$l = (0, 1)$ is added as described in \ref{desingularization}. Using the relations $v_{i + 1} l_{i - 1}
+ a_i l_i + v_i l_{i + 1} = 0$, we compute $l_1 = (\frac{v_2 k - a_2}{v_3}, -v_2)$ and
$l_4 = (\frac{a_3 k - v_4}{v_3}, -a_3)$ (note that these will coincide if $n = 3$).
Then we immediately obtain the relations
\begin{equation*}
l_1 + a l_2 + v_2 l = 0 \ \text{ with } \ a = \frac{a_2 - v_2 k}{v_3} \quad \text{ and } \quad
v_4 l + b l_3 + k l_4 = 0 \ \text{ with } \ b = \frac{a_3 k - v_4}{v_3}.
\end{equation*}
Now, the effect of the canonical self-intersection number is local in the sense that, if we write
$K_X^2 =: R + \frac{a_2 + v_2 + v_3}{v_2 v_3} + \frac{a_3 + v_3 + v_4}{v_3 v_4}$, then
\begin{align*}
K_{Z_1}^2 - R & = \frac{a + v_2 + 1}{v_2} + \frac{-v_2 + 1 + k}{k} + \frac{b + k + v_4}{k v_4}\\
& = \frac{a_2 + v_2 + v_3}{v_2 v_3} + \frac{a_3 + v_3 + v_4}{v_3 v_4} + 2 + \frac{2}{k} -
\frac{k}{v_3} - \frac{v_3}{k} - \frac{1}{v_3 k} - \frac{2}{v_3}\\
& = K_X^2 - R - \frac{(k - v_3 + 1)^2}{k v_3}.
\end{align*}
Iterating this procedure implies the assertion.
\end{proof}

It is useful to understand also how the terms $\det(q_i, q_{i + 1})$ behave under resolution of
singularities.

\begin{corollary}\label{ksquaredesing2}
Let $l_1, l_2, l_3$ and $w, a, v$ as in \ref{minlindep} such that $l_1$ and $l_2$ generate a singular
cone and assume that we have a minimal resolution corresponding to a continued fraction
$[b_1, \dots, b_r]$. If we denote $v =: V_0, V_1, \dots V_s$ for $s \leq r$ the associated volumes
and $\xi_1, \dots, \xi_s$ the successively added primitive vectors, then we have a relation
\begin{equation*}
w \xi_s + b l_2 + V_s l_3 = 0 \quad \text{ such that } \quad \frac{a + v + w}{vw} -
\frac{b + V_s + w}{V_s w} = \sum_{i = 1}^s \frac{V_i - V_{i - 1} + 1}{V_{i - 1} V_i}.
\end{equation*}
\end{corollary}

\begin{proof}
As in the proof of \ref{ksquaredesing} it is enough to consider one iteration step. We have seen that
we can write $l_1 = (1, 0)$, $l_2 = (-V_1, w)$, $l_3 = (y, -a)$ and thus we obtain the equation
$y = \frac{1}{v}(a V_1 - w)$. It follows that for $\xi_1 = (0, 1)$ we have a relation $w \xi_1
+ b l_2 + V_1 l_3$, where $b := y$. By plugging in $b$, we get
$\frac{a + V_0 + w}{V_0 w} - \frac{b + V_1 + w}{V_1 w} = \frac{V_1 - V_0 + 1}{V_0 V_1}$.
\end{proof}


\begin{thebibliography}{BGKS12}


\bibitem[Aig13]{Aigner13}
M.~Aigner.
\newblock {\em {Markov's theorem and 100 years of the uniqueness conjecture}}.
\newblock Springer, 2013.

\bibitem[BGKS15]{BBKS12}
C.~B\"ohning, H.-C. {Graf von Bothmer}, L.~Katzarkov, and P.~Sosna.
\newblock {Determinantal Barlow surfaces and phantom categories}
\newblock{\em J. Eur. Math. Soc.}, 17(7): 1569--1592, 2012.

\bibitem[Bon90]{Bondal90}
A.~I. Bondal.
\newblock {Representation of associative algebras and coherent sheaves}.
\newblock {\em Math. USSR Izvestiya}, 34(1):23--42, 1990.

\bibitem[CLS11]{CoxLittleSchenck}
D.~A. Cox, J.~B. Little, and H.~K. Schenck.
\newblock {\em {Toric varieties}}, volume 124 of {\em Graduate Studies in
  Mathematics}.
\newblock American Mathematical Society, Providence, RI, 2011.

\bibitem[DL85]{DrezetLePotier}
J.~M. Drezet and J.~{Le Potier}.
\newblock {Fibr\'es stables et fibr\'es exceptionnels sur $\mathbb{P}_2$}.
\newblock {\em {Ann. scient. \'Ec. Norm. Sup. $4^e$ s\'erie}}, 18:193--244,
  1985.

\bibitem[Hac13]{Hacking13}
P.~Hacking.
\newblock {Exceptional bundles associated to degenerations of surfaces}.
\newblock {\em Duke Math. J.}, 162(6):1171--1202, 2013.

\bibitem[HP10]{HackingProkhorov10}
P.~Hacking and Y.~Prokhorov.
\newblock {Smoothable del Pezzo surfaces with quotient singularities}.
\newblock {\em Compos. Math.}, 146(1):169--192, 2010.

\bibitem[HP11]{HillePerling11}
L.~Hille and M.~Perling.
\newblock {Exceptional sequences of invertible sheaves on rational surfaces}.
\newblock {\em Compos. Math}, 147(4):1230--1280, 2011.

\bibitem[KNP15]{KasprzykNillPrince15}
A.~Kasprzyk, B.~Nill, and T~Prince.
\newblock Minimality and mutation-equivalence of polygons.
\newblock preprint, 2015.
\newblock {\em arxiv:1501.05335v1}.

\bibitem[KS88]{KollarShepherdbarron88}
J.~Koll\'ar and N.~{Shepherd-Barron}.
\newblock {Threefolds and deformations of surface singularities}.
\newblock {\em Invent. Math.}, 91(2):299--338, 1988.

\bibitem[Nog91]{Nogin91}
D.~Y. Nogin.
\newblock {Spirals of period four and equations of Markov type}.
\newblock {\em Math. USSR Izvestiya}, 37(1):209--226, 1991.

\bibitem[Orl93]{Orlov93}
D.~O. Orlov.
\newblock {Projective bundles, monoidal transformations, and derived categories
  of coherent sheaves}.
\newblock {\em Russian Acad. Sci. Ivz. Math.}, 41(1):133--141, 1993.

\bibitem[Rud89]{Rudakov89a}
A.~N. Rudakov.
\newblock {Markov numbers and exceptional bundles on $\mathbb{P}^2$}.
\newblock {\em Math. USSR Izvestiya}, 32(1):99--112, 1989.

\bibitem[Rud90]{Rudakov89}
A.~N. Rudakov.
\newblock {Exceptional vector bundles on a quadric}.
\newblock {\em Math. USSR Izvestiya}, 33(1):115--138, 1990.

\bibitem[Rud90]{Rudakov90}
A.~N. Rudakov.
\newblock {\em {Helices and vector bundles: seminaire Rudakov}}, volume 148 of
  {\em London Mathematical Society lecture note series}.
\newblock Cambridge University Press, 1990.

\bibitem[Tho97]{Thomason97}
R.~W. Thomason.
\newblock {The classification of triangulated subcategories}.
\newblock {\em Compos. Math.}, 105(1):1--27, 1997.

\bibitem[{van}09]{vandenBergh02}
M.~{van den Bergh}.
\newblock Non-commutative crepant resolutions (with some corrections).
\newblock {\em arXiv:math/0211064v2}, 2002/2009.

\bibitem[Via15]{Vial15}
C.~Vial,
\newblock {Exceptional collections and the N\'eron-Severi lattice for surfaces},
\newblock preprint, 2015.
\newblock {\em arXiv:1504.01776}.

\bibitem[Wah81]{Wahl81}
J.~Wahl.
\newblock {Smoothings of normal surface singularities}.
\newblock {\em Topology}, 20(3):219--246, 1981.

\end{thebibliography}
\end{document}